\documentclass[a4paper,10pt,reqno]{amsart}
\usepackage[english]{babel}
\usepackage{amssymb,amsthm,amsmath,eufrak} 
\usepackage{mathtools}
\usepackage{mathrsfs}
\usepackage{enumerate}
\usepackage{cite}
\usepackage[symbol]{footmisc}
\usepackage{color}
\usepackage{geometry}
\usepackage[colorlinks=true, linkcolor=red, citecolor=blue, urlcolor=blue,hyperfootnotes=false]{hyperref}
\usepackage{cleveref}

\usepackage{soul} %%PARA TACHAR
\theoremstyle{plain} 
\newtheorem{theorem}{Theorem}[section]
\newtheorem{lemma}{Lemma}[section]

\newtheorem{corollary}{Corollary}[section]

\theoremstyle{definition}
\newtheorem{definition}{Definition}[section]

\theoremstyle{remark}
\newtheorem{remark}{Remark}[section]

\newcommand{\R}{\mathbb{R}}

\newcommand{\C}{\mathbb{C}}

\newcommand{\normeq}[1]{{\left\vert\kern-0.25ex\left\vert\kern-0.25ex\left\vert #1 
    \right\vert\kern-0.25ex\right\vert\kern-0.25ex\right\vert}}

{\left\lbrace\begin{array}{r@{\hspace{1mm}}ll}}%
{\end{array}\right.}

\title[NLS-DNN]{Error bounds for Physics Informed Neural Networks in Nonlinear Schr\"odinger equations placed on unbounded domains}

\author{Miguel \'A. Alejo} 
\address{Departamento de Matem\'aticas. Universidad de C\'ordoba\\
C\'ordoba, Spain.}
\email{malejo@uco.es}
\thanks{M.\'A.A. was partially supported by Grant PID2022-137228OB-I00 funded by the Spanish Ministerio de Ciencia, Innovaci\'on y Universidades, MICIU/AEI/10.13039/501100011033.}

\author{Lucrezia Cossetti} 
\address{Ikerbasque and UPV/EHU, Aptdo. 644, 48080, Bilbao, Spain}
\email{lucrezia.cossetti@ehu.eus}
\thanks{L. C. was supported by the grant Ram\'on y Cajal RYC2021-032803-I funded by MCIN/AEI/10.13039/50110 0011033 and by the European Union NextGenerationEU/PRTR, the Deutsche Forschungsgemeinschaft (DFG, German Research Foundation) -- Project-ID 258734477 -- SFB 1173 and by Ikerbasque.}

\author{Luca Fanelli} 
\address{Ikerbasque \& Universidad del Pa\'is Vasco/Euskal Herriko Unibertsitatea, UPV/EHU \& BCAM, Aptdo. 644, 48080, Bilbao, Spain}
\email{luca.fanelli@ehu.es}
\thanks{L. F. was supported by the projects PID2021-123034NB-I00/MCIN/AEI/10.13039/501100011033 funded by the Agencia Estatal de Investigaci\'on, IT1615-22 funded by the Basque Government, and by Ikerbasque. He is also partially supported by the Basque Government through the BERC 2022–2025 program and by the Spanish Agencia Estatal de Investigaci\'on through BCAM Severo Ochoa excellence accreditation CEX2021-001142-S/MCIN/AEI/10.13039/501100011033.}

\author{Claudio Mu\~noz}  % in alphabetical order
	\address{Departamento de Ingenier\'{\i}a Matem\'atica and Centro
de Modelamiento Matem\'atico (UMI 2807 CNRS), Universidad de Chile.}
	\email{cmunoz@dim.uchile.cl}
	\thanks{C.M. was partially funded by Chilean research grants ANID 2022 Exploration 13220060, FONDECYT 1231250, and Basal CMM FB210005.}

\author{Nicol\'as Valenzuela}
\address{Departamento de Ingenier\'{\i}a Matem\'atica, Universidad de Chile.}
	\email{nvalenzuela@dim.uchile.cl}
	\thanks{N.V. was partially funded by Chilean research grants ANID 2022 Exploration 13220060, FONDECYT 1231250, a Latin America PhD Google Fellowship, ANID-Subdirecci\'on de Capital Humano/Doctorado Nacional/2023-21231021 and Basal CMM FB210005.}

\numberwithin{equation}{section}

\usepackage{subcaption}
\begin{document}

\begin{abstract}
We consider the subcritical nonlinear Schr\"odinger (NLS) in dimension one posed on the unbounded real line. Several previous works have considered the deep neural network approximation of NLS solutions from the numerical and theoretical point of view in the case of bounded domains. In this paper, we introduce a new PINNs method to treat the case of unbounded domains and show rigorous bounds on the associated approximation error in terms of the energy and Strichartz norms, provided a reasonable integration scheme is available. Applications to traveling waves, breathers and solitons, as well as numerical experiments confirming the validity of the approximation are also presented as well.
\end{abstract}



\maketitle

%\date{\small 18 November 2021}
%%%%%%%%%%%%%%%%%%%%%%%%%%%%%%%%%%%%%%%%%%%%%%%%%%%%%%%%%%%%%%%%%%%%%%%%%%%%%%%%%%%%%%%%%%%%%%%%%%%%%%%%%%%%%%%%%%%%%%%%%%%%%%%%%%%%%%%%%%

%\maketitle
%\vspace{-1cm}

%\nocite{*}

%------------abstract------------%

\section{Introduction}

\subsection{Setting} Among the deep learning methodologies developed in recent years, Physics-Informed Neural Networks (PINNs) have emerged as a highly relevant approach for approximating solutions to physical models using neural networks. Originally introduced in \cite{LLF,LC,RK,RPK}, the PINN framework leverages the fact that deep neural networks (DNNs) possess the {\it universal approximation property}, i.e., the ability to approximate any continuous and even measurable function \cite{Hornik,Leshno,LLP,LLF,Dmitry}. Furthermore, PINNs are able to mirror the underlying physical background of the problem, improving the modeling accuracy of DNNs.

Significant examples and advancements in the development of PINNs are provided in \cite{RK,RPK,PLK,MJK,Raissi,Raissi2, Mishra,Kar}. PINNs have also been extended to solve elliptic partial differential equations (PDEs) \cite{BDPS,Zer} and inverse problems \cite{MM20,PC}. This computational approach is now widely applied across numerous research areas, including, but not limited to, General Relativity \cite{FMRTF,LNPC,Ha}, discrete systems to uncover their underlying dynamical equations \cite{SZCK,ZKCK}, Quantum Mechanics \cite{WY,PLC}, wave propagation \cite{RHSK}, and fiber optics \cite{ZYXLCYC,ZFWCQ}.

In essence, PINNs enable the use of DNNs as ansatz spaces for solving physically-driven PDEs, combining the power of neural networks with the constraints and insights provided by physical models.

In the broader context, various results demonstrate that deep neural networks are effective in approximating solutions for certain classes of PDEs. These results span both numerical and theoretical perspectives and cover a wide range of applications. Notably, the PDEs studied using these methods include linear and semilinear parabolic equations \cite{Han, Han2, Hure, Hutz2, Beck1, Beck2, Jentz1}, stochastic time-dependent models \cite{Grohs2, Beck2, Berner1}, linear and semilinear time-independent elliptic equations \cite{LLP, Grohs}, a variety of fluid dynamics equations \cite{Lye, MJK, Mishra2}, and both time-dependent and time-independent non-local equations \cite{PLK, Raissi0, Gonnon, Val22, Val23, Castro}. We refer to these works for further developments in this rapidly expanding research area.

Moreover, the learning methods used to achieve these results can be classified into several categories. In addition to the previously mentioned PINN technique, other prominent deep learning methods include Monte Carlo approaches \cite{Val22,Val23,MuVa24}, Multilevel Picard Iterations (MLP) \cite{Hutz2, Beck}, and neural networks that approximate infinite-dimensional operators, such as DeepONets \cite{Chen, Lu, Castro2, CalderonDO}. Further methodologies are discussed in the recent review \cite{Beck23} and the monograph \cite{JKW}.

In essence, PINNs take advantage of two key sources of data commonly available in physical systems: initial and boundary conditions, although conserved quantities may also be incorporated. In the context of classical PDEs such as Schrödinger \cite{BKM22}, Burgers, Navier-Stokes \cite{Raissi2, MJK, Mishra, RK, Lye, KMPT}, Kolmogorov \cite{Jentz1}, and nonlinear diffusion in multiphase materials \cite{KK}, the PINN method provides accurate numerical approximations of classical solutions.
\vskip0.3cm
For nonlinear dispersive PDEs, one of the key challenges in developing approximation methods arises from the fact that these equations are often posed on unbounded domains. A canonical example is the nonlinear Schr\"odinger equation (NLS): PINNs applied to NLS models on bounded domains were first computed in the foundational work \cite{RPK}. One of the earliest rigorous results in this area—using PINNs to approximate dispersive models in spatially bounded domains—was achieved in \cite{BKM22}, building on earlier work in \cite{MM22}.

For wave models, PINNs were used to solve the wave equation in \cite{MMN}, and Monte Carlo methods were extended for wave models in \cite{MuVa24}. Moreover, wave models were shown to be suitable for PINN approximation in \cite{LBK24}, including applications for approximating discontinuous solutions of nonlinear hyperbolic PDEs \cite{RMM}. An interesting recent paper is \cite{JORS}, where an integration scheme is proposed for the integration of the cubic NLS on the 2D-torus. However, when restricting to bounded domains, many physically significant phenomena, such as the emergence of soliton-like solutions, cannot be captured.

One of the main difficulties in this field is ensuring the mathematically rigorous global validity of the PINN method for unbounded domains. To address this limitation, several improved PINN formulations and alternative neural network approaches have been proposed. For example, in problems posed on infinite domains, one natural computational approach involves splitting the infinite domain into finite segments using absorbing or non-reflecting boundary conditions \cite{RRSL,RRCWSL}. Another method employs a spectral expansion formulation \cite{XBC} to handle the infinite domain, where PINNs are only used to compute the unknown coefficients in the spectral expansion. Other approaches to manage unbounded domains have been explored for specific applications, such as seismic waves and thermal dynamics \cite{RRCWSL, BE}.
\vskip0.3cm
The objective of this paper is to propose a modification to the PINN technique that enables its application to unbounded problems, specifically targeting classical soliton solutions. For simplicity, we will focus on a canonical one-dimensional dispersive model that supports the existence of solitons and breathers: the 1D-focusing NLS equation posed on the real line,
\begin{equation}\label{eq:NLS-anyalpha}
	\begin{cases}
	i\partial_t u + \partial_x^2 u + |u|^{\alpha-1}u =0,
	\\
	u(0,\cdot) :=u_0(\cdot)\in H^1(\mathbb R;\mathbb C),
	\end{cases}
\end{equation}
%\begin{equation}\label{NLS}
%i\partial_t u + \partial_x^2 u + |u|^{\alpha-1} u =0, \quad (t,x)\in\R, \quad u(t,x)\in\mathbb C.
%\end{equation}
where $u=u(t,x):\R\times\R\to\C$, and $2\leq \alpha<5$, which is the natural $L^2$ subcritical case.\footnote{Some technical issues are present if $\alpha<2$, but we believe that our results also apply to that case.} Global in time solutions in $H^1$ are well-known to exist \cite{GV,GV2,Tsutsumi,Bourgain,Cazenave,Cazenave2,LP},  due to the fact that the conserved energy
\[
E(t):=\int_{\R}|\partial_xu(t)|^2\,dx - \frac{1}{\alpha+1}\int_\R|u(t)|^{\alpha+1}\,dx=E(0)
\]
is a bound from above of the $H^1$-norm (in the subcritical case).
Since no boundaries exist, we will replace the boundary data by \emph{suitable global linear data}, in the following sense. 
Let $I$ be a bounded time interval containing zero. Let
\begin{equation}
\label{eq:pq}
(p,q)=\left(\tfrac{4(\alpha+1)}{\alpha-1},\alpha+1\right)
\end{equation}
be Schr\"odinger admissible pairs (see Definition \ref{def:adm-pairs}). Let 
\begin{enumerate}
\item $f=f(x)$ be bounded differentiable, and $g_n=g_n(t,x)$, $n=2,3,4$, bounded and continuous, both complex-valued; 
\item Integers $N_1,N_2,N_3,N_4\geq 1$, $M_2,M_3,M_4\geq 1$;
\item For $n=1,2,3,4$, collocation points $(x_{n,j})_{j=1}^{N_n}\subseteq\mathbb R$, and for $n=2,3,4$, times $(t_{n,\ell})_{\ell=1}^{M_n} \subseteq I$; 
\item For $n=1,2,3,4$, vector weights $(w_{n,j})_{j=1}^{N_n}\subseteq {\color{black}\R_+}$, and for $n=2,3,4$, matrix weights $(w_{n,\ell,j})_{\ell,j=1}^{M_n,N_n} \subseteq {\color{black}\R_+}$.
\end{enumerate}
Consider the approximative norms 
\begin{equation}\label{Loss1}
\begin{aligned}
\mathcal J_{H^1,N_1}[f]:= &~{} \left(  \frac1{N_1}\sum_{j=1}^{N_1} w_{1,j} \left( |f(x_j)|^2 +|f'(x_j)|^2 \right) \right)^{1/2},\\
\mathcal J_{\infty, H^1,N_2,M_2}[g_2]:= &~{}\max_{\ell=1,\ldots,M_2} \left(  \frac1{N_2}\sum_{j=1}^{N_2} w_{2,\ell,j} \left( |g_2(t_{2,\ell},x_{2,j})|^2 +|\partial_x g_2(t_{2,\ell},x_{2,j})|^2 \right) \right)^{1/2},\\
 \mathcal J_{p,q,N_3,M_3}[g_3]:=&~{} \left( \frac1{M_3}\sum_{\ell=1}^{M_3} \left( \frac1{N_3} \sum_{j=1}^{N_3}   w_{3,j,\ell} \left(|g_3(t_{3,\ell},x_{3,j})|^q{\color{black}+|\partial_x g_3(t_{3,\ell},x_{3,j})|^q} \right) \right)^{p/q} \right)^{1/p},\\
  \mathcal J_{p',q',N_4,M_4}[g_4]:=&~{} \left( \frac1{M_4}\sum_{\ell=1}^{M_4} \left( \frac1{N_4} \sum_{j=1}^{N_4}   w_{4,j,\ell} \left(|g_4(t_{4,\ell},x_{4,j})|^{q'}{\color{black}+|\partial_x g_4(t_{4,\ell},x_{4,j})|^{q'}} \right) \right)^{p'/q'} \right)^{1/p'}.
 \end{aligned}
\end{equation}
As it is common in numerical computations, the collocation points and times will be chosen to be uniformly spaced, mimicking the standard {\color{black}quadrature rules}. Our main result demonstrates that, under smallness assumptions on the quantities $\mathcal{J}_{H^1,N_1}$ and $\mathcal{J}_{p',q',N_4,M_4}$ for certain functions, PINNs can satisfactorily approximate the NLS dynamics in unbounded domains.

\noindent
{\bf Hypotheses on integration schemes}. Following \cite{MM22}, our first requirement is a suitable way to approach integration.
\begin{enumerate}
\item[(H1)] Efficient integration. There exist efficient approximative rules for the computations of the $H^1_x(\mathbb R)$, $L^\infty_t H^1_x(I\times \mathbb R)$, $L^p_t W^{1,q}_x(I\times \mathbb R)$ and $L^{p'}_t W^{1,q'}_x(I\times \mathbb R)$ norms, in terms of the quantities defined in \eqref{Loss1}, in the following sense: 
\begin{itemize}
\item For any $\delta>0$, and all $f\in H^1(\mathbb R)$, there exist $N_1 \in\mathbb N$, $R_1>0$, points $(x_{1,j})_{j=1}^{N_1} \subseteq [-R_1,R_1]$ and weights $(w_{1,j})_{j=1}^{N_1} \subseteq {\color{black}\R_+}$ such that
\[
 \left| \|f\|_{H^1(\mathbb R)} -\mathcal J_{H^1,N_1}(f) \right| <\delta.
\]
\item  For any $\delta>0$, and all $g_1\in L^\infty_t H^1_x(I\times \R)$, there are $N_2,M_2 \in\mathbb N$, $R_2>0$, points $(t_{2,\ell}, x_{2,j})_{j=1}^{M_2,N_2} \subseteq I\times [-R_2,R_2]$ and weights $(w_{2,\ell,j})_{\ell,j=1}^{M_2,N_2} \subseteq {\color{black}\R_+}$ such that
\[
 \left| \|g_2\|_{L^\infty_t H^1_x(I\times \R)} -\mathcal J_{\infty,H^1,N_2,M_2}(g_2) \right| <\delta.
\]
\item  For any $\delta>0$, and all $g_3\in L^{p}_t W^{1,q}_x(I\times \R)$, there are $N_3,M_3 \in\mathbb N$, $R_3>0$, points $(t_{3,\ell}, x_{3,j})_{j=1}^{M_3,N_3} \subseteq I\times [-R_3,R_3]$ and weights $(w_{3,\ell,j})_{\ell,j=1}^{M_3,N_3} \subseteq {\color{black}\R_+}$ such that
\[
 \left| \|g_3\|_{L^p_t W^{1,q}_x(I\times \R)} -\mathcal J_{p,q,N_3,M_3}(g_3) \right| <\delta.
\]
\item  For any $\delta>0$, and all $g_4 \in L^{p'}_t W^{1,q'}_x(I\times \R)$, there are $N_4,M_4 \in\mathbb N$, $R_4>0$, points $(t_{4,\ell}, x_{4,j})_{j=1}^{M_4,N_4} \subseteq I\times [-R_4,R_4]$ and weights $(w_{4,\ell,j})_{\ell,j=1}^{M_4,N_4} \subseteq {\color{black}\R_+}$ such that
\[
\left| \|g_4\|_{L^{p'}_t W^{1,q'}_x(I\times \R)} -\mathcal J_{p',q',N_4,M_4}(g_4) \right| <\delta.
\]
\end{itemize}
\end{enumerate}
Note that even though we are considering functions defined in the entire real line $x\in \mathbb R$, we take advantage of the fact that given any $f$ of finite norm $\|f\|_{H^1(\mathbb R)}<+\infty$, and tolerance $\delta$, it is possible to find $R(\delta)>0$ such that $\|f\|_{H^1(|x| \geq R)}<\frac12\delta$. This allows us to restrict \emph{the numerical integration} to a large finite interval $[-R,R]$ (notice that this does not fully address the issue of defining PINNs with boundary conditions at infinity). The same argument applies in the case of time-space norms. The previous requirements are also standard in the literature \cite{MM22}. %Notice that the collocation points and weights are always required to be independent of the functions involved in the integration.

\medskip

\noindent
{\bf Hypotheses on PINNs}.
 Let $u_{\text{DNN},\#}= u_{\text{DNN},\#} (t,x)$ be a smooth bounded complex-valued function constructed by means of an algorithmic procedure (either SGD or any other ML optimization procedure) and realization of a suitable PINNs, in the following sense: under the framework \eqref{Loss1}, one has:
% given $\delta>0$, one applies (H1) on $f=u_0-u_{\text{DNN},\#}(0)$, $g_2=u_{\text{DNN},\#} $, and $g_3=$:
% \[
% \begin{aligned}
%  & \|  u_0-u_{\text{DNN},\#}(0) \|_{H^1(\R)} \leq \mathcal J_{H^1,N_1} (u_0-u_{\#}(0))  +\delta,\\
% & \| u_{\text{DNN},\#} \|_{L^\infty_tH^1(I\times \R)} \leq \mathcal J_{\infty,H^1,N_2,M_2} [u_{\#}] +\delta,\\
% & \|u_{\#}\|_{L^p_t W^{1,q}_x(I\times \R)} \leq \mathcal J_{p,q,N_3,M_3} [u_{\#}] +\delta. 
% \end{aligned}
% \]
% Under this setting, we require:
\begin{enumerate}
\item[(H2)] Uniformly bounded large time $L^\infty_tH^1_x$, initial time $H^1_x$ and large time $L^p_t W^{1,q}_x$ control for $u_0$ and $u_{\text{DNN},\#}$.  There are $A,\widetilde A,B>0$ such that the following holds. For any $N_{1,0},N_{2,0},N_{3,0}\geq 1$ and $M_{2,0},M_{3,0}\geq 1$, there are $N_j\geq N_{j,0}$ and $M_j\geq M_{j,0}$ such that each approximate norm $\mathcal J_{H^1,N_1}[(u_0-u_{\text{DNN},\#}(0))]$,  $ \mathcal J_{\infty,H^1,N_2,M_2} [u_{\text{DNN},\#}]$ and $\mathcal J_{p,q,N_3,M_3} [u_{\text{DNN},\#}] $ satisfies
\[
\begin{aligned}
&  \mathcal J_{H^1,N_1}[ (u_0-u_{\text{DNN},\#}(0))]\leq \widetilde{A}, \\
& \mathcal J_{\infty,H^1,N_2,M_2} [u_{\text{DNN},\#}] \leq A, \quad \mathcal J_{p,q,N_3,M_3} [u_{\text{DNN},\#}]  \leq B.
\end{aligned}
\]
\item[(H3)] Small linear $L^p_t W^{1,q}_x$ and nonlinear $L^{p'}_t W^{1,q'}_x$ control. Given $\varepsilon>0$, and given  $N_{4,0},N_{5,0}\geq 1$ and $M_{4,0},M_{5,0}\geq 1$, there are $N_j\geq N_{j,0}$ and $M_j\geq M_{j,0}$ and corresponding approximative norms  $ \mathcal J_{p',q',N_4,M_4} [\mathcal E[u_{\text{DNN},\#}]]$ and $\mathcal J_{p,q,N_5,M_5} [ e^{it\partial_x^2} (u_0-u_{\text{DNN},\#}(0))] $ such that
\[
\begin{aligned}
  \mathcal J_{p',q',N_4,M_4} [\mathcal E[u_{\text{DNN},\#}]] + \mathcal J_{p,q,N_5,M_5} [ e^{it\partial_x^2} (u_0-u_{\text{DNN},\#}(0))] < \varepsilon,
\end{aligned}
\]
with $\mathcal E[u_\#]:= i\partial_t u_\# + \partial_x^2 u_\# + |u_\#|^{\alpha-1}u_\#$. Here $e^{i t \partial_x^2}f(x)$ represents the standard linear Schr\"odinger solution issued of initial data $f$ at time $t$ and position $x$. 
\end{enumerate}
%
%\begin{lemma}[Long-time stability]\label{lemma:stability-long}
%Let I be a fixed compact interval containing zero. Let $u_{\#}$ be an approximate solution to~\eqref{eq:NLS-anyalpha} on $I\times \R$ in the sense of Definition~\ref{def:appr}. Assume that 
%	\begin{align}
%		\label{eq:bound-B}
%		\|u_{\#}\|_{L^p_t W^{1,q}_x(I\times \R)}&\leq B,\\ 
%		\label{eq:flux-alt}
%		\|e^{it\partial_x^2} (u_0-u_{\#}(0))\|_{L^p_t W^{1,q}_x(I\times \R)}&\leq \varepsilon,\\
%		\label{eq:E-alt}
%		\|\mathcal{E}[u_{\#}]\|_{L^{p'}_t W^{1,q'}_x(I\times \R)}&\leq \varepsilon,
%	\end{align}
%	for some constant $B>0,$ $0<\varepsilon\leq \varepsilon_1,$ with $\varepsilon_1=\varepsilon_1(A, \widetilde{A},B)>0$ a small constant.
%	
%	
%\end{lemma}
%
Notice that $u_{\text{DNN},\#}$ may not be convergent to zero as time tends to infinity; then it is only required that, given a suitable tolerance $\varepsilon$, hypothesis (H3) is satisfied with computed data in the region $[-R,R]$, and with bounded local control given by (H2). Also, thanks to Strichartz estimates, the linearly motivated bound $ \mathcal J_{p,q,N_5,M_5} [ e^{it\partial_x^2} (u_0-u_{\text{DNN},\#}(0))] $ may be obtained from $\mathcal J_{H^1,N_1} [(u_0-u_{\text{DNN},\#}(0))]\leq \widetilde{A}$ by choosing $\widetilde A$ sufficiently small.

Finally, recall $\mathcal{C}'$ as the Strichartz space introduced in Definition \ref{SS}.

\begin{theorem}\label{MT}
Let $u_0\in H^1(\mathbb R)$, and assume \emph{(H1)}. Let $A,\widetilde A, B>0$ be fixed numbers, and let $0<\varepsilon<\varepsilon_1$ sufficiently small. Finally, let $u_{\emph{DNN},\#}$ be a DNN satisfying \emph{(H2)}-\emph{(H3)}. Then there exists a solution $u\in \mathcal{C}'(I\times \R)$ to~\eqref{eq:NLS-anyalpha} on $I\times \R$ with initial datum $u_0\in H^1(\R)$ such that for all $R>0$ sufficiently large and constants $C$,
\vspace{-0.1cm}
	\begin{align}
		\label{eq:small-alt_0}
		\|u-u_{\emph{DNN},\#}\|_{L^p_t W^{1,q}_x(I\times [-R,R])}&\leq C(A,\widetilde{A},B) \varepsilon,\\
		%\label{eq:S}
		%\|(i\partial_t + \partial_x^2)(u-u_{\#}) + \mathcal{E}[u_{\#}]\|_{L^{p'}_t W^{1,q'}_x(I\times \R)}&\leq C(A,\widetilde{A},B) \varepsilon\\
		\label{eq:large-alt_0}
		\|u-u_{\emph{DNN},\#}\|_{\mathcal{C}'(I\times [-R,R])}&\leq C(A,\widetilde{A},B),\\
		\label{eq:H1-alt_0}
		\|u-u_{\emph{DNN},\#}\|_{L^\infty_t H^1_x(I\times [-R,R])}&\leq C(A,\widetilde{A},B).
		\end{align}
		Here $(p,q)$ represents the Strichartz admissible pair defined in~\eqref{eq:pq}.
\end{theorem}

The previous result indicates that PINNs satisfying the bounds (H2) and (H3) will remain close to the solution of the nonlinear Schr\"odinger equation (NLS) in global norms. In particular, the $L^p_t W^{1,q}_x(I\times [-R,R])$ norm remains small. It is important to note that smallness in the $L^\infty_t H^1_x(I\times [-R,R])$ norm cannot be achieved unless the parameter $\widetilde A$ is chosen to be small. Additionally, since $u_{\text{DNN},\#}$ may not converge to zero as $|x|\to\infty$, the conclusions drawn within the interval $[-R,R]$, with $R$ arbitrary but finite, are optimal in a certain sense.

We also stress that the smallness condition (H3) on $\mathcal J_{p',q',N_4,M_4} [\mathcal E[u_{\text{DNN},\#}]]$ is standard when one studies PINNs. However, the smallness condition on $\mathcal J_{p,q,N_{5},M_{5}} [ e^{it\partial_x^2} (u_0-u_{\text{DNN},\#}(0))]$ is one of the key new ingredients in this paper. This condition indicates the necessity of effectively controlling the \emph{linear evolution} of nearby initial data when evaluated using the PINNs method. This computation is straightforward, as the linear evolution can be efficiently solved with existing numerical methods and the fact that $e^{it\partial_x^2} f$ is given by
$$
\mathcal F_{\xi \to x}^{-1} \left( e^{-i t\xi^2} \mathcal F_{x\to \xi}(f)(\xi)\right) =\frac1{(4i\pi t)^{1/2}} \int_y e^{-|x-y|^2/(4 i t)} f(y)dy,
$$
with $\mathcal F$ the standard Fourier transform. In particular, $\mathcal J_{p,q,N_3,M_3} [ e^{it\partial_x^2} (u_0-u_{\text{DNN},\#}(0))]$ can be easily estimated and computed. Moreover, if $\widetilde A=O(\varepsilon)$, then the bound (H3) on $\mathcal J_{p',q',N_4,M_4} [\mathcal E[u_{\text{DNN},\#}]]$ is naturally satisfied. See Remark \ref{R3p1} for additional details.

An important advantage of Theorem \ref{MT} is that the hypotheses are independent of the specific PINNs considered. This means that both classical and original definitions \cite{LLF,LC,RK,RPK} are as suitable as modern contributions to the field \cite{PLC,PLK}, among others.

Now we will explain how Theorem \ref{MT} is proved. A key element in the proof is a long-time stability property specifically developed for the interaction between deep neural networks (DNNs) and the NLS case. This stability property has been extensively studied in previous years \cite{TV,Tao,Iteam} (and references therein) as a tool to establish global well-posedness for critical equations. In our approach, we have drawn inspiration from this method to extend the PINNs restrictions in (H2)-(H3) to the full solution over the interval $I\times[-R,R]$, $R$ being arbitrary. Furthermore, Theorem \ref{MT} remains valid for any suitable DNN optimization method, other than PINNs, that satisfies conditions (H1)-(H3).

Finally, note that \eqref{eq:H1-alt} may not be satisfactory, since one would like to also achieve a small $L^\infty_t H^1_x$ norm on the error. Ensuring this is naturally challenging due to the presence of the soliton and solitary wave solutions, which evolve over time by transferring energy from compact intervals to infinity. Thus, effective control will require generalizing the standard orbital stability theory to accommodate deep learning techniques. However, we are able to present an interesting preliminary result in this direction, as seen in Corollary \ref{Cor1}. In particular, the finite speed of propagation suggests that the approximation method should work even better for wave than for NLS.

%In \cite{Zhang}, PINNs were used to find solutions to the Schr\"odinger equation in Quantum Mechanics.

%most dominant methodologies
\subsection*{Organization of this work} In Section \ref{sec:WP} we provide some standard well-posedness results for NLS in $H^1$. In Section \ref{section:stability} we prove a short time stability result. We emphasize that the results presented in Section \ref{sec:WP} and Section \ref{section:stability} are fairly standard and may be skipped by expert readers already familiar with the underline theory.  In Section \ref{sec:LTS} we provide the main tool in this paper, a long time stability result. In Section \ref{sec:PMR} we prove the main results in this work. In Section \ref{sec:numerical} we provide numerical computations supporting our findings. Finally, in Section \ref{sec:discussion} we provide a discussion and conclusions.

\subsection*{Acknowledgments} C. M. would like to thank the Erwin Schr\"odinger Institute ESI (Vienna) where part of this work was written. N. V. thanks BCAM members for their hospitality and support during research visits in 2023 and 2024. M. \'A. A. would like to thank the DIM-University of Chile, for its support and hospitality during research stays  while this work was written.
\section{Well-posedness 1D-NLS with focusing nonlinearity: $H^1$-theory}\label{sec:WP}
In this preliminary section we examine the Cauchy problem associated to the 1D-focusing NLS  \eqref{eq:NLS-anyalpha}. As is customary, we first show \emph{local} well-posedness of~\eqref{eq:NLS-anyalpha}, then \emph{global} well-posedness follows as an easy consequence of Gagliardo-Nirenberg inequalities.  
In our first result we show that the problem~\eqref{eq:NLS-anyalpha} is locally well-posed in $H^1(\R)$ for \emph{any} power of the nonlinearity. To present this result, we provide the following preliminary definitions and recall some useful concepts.
\begin{definition}[Admissible pairs]\label{def:adm-pairs}
	We say that a pair of exponents $(p,q)$ is Schr\"odinger admissible if 
	\begin{equation*}
		\frac{2}{p} + \frac{1}{q}=\frac{1}{2},
		\qquad 
%		\begin{cases}
%			2\leq q<\frac{2n}{n-2},  &\text{if }n\geq 3,\\
%			2\leq q<\infty, &\text{if }n=2,\\ 
			2 \leq q\leq \infty.% , &\text{if }n=1.
%		\end{cases}
	\end{equation*}
\end{definition}
The following estimates are standard in the literature, see e.g. \cite{Tao,LP}.
\begin{lemma}[Strichartz estimates]
Let $(p,q)$ and $(\widetilde{p},\widetilde{q})$ Strichartz admissible pairs (\emph{cf.} Definition~\ref{def:adm-pairs}), then the following estimates hold
\begin{align}
	\label{eq:S1}\tag{$S1$}
	&\|e^{it\Delta} f\|_{L^p_tL^q_x}\lesssim\|f\|_{L^2_x},\\
	\label{eq:S2}\tag{$S2$}
	&\Big\|\int_{\R}e^{i(t-s)\Delta} g(\cdot, s)\, ds\Big\|_{L^p_tL^q_x}\lesssim\|g\|_{L^{p'}_tL^{q'}_x},
	\qquad \left(\Big\|\int_0^te^{i(t-s)\Delta} g(\cdot, s)\, ds\Big\|_{L^p_tL^q_x}\lesssim\|g\|_{L^{p'}_tL^{q'}_x} \right),\\
	\label{eq:S3}\tag{$S3$}
	&\Big\|\int_{\R}e^{it\Delta} g(\cdot, t)\, dt\Big\|_{L^2_x}\lesssim\|g\|_{L^{p'}_tL^{q'}_x}.
\end{align}
Finally %Combining~\eqref{eq:S2} and~\eqref{eq:S3} one gets
\begin{equation}\label{eq:S4} \tag{$S4$}
	\Big\|\int_0^t e^{i(t-s)\Delta} g(\cdot, s)\, ds\Big\|_{L^p_tL^q_x}\lesssim \|g\|_{L^{\widetilde{p}'}_t L^{\widetilde{q}'}_x}.
\end{equation}
\end{lemma}
\begin{definition}[Strichartz space]\label{SS}
Let $I:=[-T,T].$ We define the $\mathcal{C}'(I\times\R)$ Strichartz norm by
\begin{equation*}
	\|u\|_{\mathcal{C}'(I\times\R)}
	=\|u\|_{\mathcal{C}'}
	:=\sup\|u\|_{L^p_t W^{1,q}_x(I\times \R)},
\end{equation*}
where the supremum is taken over all 1D admissible pairs $(p,q).$
The associated Strichartz space $\mathcal{C}'(I\times\R)$ is defined as the closure of the test functions under this norm.
\end{definition}  
\begin{theorem}[Local well-posedness]\label{thm:LWP}
	If $1<\alpha<\infty,$ then for all $u_0\in H^1(\R)$ there exists $T=T(\|u_0\|_{H^1(\R)},\alpha)>0$ and a unique solution $u$ of~\eqref{eq:NLS-anyalpha} in $\mathcal{C}'(I\times \R).$
\end{theorem}
\begin{proof}
	Let $X_T$ denote the space where we will perform the fixed point argument and that will be determined below. The subscript in the notation $X_T$ emphasizes that the space will always run on a fixed period of time $I=[-T,T]$, with $T>0$. For all positive constants $T$ and $a$ we define $B_T(0,a)=\{u\in X_T\colon \|u\|_{X_T} \leq a\}.$ For appropriate values of $a$ and $T>0$ we shall show that 
	\begin{equation*}
		\Phi(u)(t):=e^{it \partial_x^2} u_0 +i\int_0^t e^{i(t-s)\partial_x^2}|u|^{\alpha-1}u\, ds,
	\qquad t\in I
	\end{equation*}	 
	defines a contraction map on $B_T(0,a).$ In particular, if $\Phi$ is proven to have the contracting property in $X_T=\mathcal{C}'(I\times \R),$ then the result follows. To show that we proceed in three steps. We start considering as $X_T$ the energy space $X_T=L^\infty_t H^1_x(I\times \R).$
	\begin{description}
		\item [\underline{Case $X_T=L^\infty_t H^1_x(I\times \R)$}] Due to the embedding $H^1(\R)\hookrightarrow L^\infty(\R)$ valid in the one dimensional framework, closing the fixed point argument in $X_T$ is almost immediate. Indeed
		\begin{equation*}
			\begin{split}
			\|\Phi(u)\|_{L^\infty_t H^1_x}
			&\leq \|e^{it\partial_x^2}u_0\|_{L^\infty_t H^1_x}
			+ \Bigg\|\int_0^t e^{i(t-s)\partial_x^2}|u|^{\alpha-1}u\, ds\Bigg\|_{L^\infty_t H^1_x}\\
			&=\|u_0\|_{H^1_x} 
			+ \Bigg\| \Big\| \int_0^t e^{i(t-s)\partial_x^2} |u|^{\alpha-1}u\, ds \Big\|_{H^1_x}\Bigg\|_{L^\infty_t}
			\end{split},
		\end{equation*}
		where here we just used that $e^{it\partial_x^2}$ is unitary in $H^1(\R).$ 
		Using Minkowsky, the unitarity of $e^{it\partial_x^2}$ and again the Sobolev embedding $H^1(\R)\hookrightarrow L^\infty(\R)$ one has
		\begin{equation*}
			\begin{split}
				\Big\| \int_0^t e^{i(t-s)\partial_x^2} |u|^{\alpha-1}u\, ds \Big\|_{H^1_x}
				&\leq \int_0^t \| |u|^{\alpha-1} u\|_{H^1_x}\, ds\\
				&\leq \int_0^t \|u\|_{H^1_x} \|u\|_{L^\infty}^{\alpha-1}\, ds\lesssim \int_0^t \|u\|_{H^1_x}^\alpha\, ds.
			\end{split}
		\end{equation*} 
		Using the latter in the former one gets
		\begin{equation*}
			\|\Phi(u)\|_{L^\infty_t H^1_x}
			\lesssim 
			\|u_0\|_{H^1_x} + T \|u\|_{L^\infty_t H^1_x}^\alpha.
		\end{equation*}
		If we assume $u\in B_T(0, a)$	in $X_T=L^\infty_t H^1_x,$ namely $\|u\|_{L^\infty_t H^1_x}\leq a$ and choosing $a=2c\|u_0\|_{H^1_x},$ where $c$ from now on will denote the constant hidden in the symbol $\lesssim,$ then
		\begin{equation*}
			\|\Phi(u)\|_{L^\infty_tH^1_x}
			\leq \frac{a}{2} + Ta^\alpha. 
		\end{equation*} 
		If $T$ is sufficiently small, namely $Ta^{\alpha-1}\leq 1/2,$ then $\Phi\colon B_T(0,a)\to B_T(0,a).$ Proving that $\Phi$ is a contraction can be done similarly.
		
\medskip

		\item[\underline{Case $X_T=L^\infty_t H^1_x \cap L^\frac{4(\alpha+1)}{\alpha-1}_tW^{1, \alpha +1}_x$}]\footnote{The space $L^\frac{4(\alpha+1)}{\alpha-1}_tL^{\alpha +1}_x$ is the generalisation of $L^8_tL^4_x$ for the cubic nonlinearity $\alpha=3.$} To close the fixed point argument in this case we  only need  to estimate the $L^\frac{4(\alpha+1)}{\alpha-1}_tW^{1, \alpha +1}_x$-norm of the Duhamel's term. We shall see the estimate for $\partial_x u,$  the estimate for $u$ is performed similarly. Using~\eqref{eq:S2} and H\"older in space one has
		\begin{equation*}
			\begin{split}
			\Big\|\int_0^t e^{it\partial_x^2} |u|^{\alpha-1}\partial_x u\, ds\Big\|_{L^\frac{4(\alpha+1)}{\alpha-1}_tL^{\alpha +1}_x}
			&\lesssim
			\||u|^{\alpha-1}\partial_x u\|_{L^\frac{4(\alpha+1)}{3\alpha+5}_tL^\frac{\alpha +1}{\alpha}_x}\\
			&\leq 
			\Big\| \|\partial_x  u\|_{L^{\alpha+1}_x} \|u\|_{L^{\alpha + 1}_x}^{\alpha-1} \Big\|_{L^\frac{4(\alpha+1)}{3\alpha+5}_t}.
			\end{split}
		\end{equation*}
		Notice that since $\alpha+1>2,$ then $H^1_x(\R)\hookrightarrow L^{\alpha +1}_x(\R).$ In particular, $\|u\|_{L^{\alpha+1}_x}\leq \|u\|_{H^1}.$ Using H\"older in time gives
		\begin{equation*}
			\begin{split}
				\Big\| \|\partial_x  u\|_{L^{\alpha+1}_x} \|u\|_{L^{\alpha + 1}_x}^{\alpha-1} \Big\|_{L^\frac{4(\alpha+1)}{3\alpha+5}_t}
				&\leq \|u\|_{L^\infty_t H^1_x}^{\alpha-1} \|\partial_x  u\|_{L^\frac{4(\alpha+1)}{3\alpha+5} L^{\alpha+1}_x}\\
				&\leq T^\delta \|u\|_{L^\infty_t H^1_x}^{\alpha-1} \|\partial_x  u\|_{L^\frac{4(\alpha+1)}{\alpha-1} L^{\alpha+1}_x}\\
				&\leq T^\delta \|u\|_{X_T}^\alpha,
			\end{split}
		\end{equation*}
		here $\delta$ is the H\"older conjugate exponent, namely
		\begin{equation*}
			\frac{3\alpha + 5}{4(\alpha+1)}= \frac{\alpha-1}{4(\alpha+1)} + \frac{1}{\delta}.
		\end{equation*}
		Notice that $\delta>0.$ This, with a similar reasoning as above, closes the fixed point argument.
		This intermediate step involving the Strichartz space $L^\frac{4(\alpha+1)}{\alpha-1}_t L^{\alpha +1}_x$ is needed to treat the 1D admissible endpoint case which is object of the next step.
		
\medskip	
	
		\item[\underline{Case $X_T=L^\infty_tH^1_x\cap L^4_t W^{1,\infty}_x$}] As above we need an estimate only for the Duhamel term. In order to do that we use the Strichartz estimate~\eqref{eq:S4} with $(\widetilde{p}, \widetilde{q})=(4(\alpha+1)/(3\alpha + 5), (\alpha+1)/\alpha).$ More precisely, using~\eqref{eq:S4} one has
		\begin{equation}\label{eq:end-point}
			\begin{split}
			\Big\| \int_0^t e^{i(t-s)\partial_x^2} |u|^{\alpha-1} \partial_x  u\, ds \Big\|_{L^4_t L^\infty_x}
			&\lesssim \||u|^{\alpha-1}\partial_x  u\|_{L_t^\frac{4(\alpha +1)}{3\alpha + 5} L^\frac{\alpha+1}{\alpha}_x}\\
			&\lesssim T^\delta \|u\|_{L^\infty_t H^1_x}^{\alpha-1} \|\partial_x  u\|_{L^\frac{4(\alpha+1)}{\alpha-1}_t L^{\alpha+1}_x},
			\end{split}
		\end{equation}  
		where the last inequality follows from the estimate obtained in the previous step. Now to conclude we want to use interpolation: let $\theta\in (0,1),$ then
		\begin{equation*}
			\begin{split}
				\|\partial_x  u\|_{L^{\alpha+1}_x}
				&= \Big( \int_{\R}|\partial_x  u|^{(\alpha+1)\theta} |\partial_x  u|^{(\alpha+1)(1-\theta)} \Big)^{1/(\alpha+1)}\\
				&\leq \|\partial_x  u\|_{L^\infty_x}^\frac{\alpha-1}{\alpha+1} \|\partial_x  u\|_{L^2_x}^\frac{2}{\alpha+1},
			\end{split}
		\end{equation*}
		where in the last estimate we used H\"older inequality and we chose $\theta=2/(\alpha+1).$ Using the latter to estimate $\|\partial_x  u\|_{L^\frac{4(\alpha+1)}{\alpha-1}_t L^{\alpha+1}_x}$ in~\eqref{eq:end-point} one has
		\begin{equation*}
			\begin{split}
			\|\partial_x  u\|_{L^\frac{4(\alpha+1)}{\alpha-1}_t L^{\alpha+1}_x}
			&\leq \Big \| \|\partial_x  u\|_{L^\infty_x}^\frac{\alpha-1}{\alpha+1} \|\partial_x  u\|_{L^2_x}^\frac{2}{\alpha+1}\Big \|_{L^\frac{4(\alpha +1)}{\alpha-1}_t}\\
			&\leq \|\partial_x  u\|_{L^\infty_t L^2_x}^\frac{2}{\alpha +1} \|\partial_x  u\|_{L^4_t L^\infty_x}^\frac{\alpha-1}{\alpha +1} \leq \|u\|_{X_T}.
			\end{split}
		\end{equation*}
		Using the last estimate in~\eqref{eq:end-point} we easily close the fixed point argument.
		
		In order to prove the theorem one observes that the pairs $(\infty, 2)$ and $(4, \infty)$ represent the endpoints of 1D admissible pairs $(p, q).$ Thus, the conclusion of the theorem simply follows from the interpolation inequality.   
		\end{description}
\end{proof}

\section{Short time stability}\label{section:stability}

\subsection{Preliminaries}
We need the following preliminary definition.
\begin{definition}[Approximate solution]\label{def:appr}
	We say that $u_{\#}$ is an approximate solution to \eqref{eq:NLS-anyalpha} if $u_{\#}$ satisfies the perturbed equation
	\begin{equation}\label{eq:perturbed}
	i \partial_t u_{\#} + \partial_x^2 u_{\#} + |u_{\#}|^{\alpha-1} u_{\#}= \mathcal{E}[u_{\#}],
	\qquad 2\leq\alpha<5,
\end{equation} 
for some error function $\mathcal{E}.$
\end{definition}
In this section we aim to develop a \emph{stability theory} for the NLS equation~\eqref{eq:NLS-anyalpha}: We will demonstrate that an approximate solution to the NLS equation~\eqref{eq:NLS-anyalpha} (in the sense of Definition~\ref{def:appr}) does not significantly deviate from the actual solution if both the error function $\mathcal{E}$ and the initial data error are small in a suitable sense. It is important to note that this result generalizes the property of continuous dependence on initial data, which corresponds to the special case where $\mathcal{E}=0$, as well as the uniqueness property, which corresponds to the case $\mathcal{E}=0$ and zero initial data error, namely $u(0)=u_{\#}(0).$

\subsection{Short time stability}
As a first stability result we show that we can prove the existence of an exact solution $u$ to~\eqref{eq:NLS-anyalpha} close enough to our approximate solution $u_{\#}$ of~\eqref{eq:perturbed}, even allowing $u_0-u_{\#}(0)$ to have large energy, provided that the error $\mathcal{E}[u_{\#}]$ of the near solution, the near solution $u_{\#}$ itself and the free evolution of the perturbation $u-u_{\#}$ are small in suitable space-time norms. The precise statement is contained in the next lemma.
\begin{lemma}[Short-time stability]
\label{lemma:stability}
	Let $I=[-T,T]$ be a fixed compact interval containing zero and assume $(0<)T\leq 1.$ Let $u_{\#}$ be an approximate solution to~\eqref{eq:NLS-anyalpha} on $I\times \R$ in the sense of Definition~\ref{def:appr}. Assume that 
	%\begin{equation*}
		$\|u_{\#}\|_{L^\infty_t H^1_x(I\times \R)}\leq A,$
	%\end{equation*}
	for some constant $A>0.$
	Let $u_0\in H^1(\R)$ such that 
	\[%begin{equation}\label{eq:H1}
		\|u_0-u_{\#}(0)\|_{H^1(\R)}\leq \widetilde{A},
	\]%end{equation}
	for some $\widetilde{A}>0.$
	 Assume the smallness conditions
	\begin{align}
		\label{eq:ug}
		\|u_{\#}\|_{L^p_t W^{1,q}_x(I\times \R)}&\leq \varepsilon_0,\\
		\label{eq:flux}
		\|e^{it\partial_x^2} (u_0-u_{\#}(0))\|_{L^p_t W^{1,q}_x(I\times \R)}&\leq \varepsilon,\\
		\label{eq:E}
		\left\|\mathcal{E}[u_{\#}]\right\|_{L^{p'}_t W^{1,q'}_x(I\times \R)}&\leq \varepsilon,
	\end{align}
	for some $0<\varepsilon\leq \varepsilon_0,$ with $\varepsilon_0=\varepsilon_0(A, \widetilde{A}, T)$ a sufficiently small constant and where $(p,q)$ is as in~\eqref{eq:pq}. 
	
	Then there exists a solution $u\in \mathcal{C}'(I\times \R)$ to~\eqref{eq:NLS-anyalpha} on $I\times \R$ with initial datum $u_0\in H^1(\R)$ such that 
	\begin{align}
		\label{eq:small}
		\|u-u_{\#}\|_{L^p_t W^{1,q}_x(I\times \R)}&\lesssim \varepsilon,\\
		\label{eq:S}
		\|(i\partial_t + \partial_x^2)(u-u_{\#}) + \mathcal{E}[u_{\#}]\|_{L^{p'}_t W^{1,q'}_x(I\times \R)}&\lesssim \varepsilon\\
		\label{eq:large}
		\|u-u_{\#}\|_{\mathcal{C}'(I\times \R)}&\lesssim \widetilde{A} + \varepsilon.
	\end{align}
	In particular, from~\eqref{eq:large} one has  $\|u-u_{\#}\|_{L^\infty_t H^1_x(I\times \R)}\lesssim \widetilde{A} + \varepsilon.$
\end{lemma}

\begin{remark}\label{R3p1}
Notice that from the Strichartz estimate~\eqref{eq:S1}, the hypothesis~\eqref{eq:flux} is redundant if one is willing to assume also smallness of the energy of $u_0-u_{\#}(0),$ namely asking for $\widetilde{A}=\mathcal{O}(\varepsilon).$  
\end{remark}

\begin{remark}[$\alpha$ restrictions in~\ref{def:appr}]
As already emphasized in the introduction, the lower bound $\alpha \geq 2$ seems to us just a technical requirement needed in the proof (see equation~\eqref{eq:fix-point-2} below). In other words, we believe that our results should remain valid even in the range $1<\alpha< 5.$ Regarding the upper bound $\alpha< 5,$ while it corresponds to the standard $L^2$-subcritical regime and may seem somewhat unusual in the context of our $H^1$-setting, it remains reasonable in this framework. Indeed, the  proof of Lemma 3.1 must, in a sense, ``forget" the energy space $H^1$: since it relies only on the smallness condition in hypothesis~\eqref{eq:flux} and not on the smallness of the difference of initial data in the energy space (we only assume $\|u_0-u_{\#}(0)\|_{H^1(\R)}\leq \widetilde{A},$ with $\widetilde{A}>0$), one is compelled to employ bounds that circumvent the energy space. This leads naturally to the $L^2$-motivated constraint $\alpha<5.$
\end{remark}

\begin{proof}[Proof of Lemma~\ref{lemma:stability}]
	By the well-posedness theory developed in the previous section (Theorem~\ref{thm:LWP}), we can assume that the solution $u$ to~\eqref{eq:NLS-anyalpha} already exists and belongs to $\mathcal{C}'.$ Thus to prove the result we only need to show~\eqref{eq:small}-\eqref{eq:large} as a priori estimates. We establish these bounds for $t\geq 0,$ the portion of $I$ corresponding to $t<0$ can be treated similarly.
	  
	Let $v:=u-u_{\#}.$ Then $v$ satisfies the following equation
	\begin{equation}\label{eq:v}
		i\partial_t v + \partial_x^2 v + f(u_{\#} + v) - f(u_{\#}) + \mathcal{E}[u_{\#}]=0,
		\qquad f(s):=|s|^{\alpha-1}s.
	\end{equation}
	Using the integral equation associated to~\eqref{eq:v} one has
	\begin{equation}\label{eq:integral}
		v(t)=e^{it\partial_x^2}(u_0-u_{\#}(0)) 
		+ i \int_0^t e^{i(t-s)\partial_x^2}[f(u_{\#}+v)-f(u_{\#})]\, ds
		+i \int_0^t e^{i(t-s)\partial_x^2} \mathcal{E}[u_{\#}]\, ds.
	\end{equation}
	Let $\tau\in I,$ we will now work on the slab $[0,\tau]\times \R.$ 
	
	Using~\eqref{eq:S1} for the first term in~\eqref{eq:integral},~\eqref{eq:S2} for the second and the third terms and then hypotheses~\eqref{eq:ug} and~\eqref{eq:E} we have
	\begin{equation}\label{eq:sup}
		\begin{split}
		\|v\|_{\mathcal{C}'}
		&\lesssim
		\|u_0-u_{\#}(0)\|_{H^1_x}
		+\|f(u_{\#}+v)-f(u_{\#})\|_{L^{p'}_t W^{1, q'}_x}
		+\|\mathcal{E}[u_{\#}]\|_{L^{\widetilde{p}'}_t W^{1, \widetilde{q}'}_x}\\
		&\lesssim \widetilde{A} + S(\tau) + \varepsilon,  
		\end{split}
	\end{equation}
	where we have defined
	\begin{equation}\label{ST}
		S(\tau):=\|f(u_{\#}+v)-f(u_{\#})\|_{L^{p'}_t W^{1, q'}_x([0,\tau]\times \R)}.
	\end{equation}
	Consequently (see \cite[Remark 2.4]{KM}),
	\begin{equation}\label{eq:diff-est}
	\begin{aligned}
	|\partial_x [f(u_{\#}+v)-f(u_{\#})]| \lesssim   |\partial_x  v||u_{\#}|^{\alpha-1} + |\partial_x  v||v|^{\alpha-1}+ |\partial_x  u_{\#}| |v|^{\alpha-1} + |u_\#|^{\alpha-2}|\partial_x u_\#||v|. 
	%\lesssim &~{} |\partial_x  u_{\#}| |v|^{\alpha-1} + |\partial_x  v||u_{\#}|^{\alpha-1} + |\partial_x  v||v|^{\alpha-1} \\
%		&~{} \color{blue} + | u_\#|^{\alpha-2} | \partial_x u_\#  |  | v | + | \partial_x u_\#  | | v |^{\alpha-1} \\
%		&~{} \color{blue} + | u_\#|^{\alpha-2}|\partial_xv| + | u_\#|^{\alpha-3}|\partial_xu_\#| |v| +  | u_\#|^{\alpha-3}|v| |\partial_xv| \\
%	         & ~{} \color{blue}  + |u_\#|  | v |^{\alpha-3}|\partial_xv| + |\partial_xu_\#|| v |^{\alpha-2} + | v |^{\alpha-2}|\partial_xv|\\
%		&~{} \color{blue} + | v | | u_\#|^{\alpha-3}|\partial_x u_\# |+  | v |^{\alpha-3}|u_\#| |\partial_x u_\#|  \\
		%\lesssim &~{} \color{blue} |u_{\#}|^{\alpha-1}  |\partial_x  v| + |\partial_x  v||v|^{\alpha-1}  + | \partial_x u_\#  | | v |^{\alpha-1} + | u_\#|^{\alpha-2} | \partial_x u_\#  |  | v |\\
		%&~{} \color{blue} + | u_\#|^{\alpha-2}|\partial_xv| + | u_\#|^{\alpha-3}|\partial_xu_\#| |v| +  | u_\#|^{\alpha-3}|v| |\partial_xv| \\
	        % & ~{} \color{blue}  + |u_\#|  | v |^{\alpha-3}|\partial_xv| + |\partial_xu_\#|| v |^{\alpha-2} + | v |^{\alpha-2}|\partial_xv| + |u_\#| |\partial_x u_\#|  | v |^{\alpha-3}.
	\end{aligned}
	\end{equation}
	%(for example, see~\cite{TV}). 
	Thus, in order to bound $S(t)$ we need to estimate terms of the following type $\||w_1|^{\alpha-1}\partial_x  w_2\|_{L^{p'}_t L^{q'}_x}.$ Using the H\"older inequality one obtains
	\begin{equation}\label{eq:fix-point}
		\begin{split}
			\||w_1|^{\alpha-1}\partial_x  w_2\|_{L^{p'}_t L^{q'}_x}
			=\||w_1|^{\alpha-1}\partial_x  w_2\|_{L^\frac{4(\alpha+1)}{3\alpha +5}_t L^\frac{\alpha+1}{\alpha}_x}
			&\leq \Big\| \|\partial_x  w_2\|_{L^{\alpha+1}_x} \|w_1\|_{L^{\alpha+1}_x}^{\alpha-1}\Big\|_{L^\frac{4(\alpha+1)}{3\alpha + 5}_t}\\
			&\leq\|\partial_x  w_2\|_{L^\frac{4(\alpha+1)}{\alpha-1}_t L^{\alpha +1}_x} \|w_1\|_{L^\frac{2(\alpha^2-1)}{\alpha + 3}_t L^{\alpha+1}_x}^{\alpha-1}\\
			&\leq T^\delta \|\partial_x  w_2\|_{L^\frac{4(\alpha+1)}{\alpha-1}_t L^{\alpha +1}_x}\|w_1\|_{L^\frac{4(\alpha+1)}{\alpha-1}_t L^{\alpha +1}_x}^{\alpha-1}\\
			&\leq T^\delta \|\partial_x  w_2\|_{L^{p}_t L^{q}_x}\|w_1\|_{L^{p}_t L^{q}_x}^{\alpha-1},
		\end{split}
	\end{equation}
	where $\delta$ is a H\"older conjugate exponent which is positive if $\alpha <5$ being $\frac{2(\alpha^2-1)}{\alpha+3}<\frac{4(\alpha+1)}{\alpha-1}.$
	Similarly, if $\alpha \geq 2$
	\begin{equation}\label{eq:fix-point-2}
	\begin{aligned}
	\||w_1|^{\alpha-2} w_2\partial_x  w_3\|_{L^{p'}_t L^{q'}_x} \le &~{} \left\| \| w_1\|_{L_x^{\alpha+1}}^{(\alpha-2)}\| w_2\|_{L_x^{\alpha+1}} \|\partial_x  w_3\|_{L_x^{\alpha+1}}\right\|_{L^{p'}_t}\\
	\le &~{} T^{\delta'} \|w_1\|_{L^p_tL^q_x}^{\alpha-2} \|w_2\|_{L^p_tL^q_x}\| w_3\|_{L^p_tW^{1,q}_x}.
	\end{aligned}
	\end{equation}
	Here, $\delta'=\frac{3\alpha+5}{4(\alpha +1)} - \frac{\alpha(\alpha-1)}{4(\alpha +1)}>0$ in the subcritical range $\alpha\in [2,5)$. Using~\eqref{eq:diff-est} and~\eqref{eq:fix-point}-\eqref{eq:fix-point-2} and using that we are assuming $T\leq 1,$ one has
	\begin{equation}\label{eq:ST}
		\begin{split}
			S(\tau)
			&\lesssim \|\partial_x  u_{\#}|v|^{\alpha-1}\|_{L^{p'}_t L^{q'}_x}+\|\partial_x  v |u_{\#}|^{\alpha-1}\|_{L^{p'}_t L^{q'}_x} \\
			& \qquad +\|\partial_x  v |v|^{\alpha-1}\|_{L^{p'}_t L^{q'}_x} + \|\partial_x u_\# |u_\#|^{\alpha-2} v \|_{L^{p'}_t L^{q'}_x} \\
			&\lesssim  \|\partial_x  u_{\#}\|_{L^{p}_t L^{q}_x}\|v\|_{L^{p}_t L^{q}_x}^{\alpha-1} + \|\partial_x  v\|_{L^{p}_t L^{q}_x}\|u_{\#}\|_{L^{p}_t L^{q}_x}^{\alpha-1} \\
			& \qquad \quad + \|\partial_x  v\|_{L^{p}_t L^{q}_x}\|v\|_{L^{p}_t L^{q}_x}^{\alpha-1} + \| \partial_x u_\# \|_{L^{p}_t L^{q}_x} \|u_\# \|_{L^{p}_t L^{q}_x}^{\alpha-2} \| v \|_{L^{p}_t L^{q}_x}  \\
			&\lesssim  \| u_{\#}\|_{L^{p}_t W^{1,q}_x}\|v\|_{L^{p}_t W^{1,q}_x}^{\alpha-1} + \| v\|_{L^{p}_t W^{1,q}_x} \|u_{\#}\|_{L^{p}_t W^{1,q}_x}^{\alpha-1} 
			\\
			& \qquad \quad + \| v\|_{L^{p}_t W^{1,q}_x}^\alpha + \| u_\# \|_{L^{p}_t W^{1,q}_x} \|u_\# \|_{L^{p}_t L^{q}_x}^{\alpha-2} \| v\|_{L^{p}_t L^{q}_x}.
		\end{split}
	\end{equation}
	Notice that the hidden constants in the $\lesssim$ notation in~\eqref{eq:ST} and in the forthcoming equations are independent of $T.$
	
	Using~\eqref{eq:integral} and \eqref{ST}, on the one hand one has
	\begin{equation}\label{eq:str-norm}
		\begin{split}
		\|v\|_{L^p_t W^{1,q}_x}
		&\lesssim \|e^{it\partial_x^2}(u_0-u_{\#}(0))\|_{L^p_t W^{1,q}_x} + \|f(u_{\#}+v)-f(u_{\#})\|_{L^{p'}_t W^{1,q'}_x} + \|\mathcal{E}[u_{\#}]\|_{L^{{p}'}_t W^{1,{q}'}_x}\\
		&\lesssim S(\tau) + \varepsilon.
		\end{split}
	\end{equation}
	where in the last inequality we have used~\eqref{eq:flux} and~\eqref{eq:E}. On the other hand, from~\eqref{eq:ST}, \eqref{eq:ug} and \eqref{eq:str-norm}, %we have
	\begin{equation*}
		S(\tau)\lesssim   \varepsilon_0 (S(\tau) + \varepsilon)^{\alpha-1}
		+ \varepsilon_0^{\alpha-1} (S(\tau)+ \varepsilon) + (S(\tau)+ \varepsilon)^\alpha.
	\end{equation*}
	
	Taking $\varepsilon_0$ sufficiently small, a standard continuity argument gives
	\begin{equation*}
		S(\tau)\leq \varepsilon, 
		\qquad \text{for all } \tau\in I,
	\end{equation*}
	which implies~\eqref{eq:S}.
	Using this bound in~\eqref{eq:str-norm} and~\eqref{eq:sup} gives~\eqref{eq:small} and~\eqref{eq:large}, respectively. Thus the proof of the lemma is concluded.
\end{proof}

\section{Long time stability}\label{sec:LTS}

We will now focus on iterating the aforementioned result to address the more general scenario of near-solutions with \emph{finite} yet \emph{arbitrarily large} space-time norms. Specifically, our objective is to prove the following alternative lemma.
\begin{lemma}[Long-time stability]\label{lemma:stability-long}
Let $I=[-T,T]$ be a fixed compact interval containing zero. Let $u_{\#}$ be an approximate solution to~\eqref{eq:NLS-anyalpha} on $I\times \R$ in the sense of Definition~\ref{def:appr}. Assume that 
	\begin{equation}\label{base}
		\|u_{\#}\|_{L^\infty_t H^1_x(I\times \R)}\leq A,
	\end{equation}
	for some constant $A>0.$ 
	Let $u_0\in H^1(\R)$ such that 
	\begin{equation}\label{eq:H1-alt}
		\|u_0-u_{\#}(0)\|_{H^1(\R)}\leq \widetilde{A},
	\end{equation}
	for some $\widetilde{A}>0.$
	 Assume
	\begin{align}
		\label{eq:bound-B}
		\|u_{\#}\|_{L^p_t W^{1,q}_x(I\times \R)}&\leq B,\\ 
		\label{eq:flux-alt}
		\|e^{it\partial_x^2} (u_0-u_{\#}(0))\|_{L^p_t W^{1,q}_x(I\times \R)}&\leq \varepsilon,\\
		\label{eq:E-alt}
		\|\mathcal{E}[u_{\#}]\|_{L^{p'}_t W^{1,q'}_x(I\times \R)}&\leq \varepsilon,
	\end{align}
	for some constant $B>0,$ $0<\varepsilon\leq \varepsilon_1,$ with $\varepsilon_1=\varepsilon_1(A, \widetilde{A},B,T)>0$ a small constant, where $(p,q)$ is as in~\eqref{eq:pq}.
	
	Then there exists a unique solution $u\in \mathcal{C}'(I\times \R)$ to~\eqref{eq:NLS-anyalpha} on $I\times \R$ with initial data $u_0\in H^1(\R)$ such that 
	\begin{align}
		\label{eq:small-alt}
		\|u-u_{\#}\|_{L^p_t W^{1,q}_x(I\times \R)}&\leq C(A,\widetilde{A},B, T) \varepsilon,\\
		%\label{eq:S}
		%\|(i\partial_t + \partial_x^2)(u-u_{\#}) + \mathcal{E}[u_{\#}]\|_{L^{p'}_t W^{1,q'}_x(I\times \R)}&\leq C(A,\widetilde{A},B) \varepsilon\\
		\label{eq:large-alt}
		\|u-u_{\#}\|_{\mathcal{C}'(I\times \R)}&\leq C(A,\widetilde{A},B,T).		
	\end{align}
	In particular, from~\eqref{eq:large-alt} one has  $\|u-u_{\#}\|_{L^\infty_t H^1_x(I\times \R)}\leq C(A,\widetilde{A},B,T).$
\end{lemma}

%{\color{black}
\begin{remark}
There are two main differences between Lemmas \ref{lemma:stability} and \ref{lemma:stability-long}. The first one is in \eqref{eq:bound-B}, where $B$ is allowed to be large, unlike in \eqref{eq:ug}. The second is in \eqref{eq:S}, which is not present in Lemma \ref{lemma:stability-long}.
\end{remark}
%}
\begin{remark}
	We stress that, as it can be easily seen from the proof below, the $T$ dependence in the constants above should not be confused with mere continuity of the flux.
\end{remark}

Lemma~\ref{lemma:stability-long} can be obtained as an easy consequence of Lemma~\ref{lemma:stability} just using an iteration argument based on partitioning the time interval.
\begin{proof}[Proof of Lemma~\ref{lemma:stability-long}]
	Without loss of generality we may assume that $t_0:=0$ is the lower bound of the interval $I.$ Let $\varepsilon_0=\varepsilon_0(A, 2\widetilde{A},T)$ be as in Lemma~\ref{lemma:stability}. Since from~\eqref{eq:bound-B} the $L^p_tW^{1,q}_x(I\times \R)$ norm of $u_{\#}$ is finite, we can subdivide the interval $I$ into $N\leq C({\color{black}A},B,\varepsilon_0)$ sub-intervals $I_j=[t_j,t_{j+1}]$ such that $|t_{j+1}-t_j|\leq 1$ for $j\in \{0,\dots, N-1\}$ and such that on each $I_j$ we have
\begin{equation}\label{eq:small-j}
	\|u_{\#}\|_{L^p_t W^{1,q}_x(I_j\times \R)}\leq \varepsilon_0.
\end{equation}
Recall that we are assuming that $0<\varepsilon\leq \varepsilon_1$. Choosing $\varepsilon_1$ sufficiently small depending on $\varepsilon_0, N, A$ and $\widetilde{A},$ we can apply Lemma~\ref{lemma:stability} to obtain that for each $j\in \{0,\dots, N-1\}$ then we have 
\begin{equation}\label{eq:Ij}
\begin{split}
		\|u-u_{\#}\|_{L^p_t W^{1,q}_x(I_j\times \R)}&\lesssim_j \varepsilon,\\
		%\|(i\partial_t + \partial_x^2)(u-u_{\#}) - \mathcal{E}[u_{\#}]\|_{L^p_t W^{1,q}_x(I_j\times \R)}&\lesssim_j \varepsilon,\\
		\|u-u_{\#}\|_{\mathcal{C}'(I_j\times \R)}&\lesssim_j \widetilde{A} + \varepsilon,	
\end{split}
\end{equation} 
provided we can show that~\eqref{eq:H1-alt} and the smallness assumption~\eqref{eq:flux-alt} holds with $t_0=0$ replaced by the lower bound $t_j$ of the sub-interval $I_j.$ Notice that assumption~\eqref{eq:E-alt} is trivially true for $I$ replaced by $I_j.$  
%smallness assumption~\eqref{eq:flux-alt} and~\eqref{eq:E-alt} hold with $t_0=0$ replaced by the lower bound $t_j$ of the sub-interval $I_j.$  First of all, notice that 
\[
(i\partial_t + \partial_x^2)(u-u_{\#}) + \mathcal{E}[u_{\#}] = f(u_{\#}) -f(u_{\#} +v), \quad u= u_{\#} +v.
\]
We shall prove by an inductive argument that
\begin{equation*}
	S(t_l):=\| (i\partial_t + \partial_x^2)(u-u_{\#}) + \mathcal{E}[u_{\#}]\|_{L^p_t W^{1,q}_x(I_l\times \R)}\leq C(l)\varepsilon,
\end{equation*}
for any $l\in \{0, \ldots,N-1\}$, and where $C(l) \geq 1$ is a fixed quantity independent of $\varepsilon$.  

From Lemma~\ref{lemma:stability} applied to $I_0=[0,t_1],$ one has $S(t_0)\lesssim \varepsilon$ (\emph{cf.}~\eqref{eq:S}). This inequality also defines $C(0).$ Assume up to the case $l-1$, where $C(l) \geq 1$ is a fixed quantity independent of $\varepsilon$. Arguing as in the proof of Lemma~\ref{lemma:stability} eqn. \eqref{eq:ST}, and for $l\in \{1, \ldots, N-1\}$, we get for $v= u-u_{\#}$ and on each interval $I_l,$
\begin{equation}\label{eq:Sj}
\begin{aligned}
S(t_l)
%\leq& \sum_{k=1}^\alpha \|v\|_{L^p_t W^{1,q}([0,t_{l+1}]\times \R)}^k\|u_{\#}\|_{L^p_t W^{1,q}(I_l\times \R)}^{\alpha-k}\\
&\lesssim |I_l|^\delta \Big(\| v\|_{L^{p}_t W^{1,q}_x([0,t_{l+1}]\times \R)}^\alpha+ \|v\|_{L^{p}_t W^{1,q}_x([0,t_{l+1}]\times \R)}^{\alpha-1}  \| u_{\#}\|_{L^{p}_t W^{1,q}_x(I_l\times \R)} \\
& \qquad \quad + \| v\|_{L^{p}_t W^{1,q}_x([0,t_{l+1}]\times \R)} \|u_{\#}\|_{L^{p}_t W^{1,q}_x(I_l\times \R)}^{\alpha-1} +  \| v\|_{L^{p}_t L^{q}_x([0,t_{l+1}]\times \R)} \| u_\# \|_{L^{p}_t W^{1,q}_x(I_l\times \R)}^{\alpha-1}\Big).
\end{aligned}
\end{equation}
From the integral equation~\eqref{eq:integral} one has
\begin{multline*}
	\|v\|_{L^p_t W^{1,q}([0,t_{l+1}]\times \R)}\\
	\begin{aligned}
	&\leq C\|e^{it\partial_x^2}(u_0-u_{\#}(0))\|_{L^p_t W^{1,q}([0,t_{l+1}]\times \R)}
	+ C \|(i\partial_t + \partial_x^2)(u-u_{\#}) +\mathcal{E}[u_{\#}]\|_{L^{p'}_t W^{1, q'}_x([0,t_{l+1}]\times \R)}\\
&		+C\|\mathcal{E}[u_{\#}]\|_{L^{{p}'}_t W^{1, {q}'}_x(I\times \R)}\\
	&\leq C\|e^{it\partial_x^2}(u_0-u_{\#}(0))\|_{L^p_t W^{1,q}([0,t_{l+1}]\times \R)}
	+C \sum_{m=0}^{l-1} S(t_m) + C S(t_l) +C \|\mathcal{E}[u_{\#}]\|_{L^{{p}'}_t W^{1, {q}'}_x(I\times \R)}\\
	&\le C_1 \varepsilon  + C_2\varepsilon \sum_{m=0}^{l-1} C(m)   +C_3 S(t_l)  \\
	&\leq C(l)( S(t_l)+ \varepsilon), \quad C(l):= \max\left\{  C_1,C_2\sum_{m=0}^{l-1} C(m) ,C_3\right\},
\end{aligned} 
\end{multline*}
where the last bound follows from~\eqref{eq:flux-alt},~\eqref{eq:E-alt} and the inductive hypothesis $S(t_m)\leq C(m)\varepsilon$ for any $m\in\{0,\ldots, l-1\}.$ Plugging this bound in~\eqref{eq:Sj}, using also~\eqref{eq:small-j},
\begin{equation*}
		S(t_l)\lesssim |I_l|^\delta \left(  \varepsilon_0C(l)^{\alpha-1} (S(t_l) + \varepsilon)^{\alpha-1}
		+ \varepsilon_0^{\alpha-1} C(l)(S(t_l)+ \varepsilon) +C(l)^\alpha (S(t_l)+ \varepsilon)^\alpha \right),
	\end{equation*}
and then a continuity argument using that $\varepsilon_0$ is small as above and $\alpha\ge 2,$ gives
\begin{equation}\label{eq:bound-Sj}
S(t_l)\leq C(l) \varepsilon.
\end{equation}
Now, using~\eqref{eq:integral}, by the unitarity in $H^1_x$ of the operator $e^{it\partial_x^2}$ and from Strichartz estimate~\eqref{eq:S4} one has
\begin{equation*}
	\begin{split}
		\|u(t_{j})-u_{\#}(t_{j})\|_{H^1_x(\R)}&\leq \|u(t)-u_{\#}(t)\|_{L^\infty_tH^1_x([0,t_{j}] \times \R)}\\
		&\leq \|u_0-u_{\#}(0)\|_{H^1_x(\R)}
		+ \sum_{l=0}^{j-1} S(t_l) 
		+\|\mathcal{E}[u_{\#}]\|_{L^{{p}'}_t W^{1, {q}'}_x(I\times \R)}\\
		&\lesssim \widetilde{A} + \varepsilon + \sum_{l=0}^{j-1} C(l)\varepsilon,
	\end{split}
\end{equation*}  
where the last bound follows from~\eqref{eq:H1-alt},~\eqref{eq:E-alt} and~\eqref{eq:bound-Sj}.

Similarly, one has
\begin{multline*}
		\|e^{i(t-t_{j})\partial_x^2} (u(t_{j})-u_{\#}(t_{j}))\|_{L^p_t W^{1,q}_x(I_{j}\times \R)}\\
		\begin{aligned}
		&\leq 
		\|e^{it\partial_x^2}(u_0-u_{\#}(0))\|_{L^p_t W^{1,q}_x([0,t_{j+1}]\times \R)}
		+ \left\| \int_0^{t_j} e^{i(t-s)\partial_x^2} (i\partial_t + \partial_x^2)(u-u_{\#})\, ds \right\|_{L^p_t W^{1,q}_x([0,t_{j+1}]\times \R)}\\
		&\leq \|e^{it\partial_x^2}(u_0-u_{\#}(0))\|_{L^p_t W^{1,q}_x([0,t_{j+1}]\times \R)} 
		 + \sum_{l=0}^{j} S(t_l) 
		+\|\mathcal{E}[u_{\#}]\|_{L^{{p}'}_t W^{1, {q}'}_x(I\times \R)}\\
		&\leq \varepsilon +  \sum_{l=0}^{j} C(l) \varepsilon. 
		\end{aligned}
\end{multline*} 
Since $0<\varepsilon <\varepsilon_1$, choosing $\varepsilon_1$ sufficiently small, one has that~\eqref{eq:Ij} holds true for any $j\in \{0,\dots, N-1\}.$ Thus, bounds~\eqref{eq:small-alt} and~\eqref{eq:large-alt} follow summing over all the intervals $I_j,$ $j\in \{0,\dots, N-1\}.$

\end{proof}

\section{Proof of the Main Result}\label{sec:PMR}

%\subsection{PINNs Neural networks} 
%
%
%\begin{proof}
%\[
%\mathcal J_{0,N}[f,g,z] = \frac1N\sum_{j=1}^N |f(z_j)- g(z_j)|^2
%\]
%\[
% \mathcal J_{edp,N}[h,w,z] =  \frac1N \sum_{j=1}^N  | ( i\partial_t h + \partial_x^2 h + |h|^{p-1} h)(w_j,z_j) |^2
%\]
%\end{proof}
%
%\subsection{End of proof} 

Consider an interval of time $I$ containing zero. Let $u_0\in H^1(\mathbb R)$. Assume (H1). Let $A,\widetilde A, B>0$ be fixed values. Let $0<\varepsilon <\varepsilon_1(A, \widetilde{A},B)$. Recall 
\[
(p,q)=\left(\tfrac{4(\alpha+1)}{\alpha-1},\alpha+1\right).
\]
Let $u(t)$ be the corresponding solution to NLS with initial data $u_0$. Let $u_{\text{DNN},\#}$ be the realization of a DNN $\Phi_{\#}$ as in Theorem \ref{MT}, such that (H2) and (H3) are satisfied in the following sense:
\begin{enumerate}
\item[(Step 1)] For $N_1,N_2,N_3$ and $M_2,M_3$ sufficiently large, collocation points, times and weights, one has
\[
\mathcal J_{\infty,H^1,N_2,M_2} [u_{\text{DNN},\#}] \leq \frac12 A, \quad \mathcal J_{H^1,N_1}[ (u_0-u_{\text{DNN},\#}(0))]\leq \frac12\widetilde{A}, \quad \mathcal J_{p,q,N_3,M_3} [u_{\text{DNN},\#}]  \leq \frac12B.
\]
\end{enumerate}
Later we will choose specific values for the parameters of the approximative norms. Similarly, we will require as well that
\begin{enumerate}
\item[(Step 2)] For all $N_4,M_4$ and $N_5,M_5$ positive integers large enough, collocation points, times and weights, 
\[
\begin{aligned}
 \mathcal J_{p',q',N_4,M_4} [\mathcal E[u_{\text{DNN},\#}]] + \mathcal J_{p,q,N_5,M_5} [ e^{it\partial_x^2} (u_0-u_{\text{DNN},\#}(0))]  < \frac12\varepsilon,
\end{aligned}
\]
with $\mathcal E[u_\#]:= i\partial_t u_\# + \partial_x^2 u_\# + |u_\#|^{\alpha-1}u_\#$. 
\end{enumerate}
Let 
\[
R:= \max_{n=1,2,3,4}\max_{j=1,\ldots, N_1}\{|x_{n,j}|\}
\]
be the maximum modulus of evaluation/collocation points among all the previous approximate integrals.  Let $[-R,R]\subseteq \mathbb R$. Let $\eta_R\in C_0^\infty(\mathbb R)$, $0\leq \eta\leq 1$, be a cut-off function such that $\eta =1$ in $[-R,R]$ and $\eta(x) =0$ if $|x|\geq 2R$. Define $u_{\#}:= u_{\text{DNN},\#}\eta_R$ now in Sobolev and Lebesgue spaces. Note that for all $t\in I$, and all $x\in[-R,R]$, one has $u_{\#}= u_{\text{DNN},\#}$. Consequently,
\[
\begin{aligned}
& \mathcal J_{\infty,H^1,N_2,M_2} [u_{\text{DNN},\#}]=\mathcal J_{\infty,H^1,N_2,M_2} [u_{\#}] , \\
& \mathcal J_{H^1,N_1} [(u_0-u_{\text{DNN},\#}(0))] =\mathcal J_{H^1,N_1} [(u_0-u_{\#}(0))], \quad \mathcal J_{p,q,N_3,M_3} [u_{\text{DNN},\#}] =\mathcal J_{p,q,N,M} [u_{\#}] ,
\end{aligned}
\]
and
\[
\begin{aligned}
& \mathcal J_{p',q',N_4,M_4} [\mathcal E[u_{\text{DNN},\#}]] =\mathcal J_{p',q',N_4,M_4} [\mathcal E[u_{\#}]],\\
& \mathcal J_{p,q,N_5,M_5} [ e^{it\partial_x^2} (u_0-u_{\text{DNN},\#}(0))]= \mathcal J_{p,q,N_5,M_5} [ e^{it\partial_x^2} (u_0-u_{\#}(0))].
\end{aligned}
 \]
Now (modulo choosing $N_j$, $M_j$ once again larger, and $R$ larger if necessary, which does not affect the last identities) we can use hypothesis (H1) and choose particular integration parameters for $f=u_0-u_{\#}(0)$ and $\delta =\frac12\varepsilon$ to get 
\[
\|  u_0-u_{\#}(0) \|_{H^1(\R)} \leq \mathcal J_{H^1,N_1} [(u_0-u_{\#}(0))]  +\frac12 \varepsilon.
\]
Similarly, for $g_2=g_3= u_{\#}$ and $\delta =\frac12\varepsilon$, we choose integration parameters such that
\[
 \begin{aligned}
&  \| u_{\#} \|_{L^\infty_tH^1(I\times \R)} \leq \mathcal J_{\infty,H^1,N_2,M_2} [u_{\#}] +\frac12\varepsilon,\\
 & \|u_{\#}\|_{L^p_t W^{1,q}_x(I\times \R)} \leq \mathcal J_{p,q,N_3,M_3} [u_{\#}] +\frac12\varepsilon. 
 \end{aligned}
 \]
Therefore,
 \[
 \begin{aligned}
 & \|  u_0-u_{\#}(0) \|_{H^1(\R)} \leq \mathcal J_{H^1,N_1} [(u_0-u_{\text{DNN},\#}(0))]  + \frac12\varepsilon \leq \frac12(\widetilde A +\varepsilon) \leq \widetilde A,\\
 & \| u_{\#} \|_{L^\infty_tH^1(I\times \R)} \leq \mathcal J_{\infty,H^1,N_2,M_2} [u_{\text{DNN},\#}] + \frac12\varepsilon \leq \frac12(A+\varepsilon)  \leq A,\\
 & \|u_{\#}\|_{L^p_t W^{1,q}_x(I\times \R)} \leq \mathcal J_{p,q,N_3,M_3} [u_{\text{DNN},\#}] +\frac12 \varepsilon \leq \frac12(B+ \varepsilon) \leq B. 
 \end{aligned}
 \]
 Additionally, using again the Hypothesis (H1),
  \[
 \begin{aligned}
 & \|\mathcal{E}[u_{\#}]\|_{L^{p'}_t W^{1,q'}_x(I\times \R)} \leq \mathcal J_{p',q',N_4,M_4} [\mathcal E[u_{\text{DNN},\#}]] + \frac12\varepsilon \leq  \frac12\varepsilon + \frac12\varepsilon  = \varepsilon,\\
 & \| e^{it\partial_x^2} (u_0-u_{\#}(0)) \|_{L^p_tW^{1,q}_x(I\times \R)} \leq \mathcal J_{p,q,N_5,M_5} [ e^{it\partial_x^2} (u_0-u_{\text{DNN},\#}(0))] + \frac12\varepsilon  \leq  \frac12\varepsilon + \frac12\varepsilon = \varepsilon.
 \end{aligned}
 \]
The hypotheses \eqref{base}-\eqref{eq:E-alt} in Lemma \ref{lemma:stability-long} are satisfied and one has that there exists a solution $u\in \mathcal{C}'(I\times \R)$ to~\eqref{eq:NLS-anyalpha} on $I\times \R$ with initial data $u_0\in H^1(\R)$ such that 
	\begin{align*}
		%\label{eq:small-alt_final}
		\|u-u_{\#}\|_{L^p_t W^{1,q}_x(I\times \R)}&\leq C(A,\widetilde{A},B) \varepsilon,\\
		%\label{eq:S}
		%\|(i\partial_t + \partial_x^2)(u-u_{\#}) + \mathcal{E}[u_{\#}]\|_{L^{p'}_t W^{1,q'}_x(I\times \R)}&\leq C(A,\widetilde{A},B) \varepsilon\\
		%\label{eq:large-alt_final}
		\|u-u_{\#}\|_{\mathcal{C}'(I\times \R)}&\leq C(A,\widetilde{A},B),\\
		\|u-u_{\#}\|_{L^\infty_t H^1_x(I\times \R)}& \leq C(A,\widetilde{A},B).	
	\end{align*}
The final conclusion \eqref{eq:small-alt_0}-\eqref{eq:large-alt_0}-\eqref{eq:H1-alt_0} are obtained after considering the subinterval $[-R,R]$ and recalling that $u_{\#}= u_{\text{DNN},\#}$ in this interval.

\section{Numerical results}\label{sec:numerical}

\subsection{Application to Solitary waves} Theorem \ref{MT} establishes a rigorous framework for the approximation of NLS solitons. Solitary waves are a fundamental aspect of the focusing NLS equation \cite{Cazenave}. When $\alpha = 3$, the NLS becomes integrable, and exhibits solitons, multisolitons, and breather solutions. For a detailed exploration of these phenomena, see \cite{Cazenave,Cazenave2, AFM, AFM2, AC, SY}. Additionally, for more information on breather solutions in other models, references such as \cite{Wad, AM, BMW, AMP} provide further insight. 

\begin{corollary}\label{Cor1}
Let $\alpha\in[2,5)$ and $\varepsilon>0$. Let $u(t)\in H^1(\R)$ be any NLS solitary wave solution; or soliton, multi-soliton or breather if $\alpha=3$. Then there is $\widetilde A=\widetilde A(\varepsilon, I)=o_\varepsilon(1)$, such that the following holds. Assume that $u_{\emph{DNN},\#}$ is a DNN-generated realization of a PINNs $\Phi$ such that \emph{(H1)-(H3)} are satisfied with $\widetilde A$ as above. Then for all $R>0$ sufficiently large,
\begin{align*}
\|u-u_{\emph{DNN},\#}\|_{L^\infty_t H^1_x(I\times [-R,R])}&\lesssim_{A,B,{\color{black}T}} \varepsilon.
\end{align*}
\end{corollary}

The proof of this fact immediately follows from Remark \ref{R3p1} and the continuity of the flow. It is an open question to introduce and mix standard stability techniques in the description of PDEs via PINNs.
Although Corollary \ref{Cor1} may be considered as a consequence of the continuity of the flow, we will argue below that numerical tests corroborate Theorem \ref{MT} and the ``long time stability'' nature of the approximate PINNs solutions.

\subsection{Numerical tests. Preliminaries}

We conducted numerical tests to validate Theorem \ref{MT} and Corollary \ref{Cor1}, specifically focusing on well-known NLS solitary waves. All the codes are available in \url{https://github.com/nvalenzuelaf/DNN-Pinns-NLS} and the simulations were carried out using Python on a 64-bit MacBook Pro M2 (2022) with 8GB of RAM.

For these tests, we selected the integrable case $\alpha=3$, where many explicit soliton solutions are available. In this scenario, the chosen exponents are $(p,q) = (8,4)$, with the H\"older conjugates being $p'=\frac{8}{7}$ and $q'=\frac{4}{3}$. Our goal was to numerically approximate the solution $u_{\text{DNN},\#}$ using an algorithm implemented in Python. The full details of the algorithm are as follows:
 \begin{enumerate}
 \item[(i)] We start with a DNN $u_{\text{DNN},\#} := u_{\text{DNN},\#,\theta}$ having $H$ hidden layers, $n_H$ neurons per layer and the classical sine activation function acting component-wise {\color{black} except by the last hidden layer where the activation is taken as tanh}. The weights and biases will be initialized in a standard way, using the Glorot uniform initializer. They will be summarized in $\theta$ as the optimization variables.
 \item[(ii)] To optimize the DNN parameters, we will use the LBFGS method (see, e.g. \cite{Fle13,KoWh19}), which in our case performed better than the standard Adam method,  with 3000 iterations.
 \item[(iii)] For $N_4,M_4,N_5,M_5$ as in (H3) we choose the points $(x_{4,j})_{j=1}^{N_4}$, $(x_{5,j})_{j=1}^{N_5} \subset [-R,R]$ and $(t_{4,\ell})_{\ell=1}^{M_4}$, $(t_{5,\ell})_{\ell=1}^{M_5}\subset [-T,T]$, to obtain two grids of $M_4\times N_4$ and $M_5\times N_5$ points $(t_{4,\ell},x_{4,j})_{\ell,j=1}^{M_4,N_4}$ and $(t_{5,\ell}, x_{5,j})_{\ell,j=1}^{M_5,N_5}$, respectively.
 
{\color{black}In the numerical simulations, $(x_{4,j})_{j=1}^{N_4}, (x_{5,j})_{j=1}^{N_5}$, $(t_{4,\ell})_{\ell=1}^{M_4}$ and $(t_{5,\ell})_{\ell=1}^{M_5}$ will be taken as uniformly partitions over $[-R,R]$ and $[-T,T]$. The values of $M_4,M_5,N_4,N_5$ will be declared in each example.}
 \item[(iv)] We will consider the loss function as the functions defined in (H3): First, for simplicity define
 \[
 v(t,\cdot) \equiv e^{it\partial_x^2} (u_0-u_{\text{DNN},\#}(0)).
 \]
 As said before, this linear evolution can be computed using an approximation on the corresponding Fourier transform. Denote as $\hat v(t_\ell,x_j)$ the approximation of $v$ at the point $(t_\ell,x_j)$. Then the loss function takes the form:
 \begin{equation}\label{eq:loss}
 \begin{aligned}
 \hbox{Loss}(\theta) &:=  {\color{black}\gamma} \mathcal J_{p',q',N_4,M_4} [\mathcal E[u_{\text{DNN},\#}]] + \mathcal J_{p,q,N_5,M_5} [\hat v ],
 \end{aligned}
 \end{equation}
 where {\color{black}$\gamma$ is properly chosen as 3 to balance the terms on the loss function and} we recall $\mathcal E[u_{\text{DNN},\#}]:= i\partial_t u_{\text{DNN},\#} + \partial_x^2 u_{\text{DNN},\#} + |u_{\text{DNN},\#}|^{\alpha-1}u_{\text{DNN},\#}$.
 \item[(v)] In each example we aim to obtain numeric values for the bounds involved in (H2). For this, along this section we consider in each iteration step the quantities
 \[
\widetilde A := \mathcal J_{H^1,N_1}[(u_0-u_{\text{DNN},\#}(0))],
\]
\[
A:= \mathcal J_{\infty,H^1,N_2,M_2}[u_{\text{DNN},\#}], \qquad B := \mathcal J_{p,q,N_3,M_3}[u_{\text{DNN},\#}],
\]
where $u_{\text{DNN},\#}$ is the respective DNN of each iteration in the LBFGS algorithm and $N_1=N_2=N_3=N_5$ and $M_2=M_3=M_5$.
 \item[(vi)] Consider $N_\text{test} = M_\text{test} = 100$. We will use as test data a grid of $N_\text{test} \times M_\text{test}$ points $(t_{\ell},x_j)_{\ell,j=1}^{M_\text{test},N_\text{test}}$ generated in the same way as the training points. If $u = u(t,x)$ is the exact solution of the NLS equation, we will compute the following errors
 \begin{align}\label{Sprime-error}
 \hbox{error}_{\mathcal{C}'} := &~{} \max_{(p,q)} \mathcal J_{p,q,N_\text{test},M_\text{test}}[u_{\text{DNN},\#}-u],\\
 \hbox{error}_{L^p_tW^{1,q}_x} := &~{}  \mathcal J_{p,q,N_\text{test},M_\text{test}}[u_{\text{DNN},\#}-u], \label{LpW1q-error}\\
 \hbox{error}_{L^{\infty}_tH^1} := &~{} \mathcal J_{\infty, H^1,N_\text{test},M_\text{test}}[u_{\text{DNN},\#}-u], \label{LinfH1-error}
 \end{align}
 where in \eqref{Sprime-error} we take the maximum over a suitable grid of values of $q$ that makes the pair $(p,q)$ Schr\"odinger admissible and $(p,q) = (8,4)$ in \eqref{LpW1q-error}. 
 
 {\color{black}For sake of completeness, in the following computations the $L_{t,x}^2$ error is also considered, which is a standard error in numerical analysis, defined as
 \begin{equation}\label{L2-error}
 \hbox{error}_{L_{t,x}^2} := ~{} \left(\frac{2T}{M_{\text{test}}} \sum_{\ell=1}^{M_{\text{test}}} \frac{2R}{N_{\text{test}}} \sum_{j=1}^{N_{\text{test}}} |u_{\text{DNN},\#}(t_\ell,x_j)-u(t_\ell,x_j)|^2\right)^{1/2},
 \end{equation}
 where we recall that $(t_{\ell})_{\ell=1}^{M_{\text{test}}} \subset [-T,T]$ and $(x_j)_{j=1}^{N_{\text{test}}} \subset [-R,R]$.
 }
 \end{enumerate}
% {\color{red}

 \begin{remark}[On the approximation of $v(t,\cdot)$] Recall that $v(t,\cdot)$ is given by
 \[
\mathcal F_{\xi \to x}^{-1} \left( e^{-i t\xi^2} \mathcal F_{x\to \xi}(u_0-u_{\text{DNN},\#}(0))(\xi)\right).
\]
We will approximate $v(t_{\ell},x_{j})$ by using the \url{fft} package in Pytorch, that uses the fast Fourier transform (FFT) to obtain a numerical approximation of the (continuous) Fourier transform $\mathcal F_{x \to \xi}$ and its inverse $\mathcal F_{\xi \to x}^{-1}$. Thus the value of $\hat v(t_\ell,x_j)$ can be obtained naturally by the FFT approximation of the Fourier transform and its inverse in the definition of $v(t_\ell,x_j)$.

 \end{remark}
 
 {\color{black} \begin{remark}[On the approximative norms] \label{rem:loss} For the loss function \eqref{eq:loss} we choose all the weights equal to 1, mimicking the root mean squared error (RMSE). However, to accurately approximate the norms of each error term, we consider the size of the domain over the time-space region $[-T,T]\times[-R,R]$ in the computation of $\widetilde{A}$, $A$, $B$, and the errors error$_{\mathcal C'}$, error$_{L_t^p W_x^{1,q}}$, and error$_{L_t^\infty H^1}$ defined in \eqref{Sprime-error}-\eqref{LinfH1-error}. 
 
 More precisely, we consider the weights in order to approximate integrals as a classical (right) Riemann sum: For $f$ defined in $[a,b]$,
 \[
 \int_a^b f(x) dx \approx \frac{1}{N} \sum_{j=1}^N (b-a) f(x_j),
 \]
 where $x_j = j \frac{b-a}N + a$. This can be generalized in order to approximate $\|f\|_{L^p_t L^q_x}.$

In the loss function, the derivative terms are weighted by a factor of $10^{-3}$ to balance the contributions of the different components and to mitigate the error introduced by the derivative approximation via the FFT. Thus, for example, the term $ \mathcal J_{p,q,N_5,M_5} [\hat v ]$ included in the loss function will take the form
\[
 \mathcal J_{p,q,N_5,M_5} [\hat v ] =~{} \left( \frac1{M_5}\sum_{\ell=1}^{M_5} \left( \frac1{N_5} \sum_{j=1}^{N_5} \left(|\hat v(t_{5,\ell},x_{5,j})|^q+ 10^{-3} \cdot |\partial_x \hat v(t_{5,\ell},x_{5,j})|^q \right) \right)^{p/q} \right)^{1/p},
\]
and similarly for $ \mathcal J_{p',q',N_4,M_4} [\mathcal E[u_{\text{DNN},\#}]] $.
 \end{remark}

 \begin{remark}[On the activation function] The sine activation function is chosen in this work due to the oscillate behavior on the real and imaginary parts of the considered examples. However, it also induces periodicity in the approximate solution, which may lead to inaccuracies near the boundaries, as the model tries to replicate the solution profile outside the space domain of interest. 

In our simulations, we observed that replacing the activation function of the last hidden layer with a hyperbolic tangent significantly mitigates this artificial periodicity, and keeps the oscillate behavior. 
 \end{remark}
}
\subsection{List of solitary waves}
 
Following \cite{AFM}, let $c, \nu > 0$ be fixed scaling parameters. The considered solutions will be: the solitary wave
\begin{equation}\label{solitonQ}
Q(t,x) = e^{i\frac \nu2 x - i\left(\frac {\nu^2}4 - c\right) t} \sqrt{2c} \hspace{.05cm}\hbox{sech}\left(\sqrt{c}(x-\nu t)\right).
\end{equation}
(notice that if $\nu= 0$, then the solitary-wave solution remains to a standing wave solution), {\color{black} the 2-Solitonic solution \cite{Hir}
\begin{equation}\label{2solitonQ}
Q_2(t,x) = \sqrt{2} \frac{G(t,x)}{F(t,x)},
\end{equation}
where 
\[
\begin{aligned}
F(t,x) &= 1 + \frac{e^{\eta_1 + \overline \eta_1}}{(\lambda_1 + \overline \lambda_1)^2} + \frac{e^{\eta_1 + \overline \eta_2}}{(\lambda_1 + \overline \lambda_2)^2}+ \frac{e^{\eta_2 + \overline \eta_1}}{(\lambda_2 + \overline \lambda_1)^2}+ \frac{e^{\eta_2 + \overline \eta_2}}{(\lambda_2 + \overline \lambda_2)^2} \\
&\hspace{.5cm}+ \frac{(\lambda_1 - \lambda_2)^2 (\overline \lambda_1 - \overline \lambda_2)^2 e^{\eta_1 + \eta_2 + \overline \eta_1 + \overline \eta_2}}{\left((\lambda_1 + \overline \lambda_1)(\lambda_1 + \overline \lambda_2)(\lambda_2 + \overline \lambda_1)(\lambda_2 + \overline \lambda_2)\right)^2},
\end{aligned}
\]
and
\[
G(t,x) = e^{\eta_1} + e^{\eta_2} + \frac{(\lambda_1 - \lambda_2)^2 e^{\eta_1 + \eta_2 + \overline \eta_1}}{\left((\lambda_1 + \overline \lambda_1)(\lambda_2 + \overline \lambda_1)\right)^2} + \frac{(\lambda_1 - \lambda_2)^2 e^{\eta_1 + \eta_2 + \overline \eta_2}}{\left((\lambda_1 + \overline \lambda_2)(\lambda_2 + \overline \lambda_2)\right)^2},
\]
where $\eta_j = \lambda_j x + i \lambda_j^2 t - \eta_j^0$. $\lambda_j$ and $\eta_j^0$ are complex constants describing the amplitude and phase of the $j$-soliton.} The Peregrine breather \cite{Per}
\begin{equation}\label{BP}
B_\text{P}(t,x) := e^{it}\left(1 - \frac{4(1+2it)}{1+4t^2+2x^2}\right),
\end{equation}
and the Kuznetsov-Ma (KM)  breather \cite{Kuz, Ma}
\begin{equation}\label{BKM}
B_\text{KM}(t,x) := e^{it}\left(1 - \sqrt 2 \tilde\beta \frac{\tilde\beta^2 \cos(\tilde\alpha t) + i \tilde\alpha \sin(\tilde\alpha t)}{\tilde\alpha \cosh(\tilde\beta x) - \sqrt 2 \tilde\beta \cos(\tilde\alpha t)}\right),
\end{equation}
with
\[
\tilde \alpha := (8a(2a-1))^{\frac 12}, \hspace{1cm} \tilde \beta:=(2(2a-1))^{\frac 12}, \quad a > \frac 12,
\]
Notice that the KM breather does not converge to zero at infinity, but after the subtraction of the background, it becomes a localized solution. The stability of the KM breather (and other breathers) has been studied in \cite{Mu17,AFM19,AFM}. 

We have chosen $N=100$ and $M=100$ in every error computation, representing a grid of $10^4$ space-time evaluation points. The training known data will be the exact soliton or solitary wave solution {\color{black}at time $t=0$, represented in the second term in the right hand side of \eqref{eq:loss}. All the other points used for training are not known data itself, but they are used by the knowledge in the physics of the NLS model.}%evaluated at $32\times 32$ points for each member in \eqref{eq:loss}.

\subsection{Solitons} We present now our first results for the case of the soliton $Q$ in \eqref{solitonQ}.  We choose parameters $c = \nu = 1$ but we ensure that our results are very similar for other similar values of $c$ and $\nu$. The space region is {\color{black}$[-10,10]$ with $N_4=N_5=32$ points}, the time region {\color{black}$[-2,2]$ with $M_4=M_5=32$}, {\color{black} and we take $H=4$ hidden layers with $n_H=20$ neurons each}. Our results are summarized in Fig.  \ref{fig:1}, where the continuous line represents the exact solution and the dashed line is the solution computed using the proposed PINNs minimization procedure. In particular, Fig. \ref{fig:1-1} and \ref{fig:1-2} present respectively the real and imaginary parts of the computed soliton solution for three different times.  

\begin{figure}[!ht]
\centering
\begin{subfigure}[t]{0.45\textwidth}
	\centering
	\includegraphics[width=\textwidth]{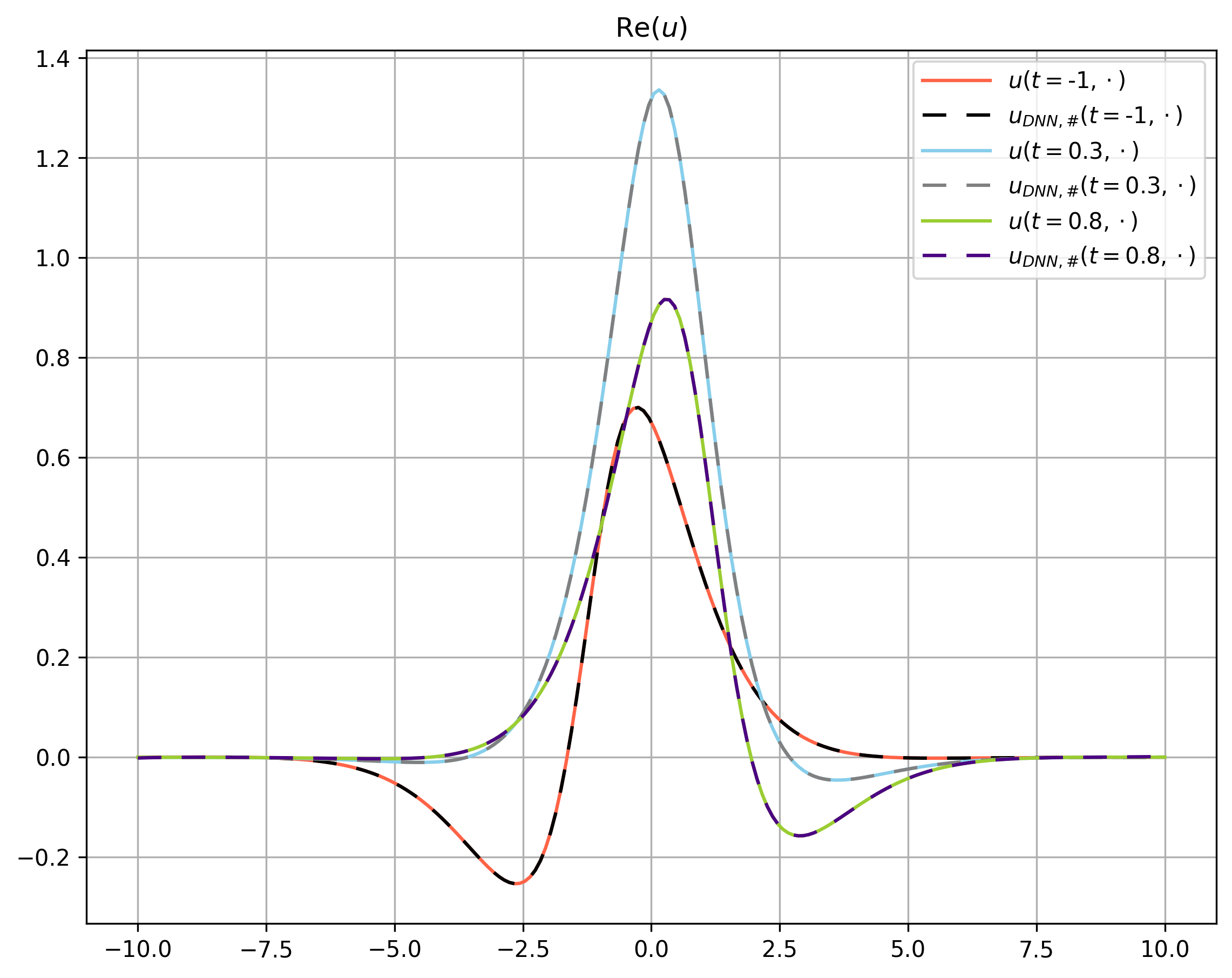}
	\caption{Real part of the exact (continuous line) and approximate (dashed line) soliton solution at times $t=$ -2, 0.5 and 1.5.}
	\label{fig:1-1}
\end{subfigure}
\hspace{.2cm}
\begin{subfigure}[t]{0.45\textwidth}
	\centering
	\includegraphics[width=\textwidth]{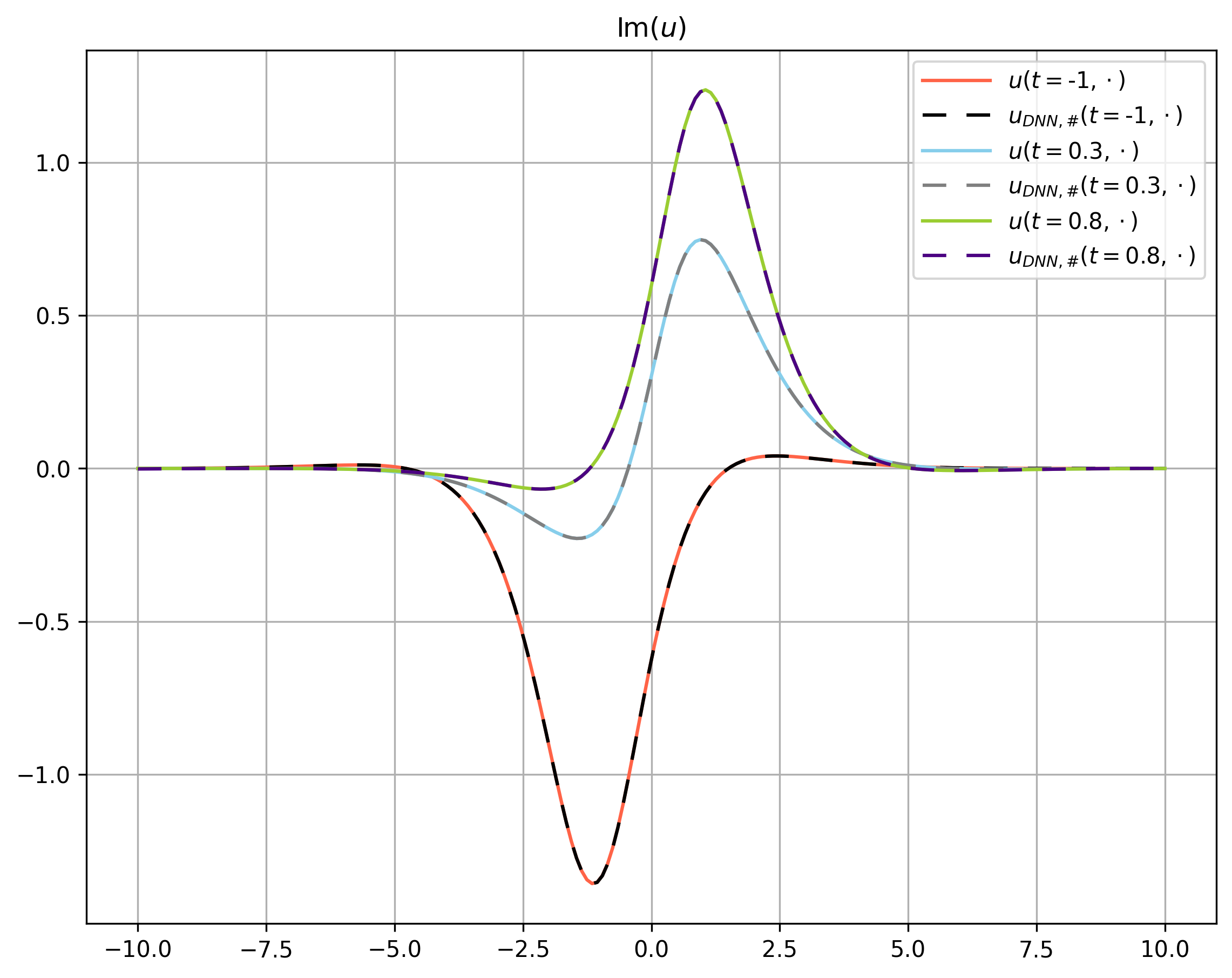}
	\caption{Imaginary part of the exact (continuous line) and approximate (dashed line) soliton solution at times $t=$ -2, 0.5 and 1.5.}
	\label{fig:1-2}
\end{subfigure}
\caption{Approximation of the soliton solution in the case $c=\nu=1$.}
\label{fig:1}
\end{figure}

In Fig. \ref{fig:2}, the evolution of the computed constants $\widetilde A$, $A$, $B$, and the error of the approximation is presented in terms of the number of iterations of the numerical algorithm for the soliton solution with values $c=\nu=1$. It is noticed in Fig. \ref{fig:2-1} that the constant $\widetilde A$, measuring the $H^1$ difference between the approximate solution at time zero and the initial data decreases in time, revealing that to get a better approximation for all posterior times, it is necessary to get very small errors at the initial time. Concerning the evolution of the computed values for the constants $A$ and $B$, graphed in Figs. \ref{fig:2-2} and \ref{fig:2-3}, they stabilize at some $O(1)$ value during a big part of the numerical procedure. The error function \eqref{eq:loss}, graphed in Fig. \ref{fig:2-4}, naturally decreases through computations to achieve an $O(10^{-3})$ value.

\begin{figure}[!ht]
\centering
\begin{subfigure}[t]{0.4\textwidth}
	\centering
	\includegraphics[width=\textwidth]{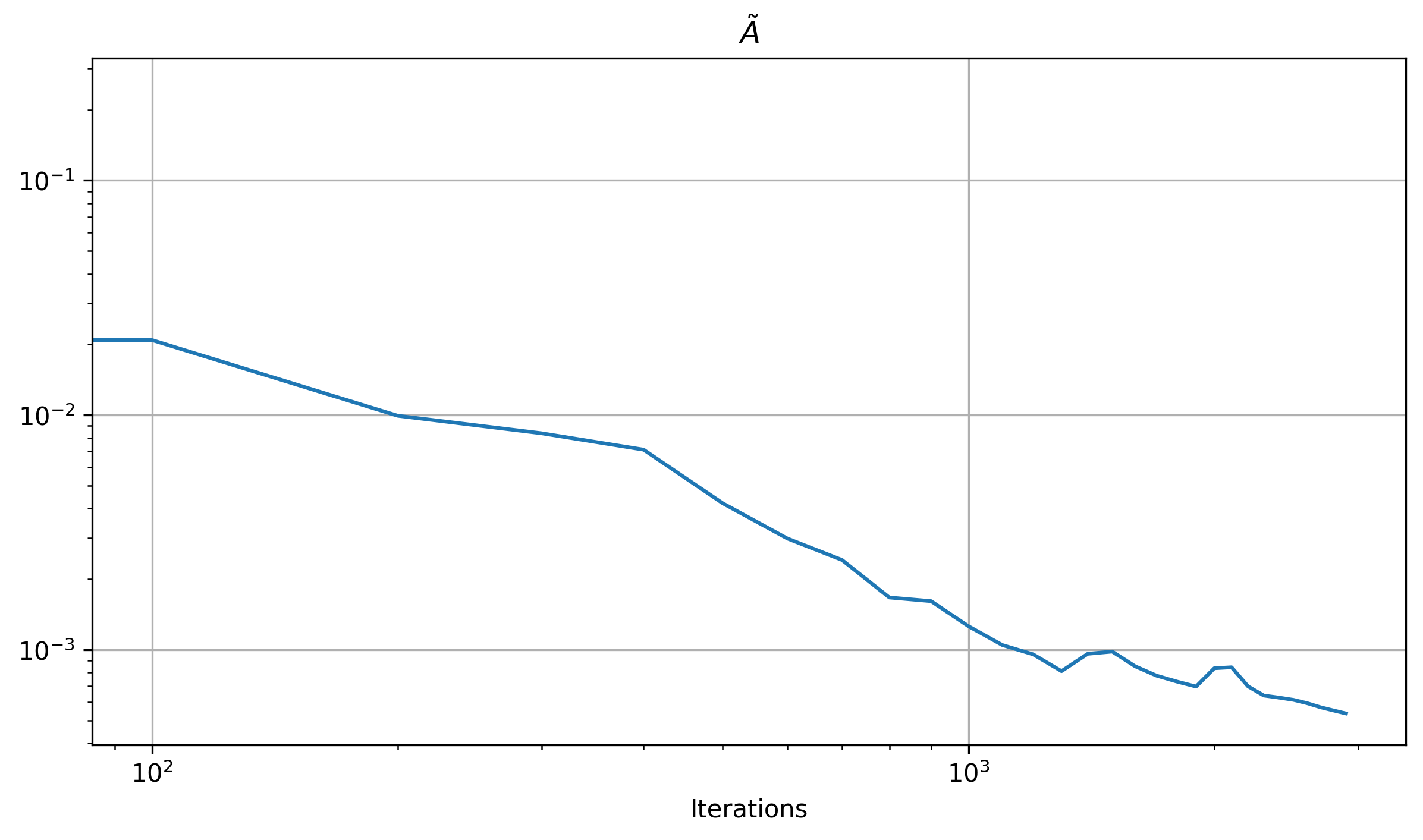}
	\caption{Evolution of the computed constant $\widetilde A$.}
	\label{fig:2-1}
\end{subfigure}
\hspace{.2cm}
\begin{subfigure}[t]{0.4\textwidth}
	\centering
	\includegraphics[width=\textwidth]{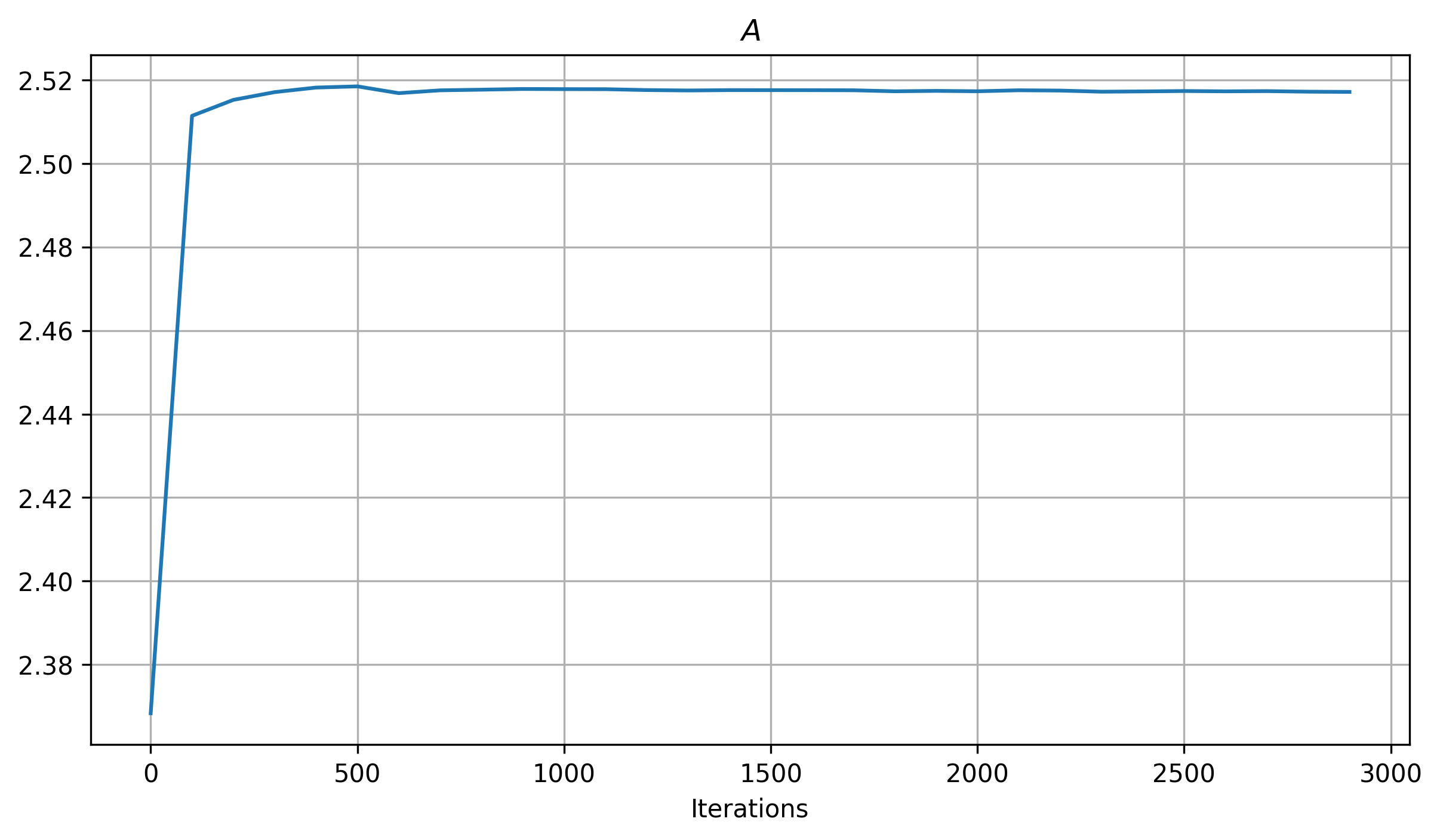}
	\caption{Evolution of the computed constant $A$.}
	\label{fig:2-2}
\end{subfigure}
\begin{subfigure}[t]{0.4\textwidth}
	\centering
	\includegraphics[width=\textwidth]{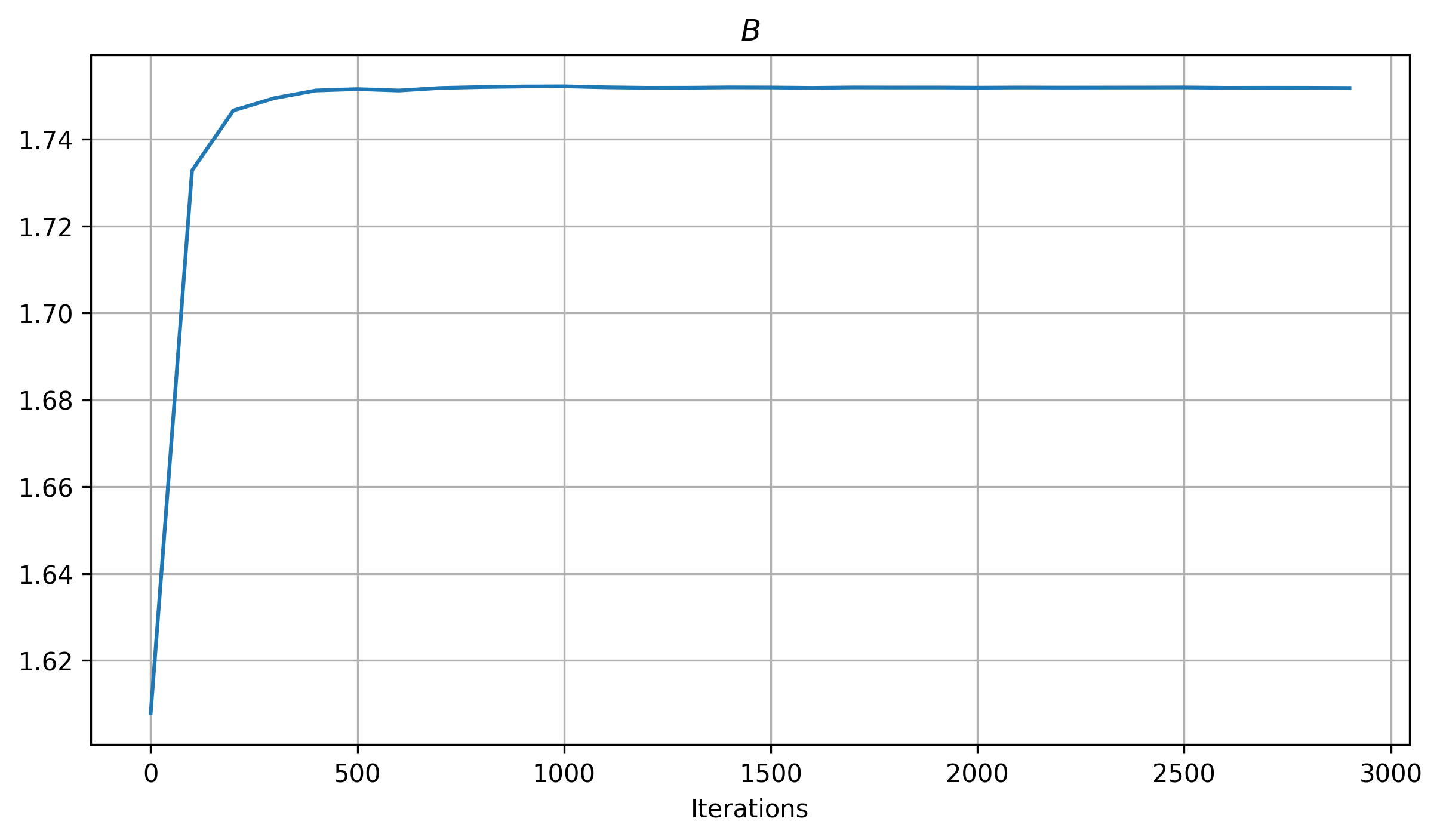}
	\caption{Evolution of the computed constant $B$.}
	\label{fig:2-3}
\end{subfigure}
\hspace{.2cm}
\begin{subfigure}[t]{0.4\textwidth}
	\centering
	\includegraphics[width=\textwidth]{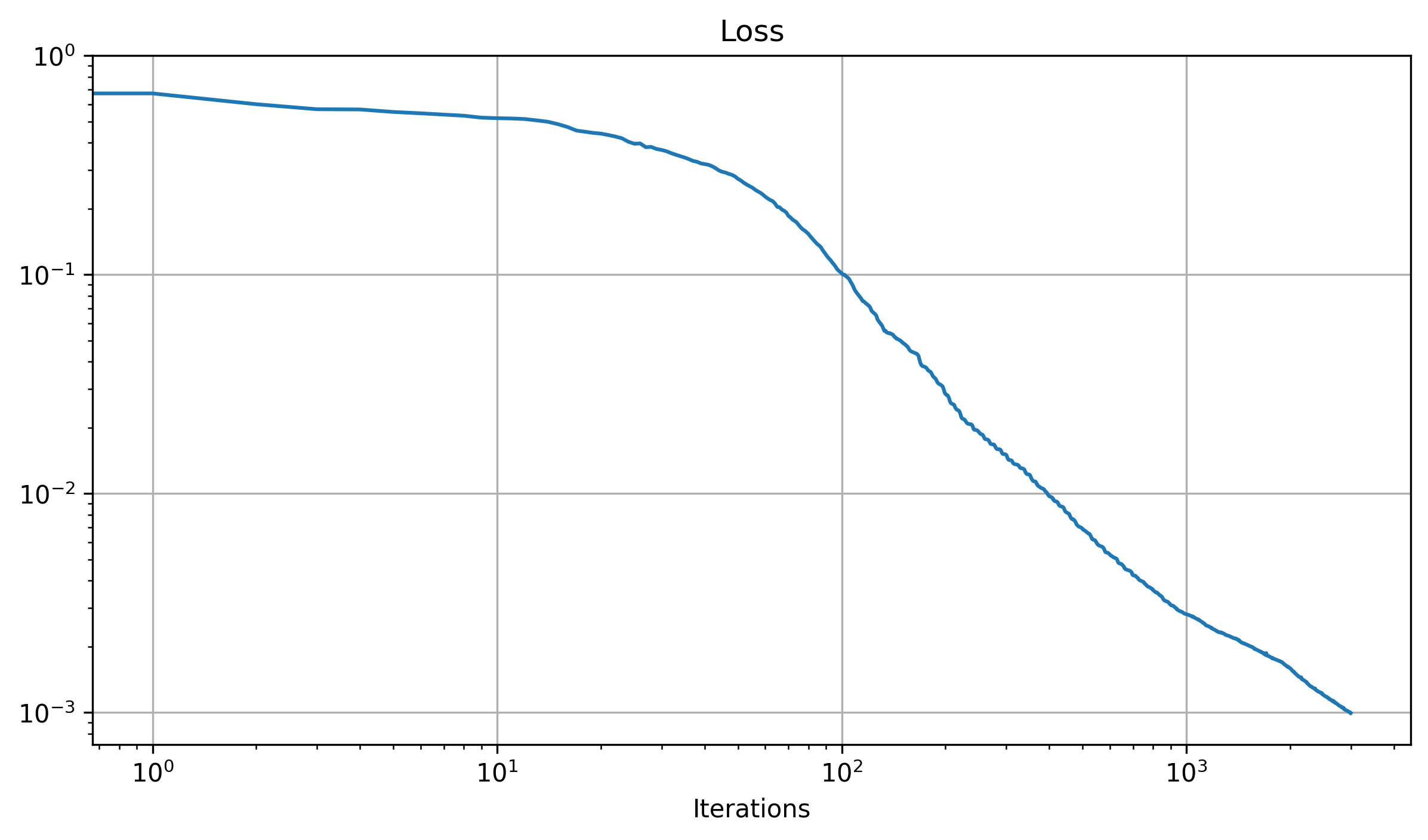}
	\caption{Error of the numerical approximation.}
	\label{fig:2-4}
\end{subfigure}
\caption{Evolution of the constants $\widetilde A, A$ and $B$ and the error \eqref{eq:loss} during the optimization algorithm.}
\label{fig:2}
\end{figure}

%Fig. \ref{fig:3} summarizes the evolution of the  associated norms of the difference between the exact and the approximate solutions in the case of the soliton with values $c=\nu=1$. The time interval in these cases is $[-2,2]$, and the space interval is $[-8,8]$. These norms are computed using \eqref{Sprime-error},\eqref{LpW1q-error} and \eqref{LinfH1-error} with uniform weights equal 1, and uniform $N_\text{test}$ and $M_\text{test}$ given by 100. As a conclusion, from Figures \ref{fig:3-1}, \ref{fig:3-2} and \ref{fig:3-3} we deduce that, during the numerical computations all the norms involved in Theorem \ref{MT} decreased to a value of order $10^{-2}$, with particular jumps that were immediately solved after new iterations. 
Table \ref{tab:iters} summarizes the elapsed time and the evolution through the iterations of the  associated norms of the difference between the exact and the approximate solutions in the case of the soliton with values $c=\nu=1$. The time interval in these cases is $[-2,2]$, and the space interval is $[-8,8]$. These norms are computed using \eqref{Sprime-error},\eqref{LpW1q-error} and \eqref{LinfH1-error} with uniform weights equal 1, and uniform $N_\text{test}$ and $M_\text{test}$ given by 100. As a conclusion, from Table \ref{tab:iters} we deduce that, during the numerical computations the norms $L^{t}_{\infty}H^1$ and $\mathcal C'$ involved in Theorem \ref{MT} decreased to a value of order {\color{black}$10^{-2}$, while the norm $L^{p}_t W^{1,q}_x$ to order $10^{-3}$}, and {\color{black}the algorithm} takes half a minute. 

\begin{table}[!ht]
\begin{tabular}{lllll}
\hline
Iterations & Elapsed time [s]  & error$_{L^{p}_tW^{1,q}_x}$ & error$_{L^{\infty}_tH^1}$ & error$_{\mathcal{C}'}$  \\ \hline
\hline
100        & {\color{black}1.169}               & {\color{black}1.080}                     & {\color{black}1.951}                     & {\color{black}1.951}                  \\
500        & {\color{black}5.569}           & {\color{black}0.125}                      & {\color{black}0.241}                     & {\color{black}0.241}                  \\
1000       & {\color{black}10.262}              & $3.507 \times 10^{-2}$                      & $6.813 \times 10^{-2}$                    & $6.813 \times 10^{-2}$                 \\
3000       & {\color{black}30.445}              & $8.761 \times 10^{-3}$    & $1.494 \times 10^{-2}$                     & $1.494 \times 10^{-2}$                   \\
\hline
\end{tabular}
\caption{{\color{black}Average} values of the elapsed time and norms involved in Theorem \ref{MT}, for five {\color{black}independent realizations of the algorithm,} in the solitonic case.}
\label{tab:iters}
\end{table}

%\begin{figure}[!ht]
%\centering
%\begin{subfigure}[t]{0.3\textwidth}
%	\centering
%	\includegraphics[width=\textwidth]{computations/Soliton/alpha1-beta1/LpW1q-error.png}
%	\caption{Evolution of the $L_t^pW_x^{1,q}$ norm of the difference exact-approximate solution during the algorithm.}
%	\label{fig:3-1}
%\end{subfigure}
%\hspace{.2cm}
%\begin{subfigure}[t]{0.3\textwidth}
%	\centering
%	\includegraphics[width=\textwidth]{computations/Soliton/alpha1-beta1/LinfH1-error.png}
%	\caption{Evolution of the $L_t^\infty H_x^1$ norm of the difference exact-approximate solution during the algorithm.}
%	\label{fig:3-2}
%\end{subfigure}
%\hspace{.2cm}
%\begin{subfigure}[t]{0.3\textwidth}
%	\centering
%	\includegraphics[width=\textwidth]{computations/Soliton/alpha1-beta1/Sprime-error.png}
%	\caption{Evolution of the $\mathcal S'$ norm of the difference exact-approximate solution during the algorithm.}
%	\label{fig:3-3}
%\end{subfigure}
%\caption{Evolution of the norms in Theorem \ref{MT} of the difference exact-approximate solution during the algorithm.}\label{fig:3}
%\end{figure}

%{\color{red}
Fig. \ref{fig:Estimation-Cprime} computes the $L_t^pW^{1,q}_x$ error \eqref{LpW1q-error} for {\color{black} one of} the trained DNNs $u_{\hbox{DNN},\#}$ for different admissible pairs $(p,q)$. The values of $q$ are chosen uniformly as a grid of {\color{black}597 points between 2 and 300, namely, $q = \frac k2$ for $k = 4,\ldots,600$}. {\color{black} Fig. \ref{fig:Estimation-Cprime} suggest that the maximum is attached in $q=2$, that is, the norm $\mathcal C'$ is governed by the norm $L^{\infty}_t H^1$. This can be also seen in each table presented in this work.  Additionally, after the maximum, the norms have a local minimum in $q=4$, which is indeed the admissible pair we are using to minimize the linear evolution of the difference. The point $q=4$ is followed by a broad range of $q$ where the error increases, until it eventually drops sharply to zero at a certain point.}

%}

\begin{figure}[!ht]
	\centering
	\includegraphics[width=0.5\textwidth]{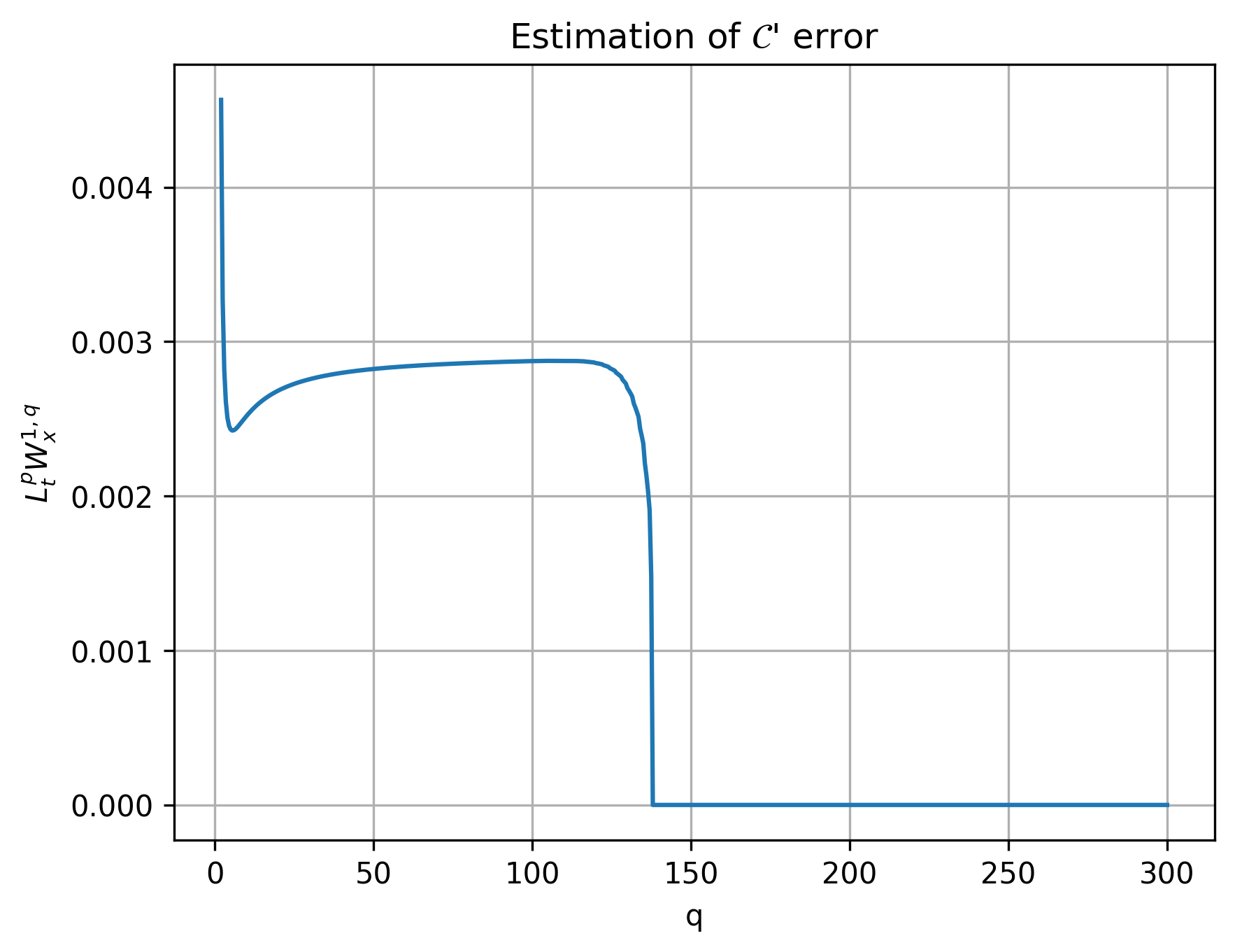}
	\caption{$L_t^pW_x^{1,q}$ error for different admisible pairs $(p,q)$ in the solitonic case. The x-axis represents $597$ points of $q$ uniformly chosen between 2 and 300,  and the y-axis represents the $L_t^pW_x^{1,q}$ error, where $p = \frac{4q}{q-2}$. The $\mathcal{C}'$ error will be computed as the maximum of the $L_t^pW_x^{1,q}$ error in the grid of $q$.}
	\label{fig:Estimation-Cprime}
\end{figure}

%\begin{table}[!ht]
%\begin{tabular}{lllllll}
%\hline
%Iterations & Elapsed time [s] & Loss                   &  $\widetilde A$         & error$_{L^{p}_tW^{1,q}_x}$ & error$_{L^{\infty}_tH^1}$ & error$_{\mathcal S'}$  \\ \hline
%\hline
%100        & 1.846                & 9.060 $\times 10^{-2}$                  & 2.702 $\times 10^{-2}$ & 0.434                     & 0.515                     & 0.668                  \\
%500        & 9.436               & 1.701 $\times 10^{-2}$ & 2.300 $\times 10^{-3}$ & 0.128                      & 0.162                     & 0.186                  \\
%1000       & 18.654               & 7.621 $\times 10^{-3}$ & 1.102 $\times 10^{-3}$ & 6.268 $\times 10^{-2}$                      & 8.161 $\times 10^{-2}$                   & 8.691 $\times 10^{-2}$                  \\
%3000       & 57.791               & 2.892 $\times 10^{-3}$ & 4.428 $\times 10^{-4}$ & 1.308 $\times 10^{-2}$     & 2.051 $\times 10^{-2}$                     & 2.051 $\times 10^{-2}$                   \\
%\hline
%\end{tabular}
%\caption{Computed values of the elapsed time, errors, and norms involved in Theorem \ref{MT}, for five different number of iterations in the solitonic case.}
%\label{tab:iters}
%\end{table}

{\color{black}\subsection{2-Soliton} 
Another solution which satisfies the hypothesis of the main Theorem is the 2-soliton. In this section we present the results for the solution presented in \eqref{2solitonQ}. For this end we choose $\lambda_1 = 1 + i$, $\lambda_2 = 1 - i$ and $\eta_1^0 = \eta_2^0 = 0$. The space region is $[-10,10]$ with $N_4=128$ and $N_5=64$, the time region $[-1,1]$ with $M_4=64$, and $M_5=32$, and we take $H=4$ with $n_H=40$ neurons each. The results are summarized in Fig. \ref{fig:2sol}, where the continuous line represents the exact solution and the dashed line is the solution computed using the proposed PINNs minimization procedure. In particular, Fig. \ref{fig:2sol-real} and \ref{fig:2sol-imag} present respectively the real and imaginary parts of the computed 2-soliton solution for three different times.

\begin{figure}[!ht]
\centering
\begin{subfigure}[t]{0.45\textwidth}
	\centering
	\includegraphics[width=\textwidth]{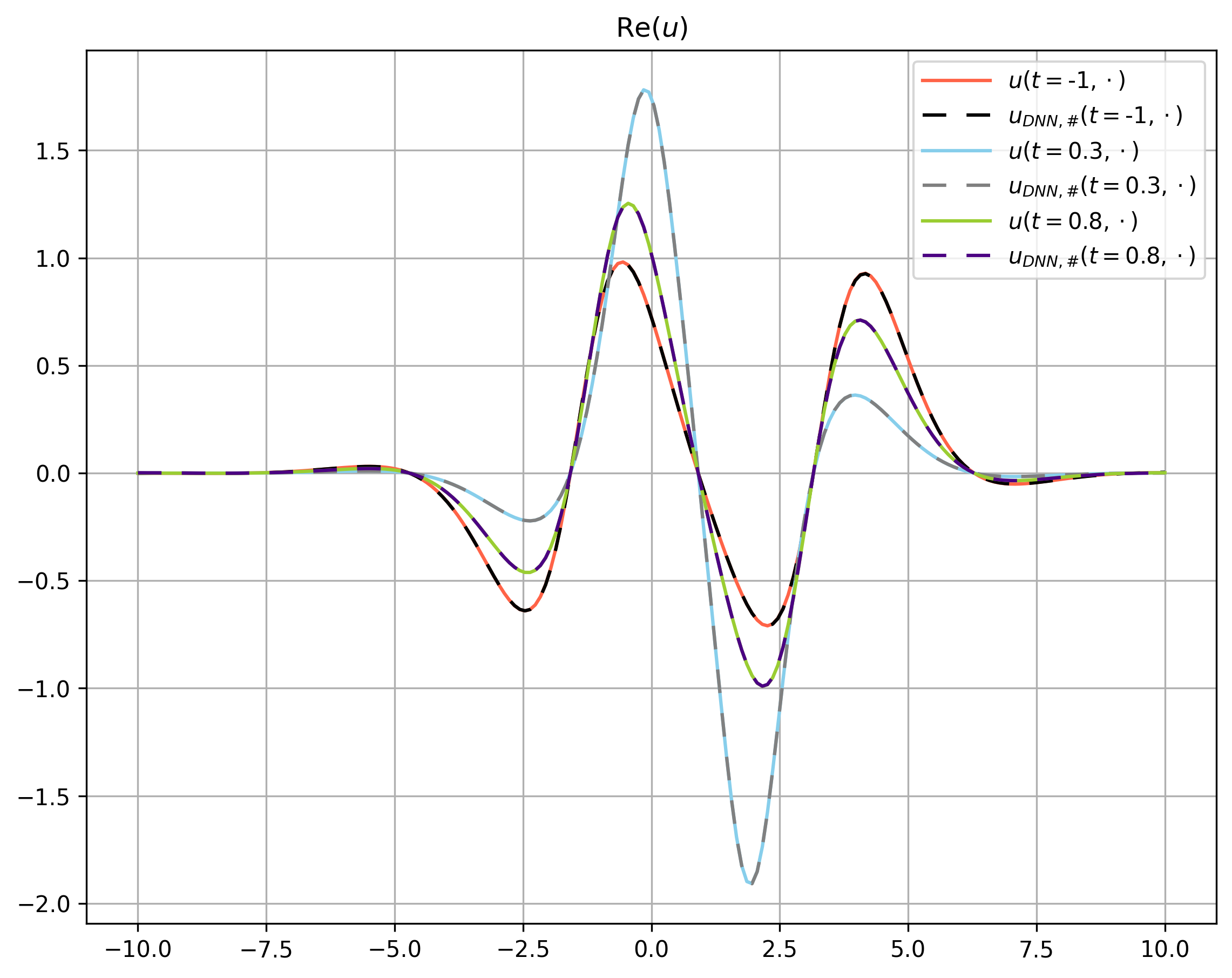}
	\caption{Real part of the exact (continuous line) and approximate (dashed line) 2-soliton solution at times $t=$ -1, 0.3 and 0.8.}
	\label{fig:2sol-real}
\end{subfigure}
\hspace{.2cm}
\begin{subfigure}[t]{0.45\textwidth}
	\centering
	\includegraphics[width=\textwidth]{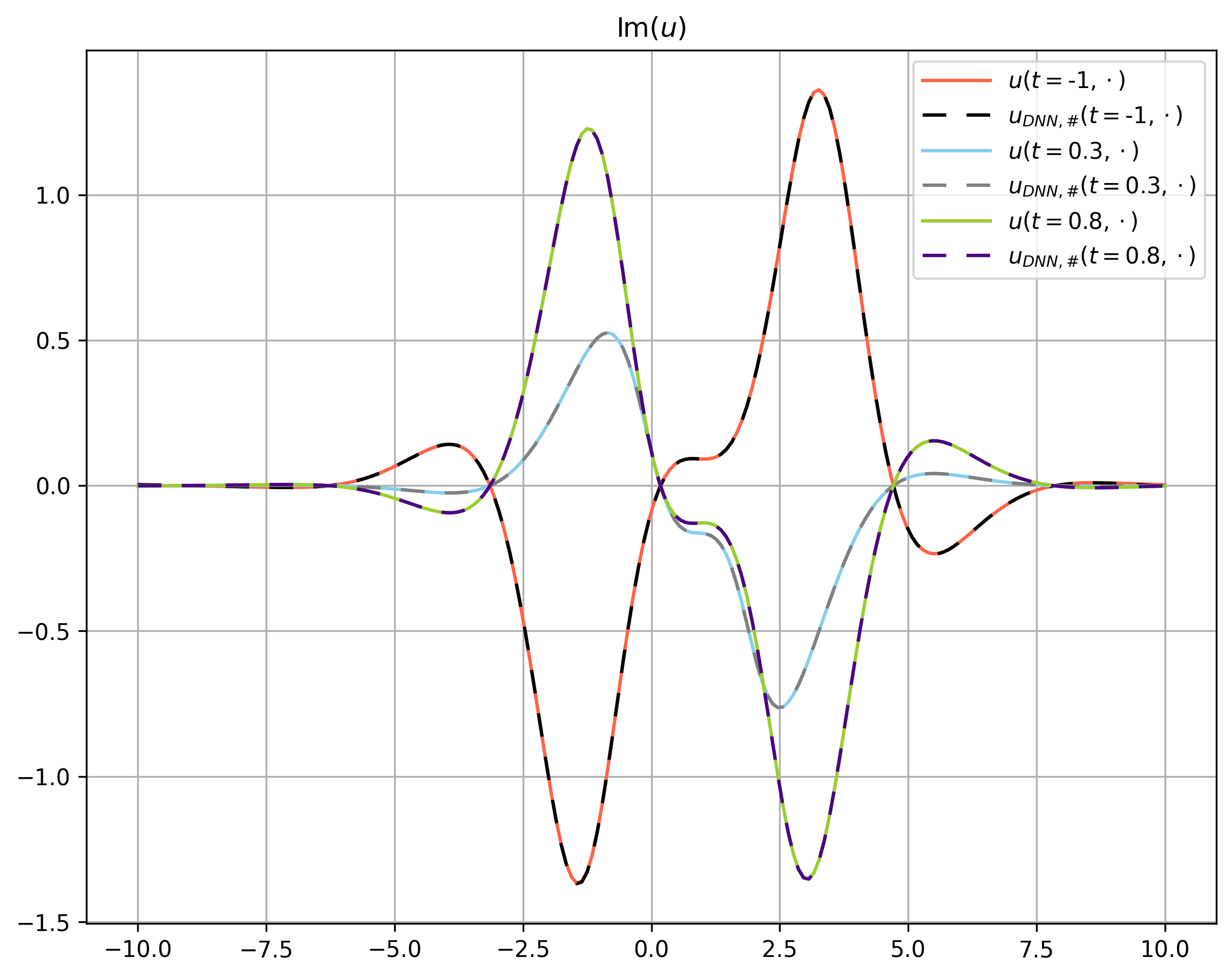}
	\caption{Imaginary part of the exact (continuous line) and approximate (dashed line) 2-soliton solution at times $t=$ -1, 0.3 and 0.8.}
	\label{fig:2sol-imag}
\end{subfigure}
\caption{Approximation of the 2-soliton solution.}
\label{fig:2sol}
\end{figure}

Table \ref{tab:2sol} summarizes the elapsed time and the evolution through the iterations of the associated norms of the difference between the exact and the approximate solutions in the case of the 2-soliton with $\lambda_1$, $\lambda_2$, $\eta_0^1$ and $\eta_0^2$ as defined as above. These norms are computed using \eqref{Sprime-error},\eqref{LpW1q-error} and \eqref{LinfH1-error} with uniform weights equal 1, and uniform $N_\text{test}$ and $M_\text{test}$ given by 100. As a conclusion, from Table \ref{tab:2sol} we deduce that, during the numerical computations all the norms involved in Theorem \ref{MT} decreased to a value of order {\color{black}$10^{-2}$}, and {\color{black}the algorithm} takes less than three minutes. 

\begin{table}[!ht]
\begin{tabular}{lllll}
\hline
Iterations & Elapsed time [s]  & error$_{L^{p}_tW^{1,q}_x}$ & error$_{L^{\infty}_tH^1}$ & error$_{\mathcal{C}'}$  \\ \hline
\hline
100        & {\color{black}6.314}               & {\color{black}2.524}                     & {\color{black}4.605}                     & {\color{black}4.605}                  \\
500        & {\color{black}30.145}           & {\color{black}0.223}                      & {\color{black}0.483}                     & {\color{black}0.483}                  \\
1000       & {\color{black}61.821}              & $4.909 \times 10^{-2}$                    & $9.838 \times 10^{-2}$                   & $9.838 \times 10^{-2}$                 \\
3000       & {\color{black}153.827}              & $2.183 \times 10^{-2}$    & $3.682 \times 10^{-2}$                    & $3.682 \times 10^{-2}$                   \\
\hline
\end{tabular}
\caption{{\color{black}Average} values of the elapsed time and norms involved in Theorem \ref{MT}, for five {\color{black}independent realizations of the algorithm,} in the 2-solitonic case.}
\label{tab:2sol}
\end{table}

In Fig. \ref{fig:2sol-2}, the evolution of the computed constants $\widetilde A$, $A$, $B$, and the error of the approximation is presented in terms of the number of iterations of the numerical algorithm for the KM solution. In this case  Figs. \ref{fig:2sol-2-2}-\ref{fig:2sol-2-3} are in agreement with the description already found in the soliton case, and Figs. \ref{fig:2sol-2-1} and \ref{fig:2sol-2-4} are not as good as the solitonic case, but they achieve an $O(10^{-2})$ value.
\begin{figure}[!ht]
\centering
\begin{subfigure}[t]{0.4\textwidth}
	\centering
	\includegraphics[width=\textwidth]{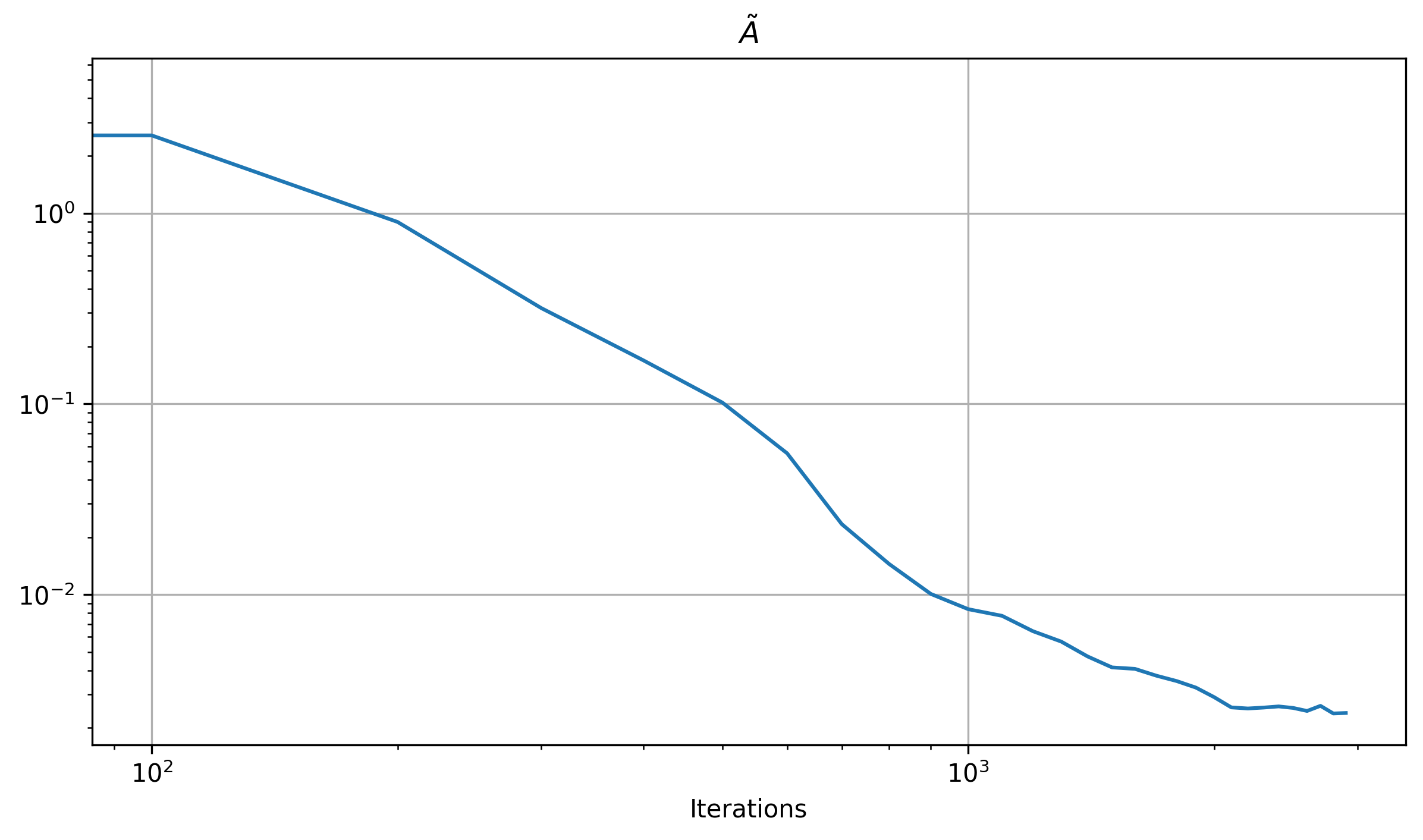}
	\caption{Evolution of the computed constant $\widetilde A$.}
	\label{fig:2sol-2-1}
\end{subfigure}
\hspace{.2cm}
\begin{subfigure}[t]{0.4\textwidth}
	\centering
	\includegraphics[width=\textwidth]{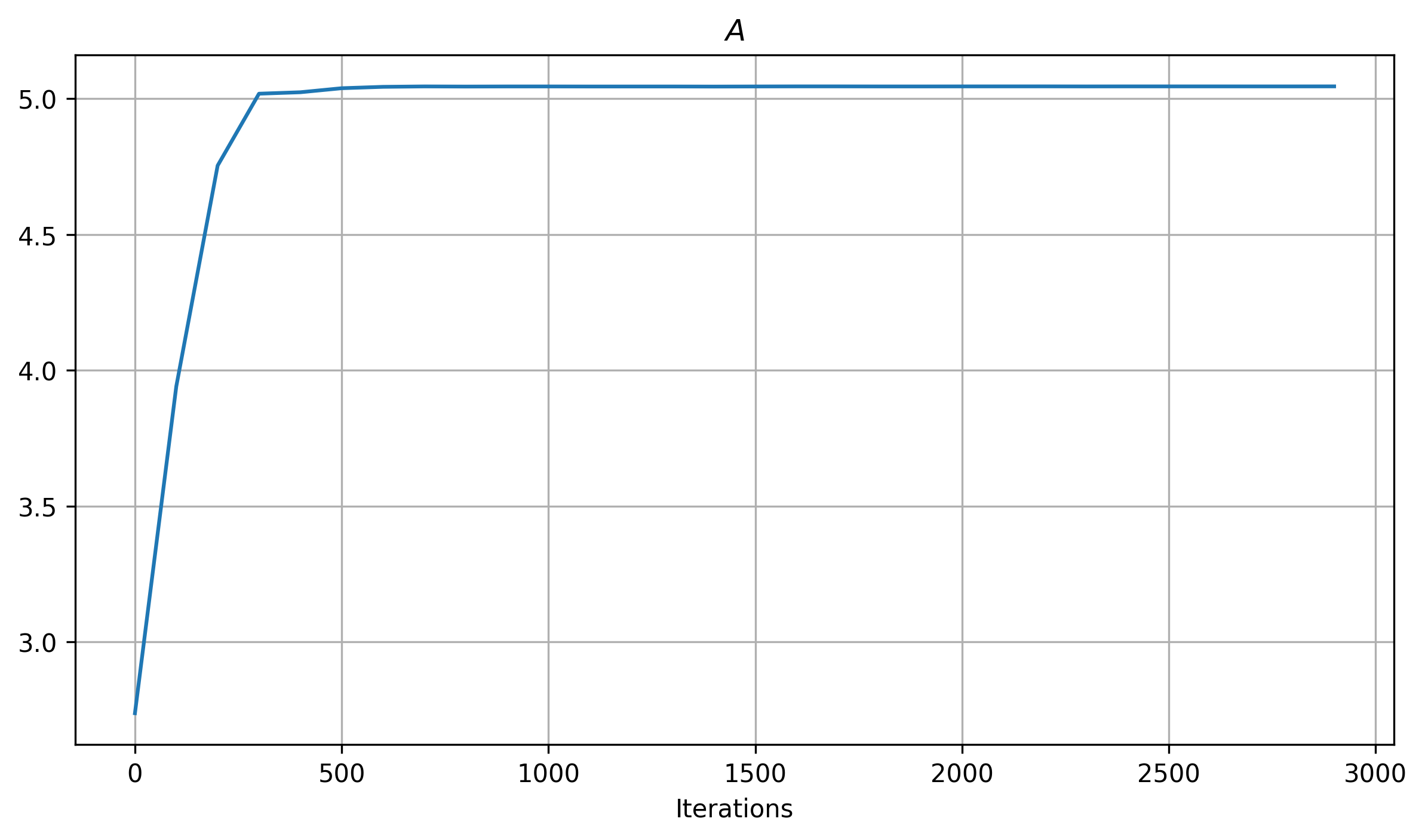}
	\caption{Evolution of the computed constant $A$.}
	\label{fig:2sol-2-2}
\end{subfigure}
\begin{subfigure}[t]{0.4\textwidth}
	\centering
	\includegraphics[width=\textwidth]{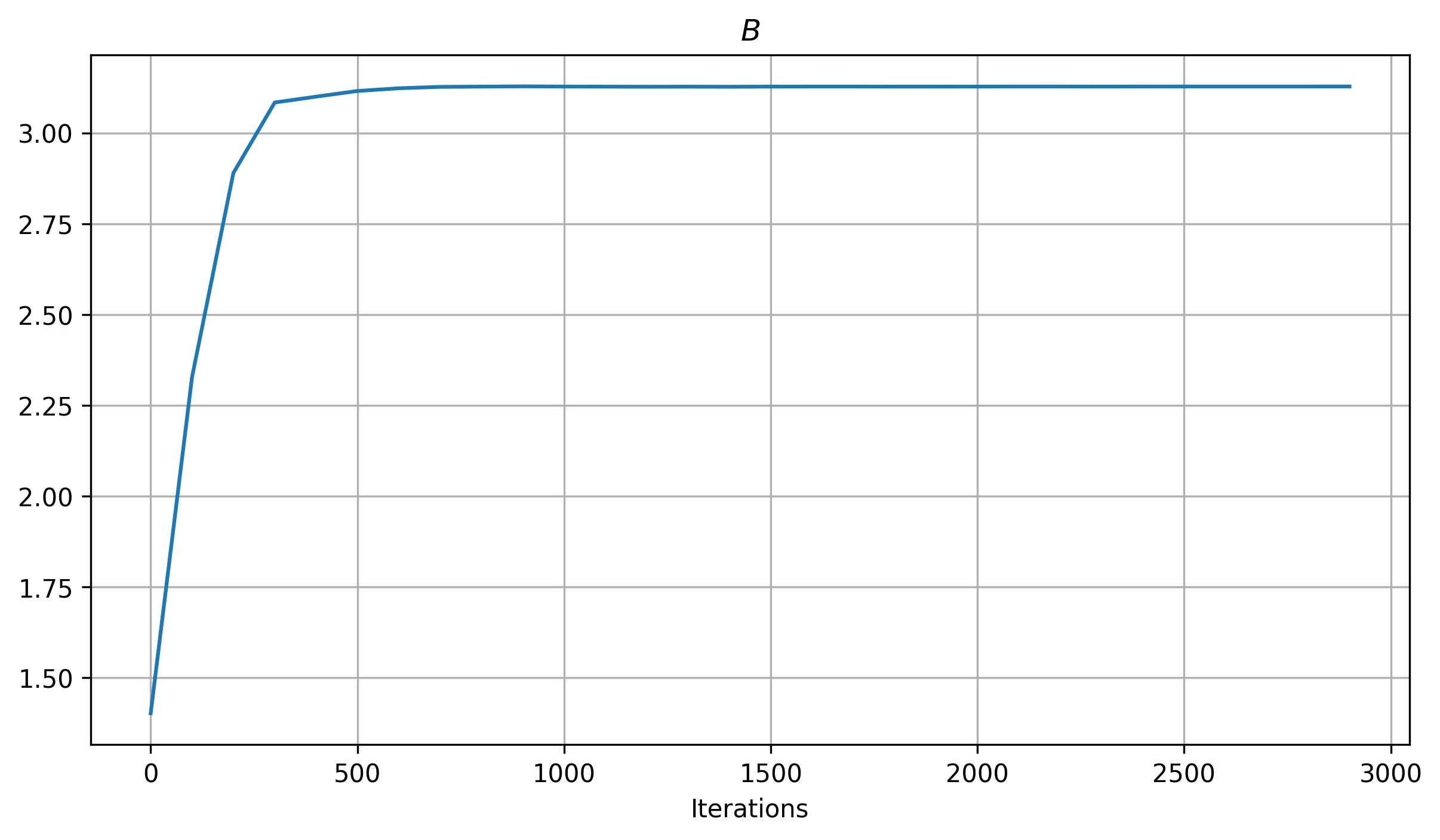}
	\caption{Evolution of the computed constant $B$.}
	\label{fig:2sol-2-3}
\end{subfigure}
\hspace{.2cm}
\begin{subfigure}[t]{0.4\textwidth}
	\centering
	\includegraphics[width=\textwidth]{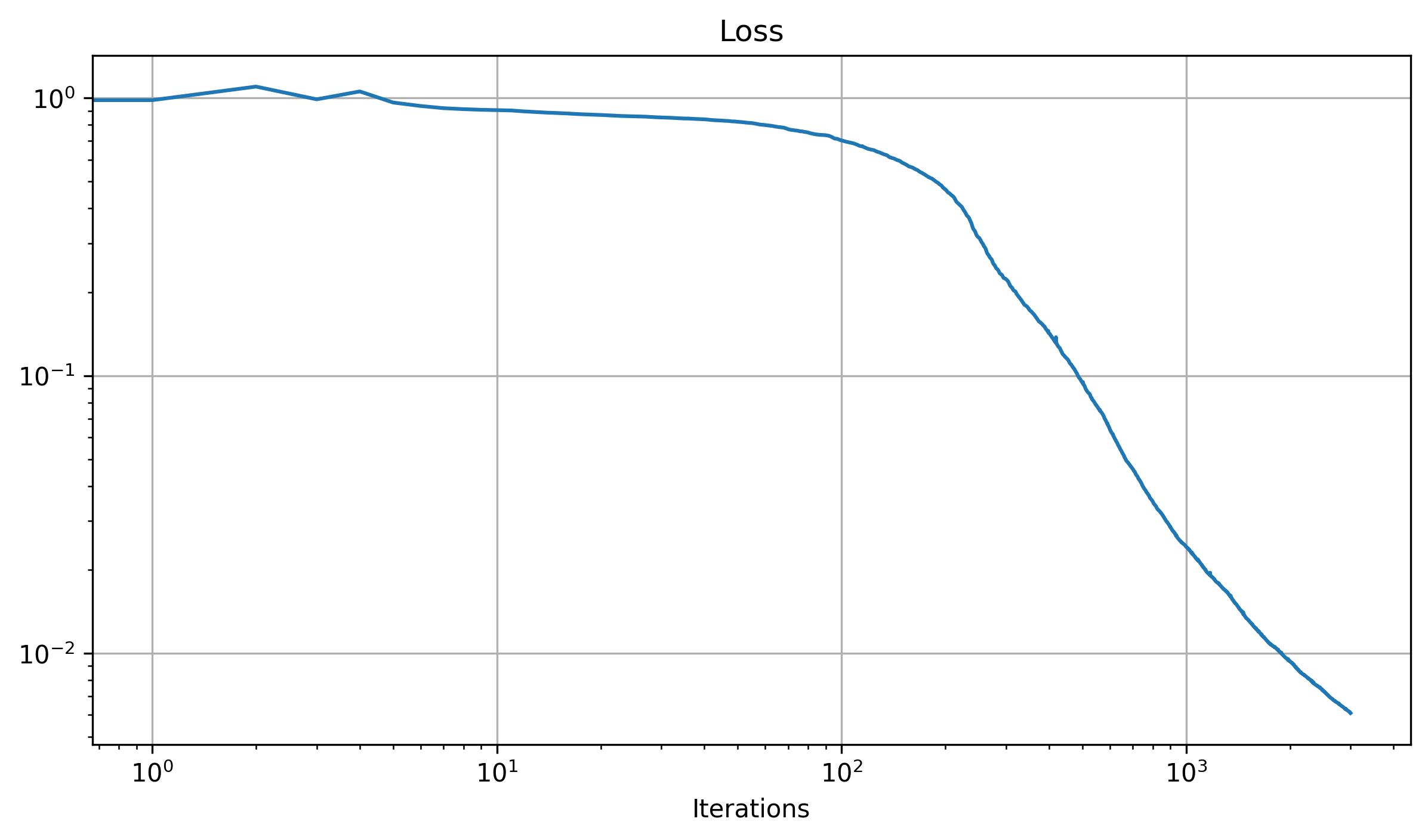}
	\caption{Error of the numerical approximation.}
	\label{fig:2sol-2-4}
\end{subfigure}
\caption{Evolution of the constants $\widetilde A, A$ and $B$ and the error \eqref{eq:loss} during the optimization algorithm for the 2-solitonic case.}
\label{fig:2sol-2}
\end{figure}

Finally, notice that the augmentation of points in the grids for the 2-soliton rather than the single soliton is due to the higher complexity on the solutions. This can be better appreciated in the absolute values of both solutions, as seen in Fig. \ref{fig:sol-2sol}. In particular Fig. \ref{fig:sol-2sol-1} shows the solitonic solution with $c=\nu=1$, while Fig. \ref{fig:sol-2sol-2} represents the 2-solitonic solution with parameters presented above.

\begin{figure}[!ht]
\centering
\begin{subfigure}[t]{0.45\textwidth}
	\centering
	\includegraphics[width=\textwidth]{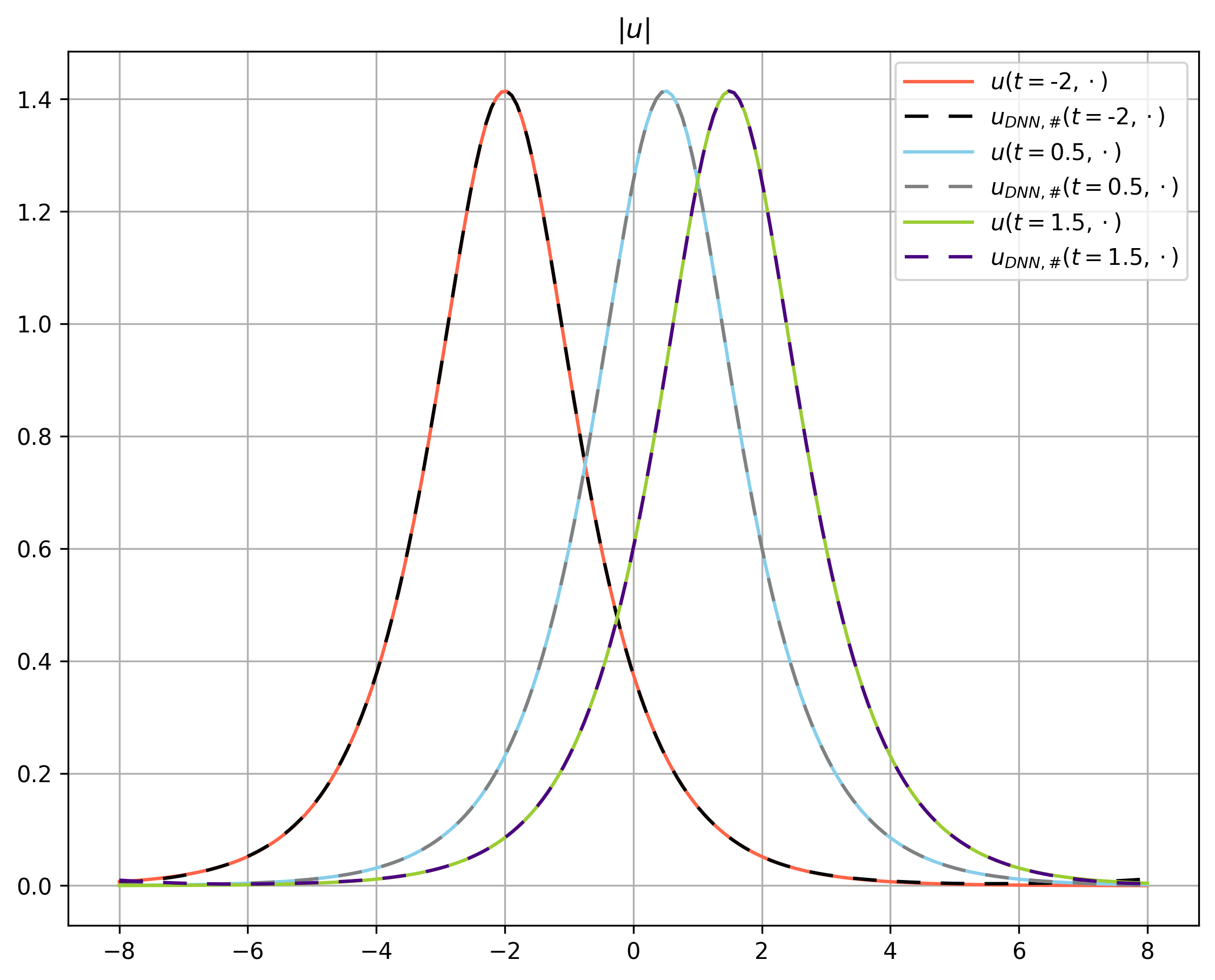}
	\caption{Absolute value of the exact (continuous line) and approximate (dashed line) soliton solution with $c=\nu=1$.}
	\label{fig:sol-2sol-1}
\end{subfigure}
\hspace{.2cm}
\begin{subfigure}[t]{0.45\textwidth}
	\centering
	\includegraphics[width=\textwidth]{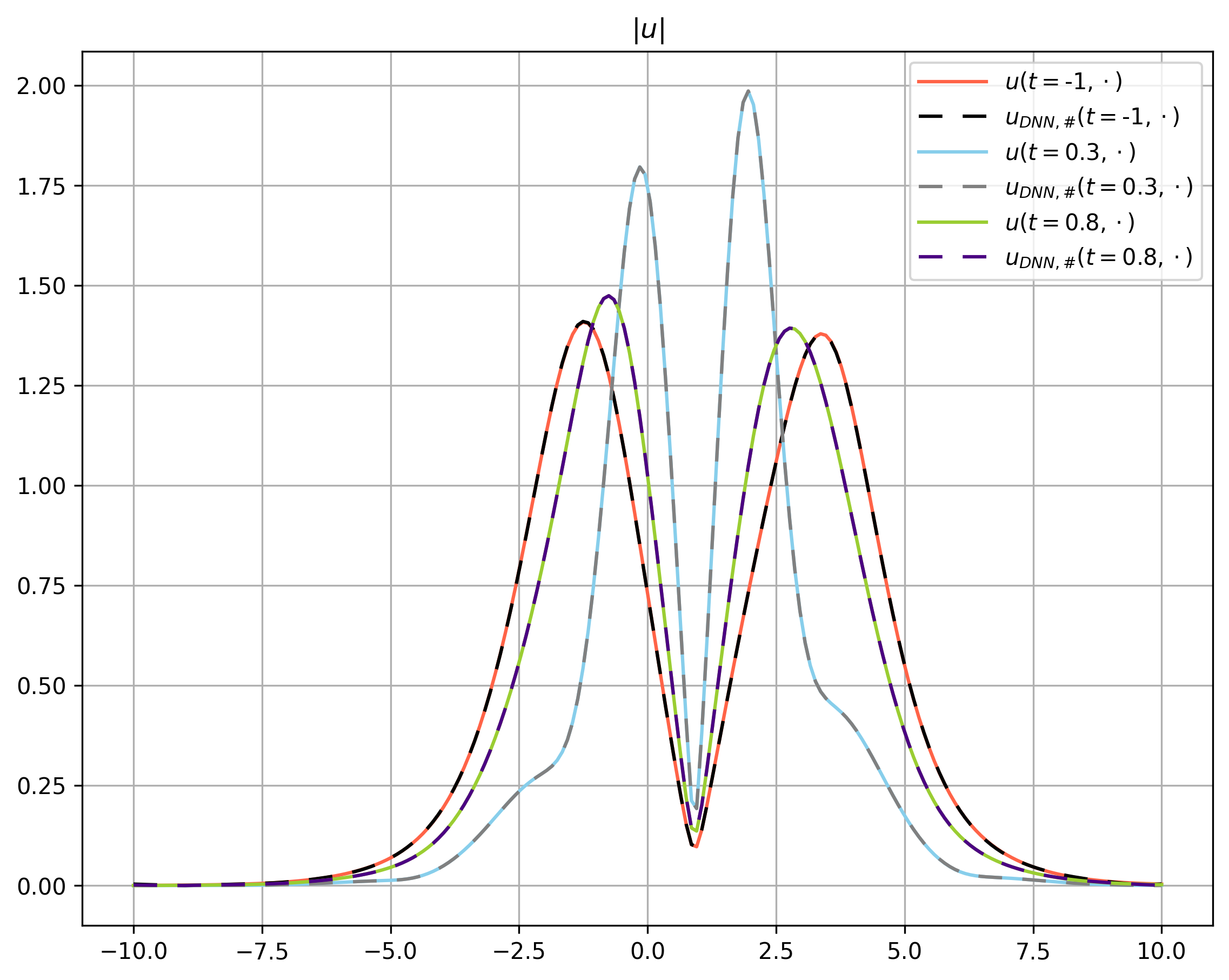}
	\caption{Absolute value of the exact (continuous line) and approximate (dashed line) 2-soliton solution.}
	\label{fig:sol-2sol-2}
\end{subfigure}
\caption{Absolute value of the solitonic and 2-solitonic solutions.}
\label{fig:sol-2sol}
\end{figure}

}

\subsection{Kuznetsov-Ma} 
Even if this breather solution is not in Sobolev spaces, one can perform a similar analysis since Theorem \ref{MT} only consider the difference of solutions. In the case of the KM breather $B_\text{KM}$ described in \eqref{BKM}, we have chosen the parameter $a= \frac34$. Notice that the larger is $a$, the more complicated are numerical simulations. {\color{black} In this case}, the space region is {\color{black}$[-5,5]$ with $N_4=128,$ and $N_5=64$ points}, the time region {\color{black} $[-1,1]$ with $M_4=M_5=32$}, {\color{black}and the hidden layers will be $H=4$ with $n_H=40$ neurons each}. Our results are summarized in Fig. \ref{fig:4}, where the continuous line represents the exact solution and the dashed line is the solution computed using the proposed PINNs minimization procedure. In particular, Fig. \ref{fig:4-1} and \ref{fig:4-2} present respectively the real and imaginary parts of the computed {\color{black}brather} solution for three different times.  %Although the approximation is good enough, we could not recover with high precision the imaginary part at large times, unless {\color{black}further iterations of the LBFGS algorithm are used. In order to obtain a high precision of the approximation of the imaginary part, we have performed this time 5000 iterations of the LBFGS algorithm. 

\begin{figure}[!ht]
\centering
\begin{subfigure}[t]{0.45\textwidth}
	\centering
	\includegraphics[width=\textwidth]{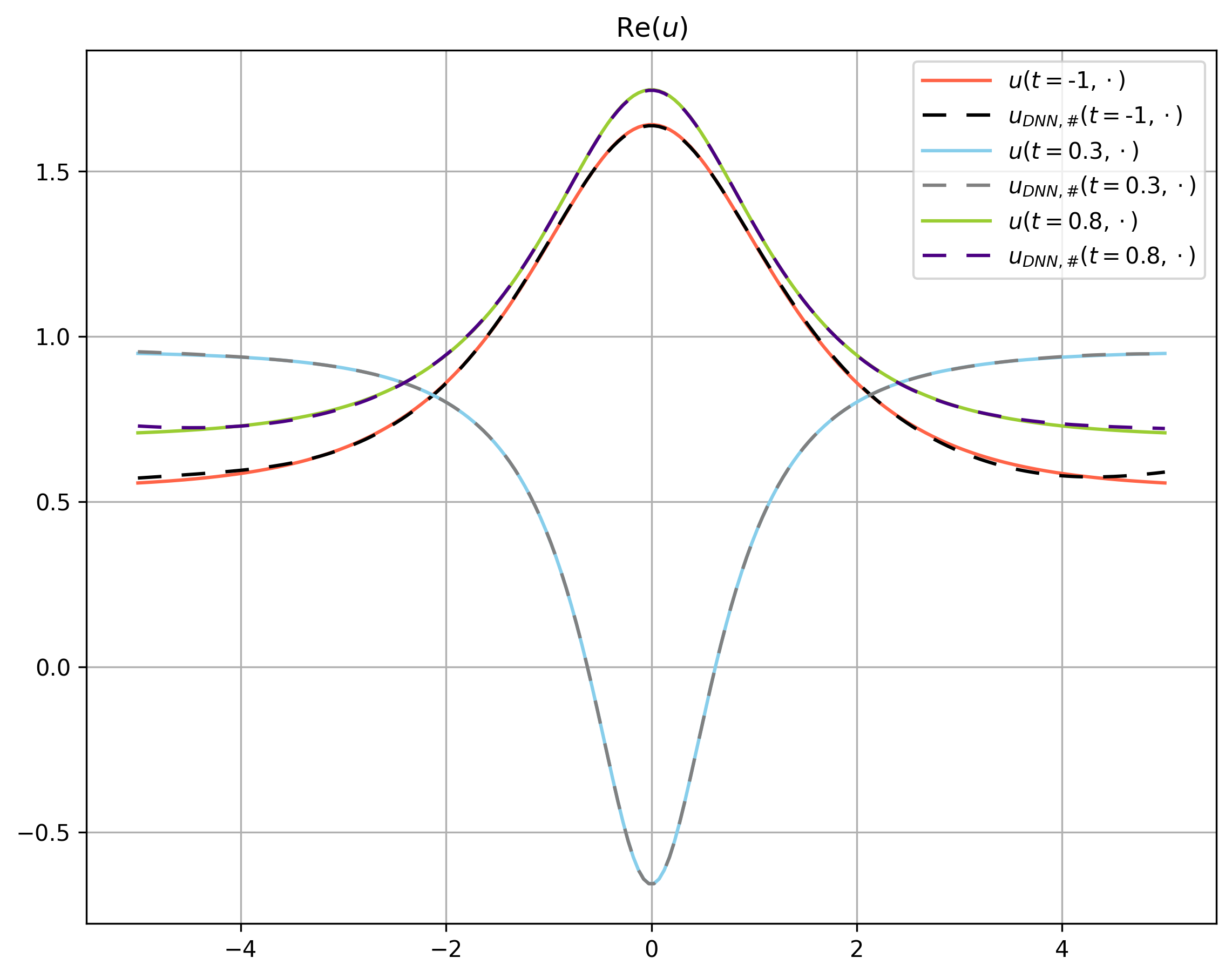}
	\caption{Real part of the exact (continuous line) and approximate (dashed line) KM breather at times $t=$ -1, 0.3 and 0.8.}
	\label{fig:4-1}
\end{subfigure}
\hspace{.2cm}
\begin{subfigure}[t]{0.45\textwidth}
	\centering
	\includegraphics[width=\textwidth]{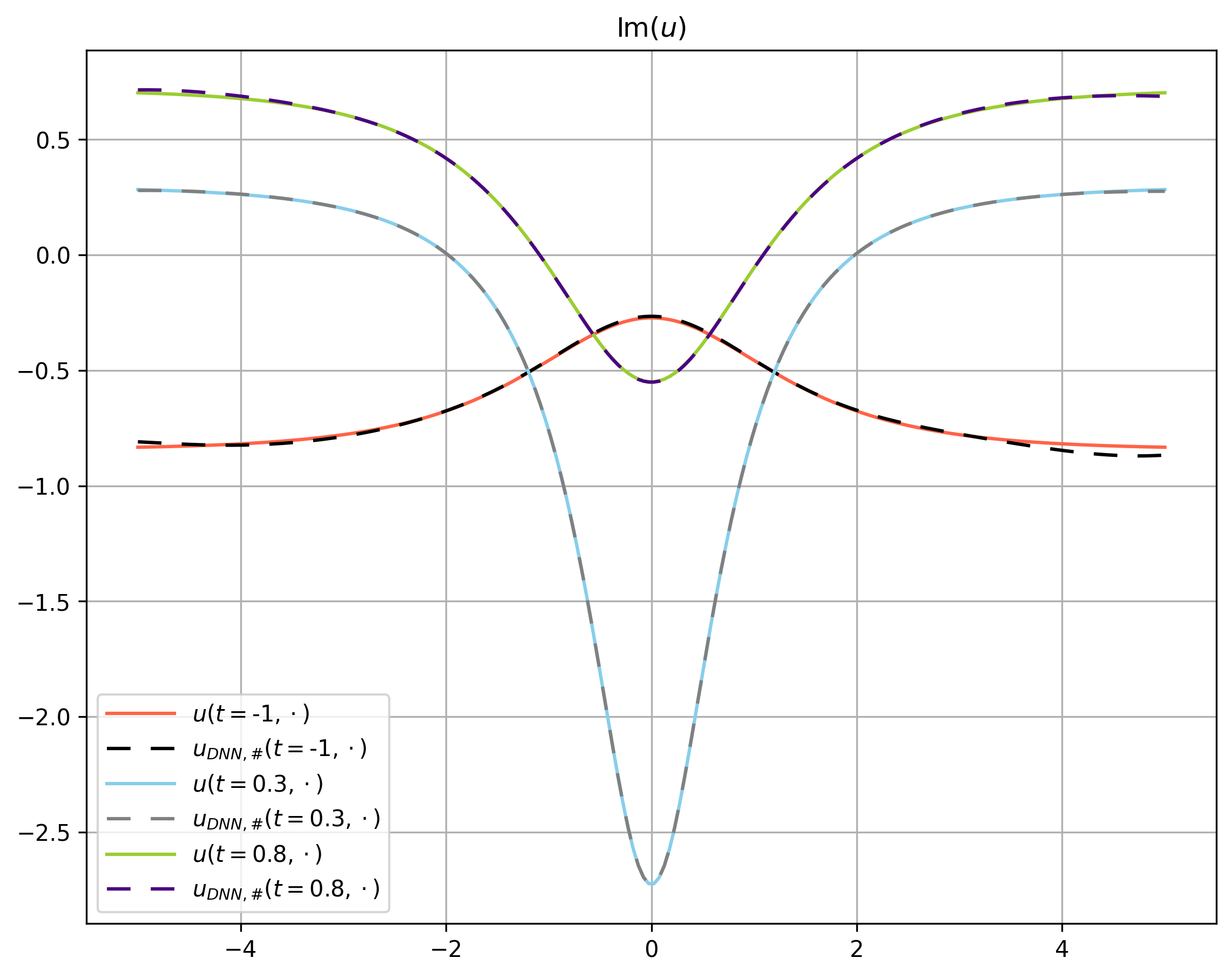}
	\caption{Imaginary part of the exact (continuous line) and approximate (dashed line) KM breather at times $t=$ -1, 0.3 and 0.8.}
	\label{fig:4-2}
\end{subfigure}
\caption{Exact (continuous line) and approximate (dashed line) KM breather.}
\label{fig:4}
\end{figure}

In Fig. \ref{fig:5}, the evolution of the computed constants $\widetilde A$, $A$, $B$, and the error of the approximation is presented in terms of the number of iterations of the numerical algorithm for the KM solution. In this case Figs. \ref{fig:5-1}-\ref{fig:5-2}-\ref{fig:5-3} are in agreement with the description already found in the previous cases, {\color{black}while Fig. \ref{fig:5-4} shows that the loss in the optimization algorithm has certain ``explosions'' during the training}.

\begin{figure}[!ht]
\centering
\begin{subfigure}[t]{0.4\textwidth}
	\centering
	\includegraphics[width=\textwidth]{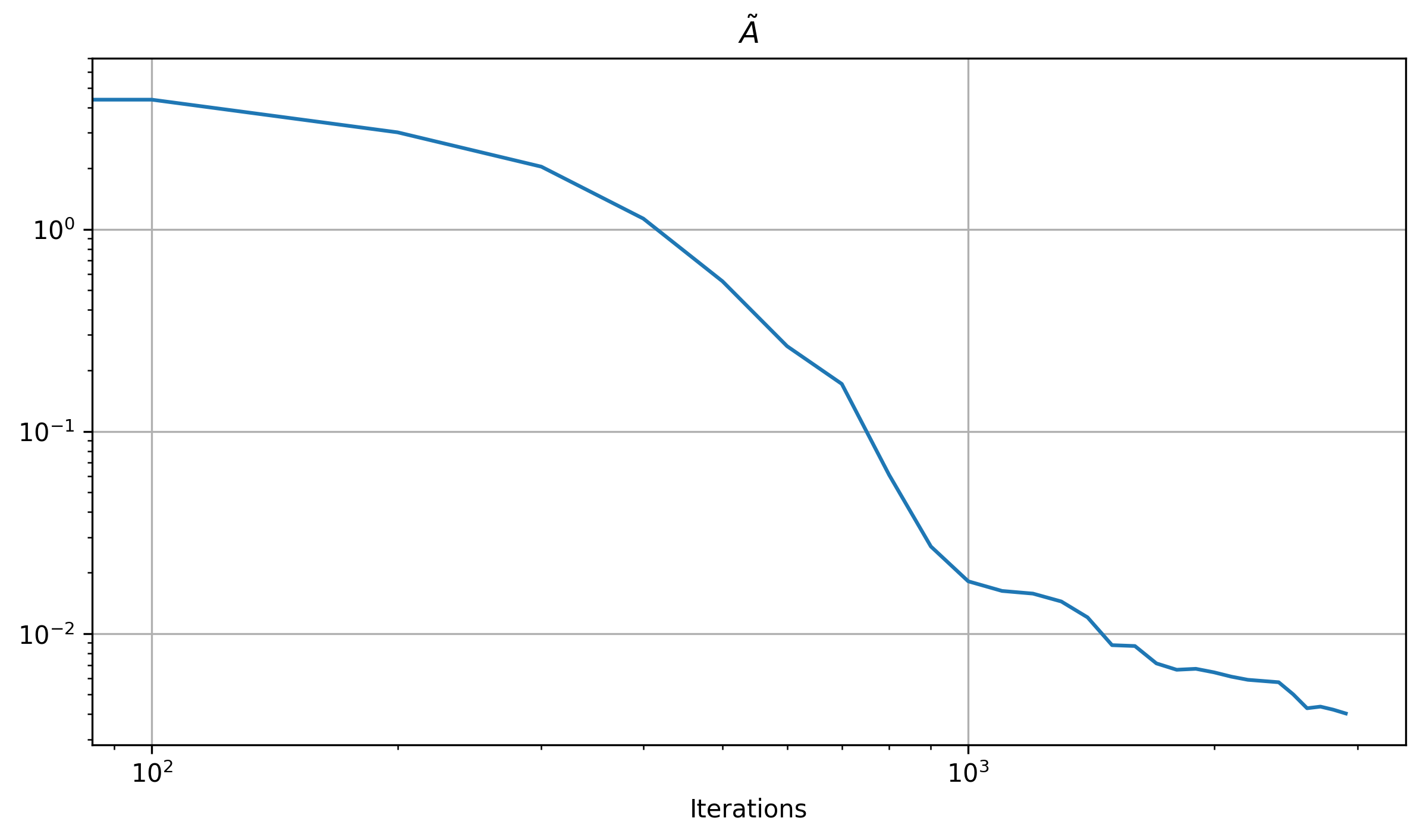}
	\caption{Evolution of the computed constant $\widetilde A$.}
	\label{fig:5-1}
\end{subfigure}
\hspace{.2cm}
\begin{subfigure}[t]{0.4\textwidth}
	\centering
	\includegraphics[width=\textwidth]{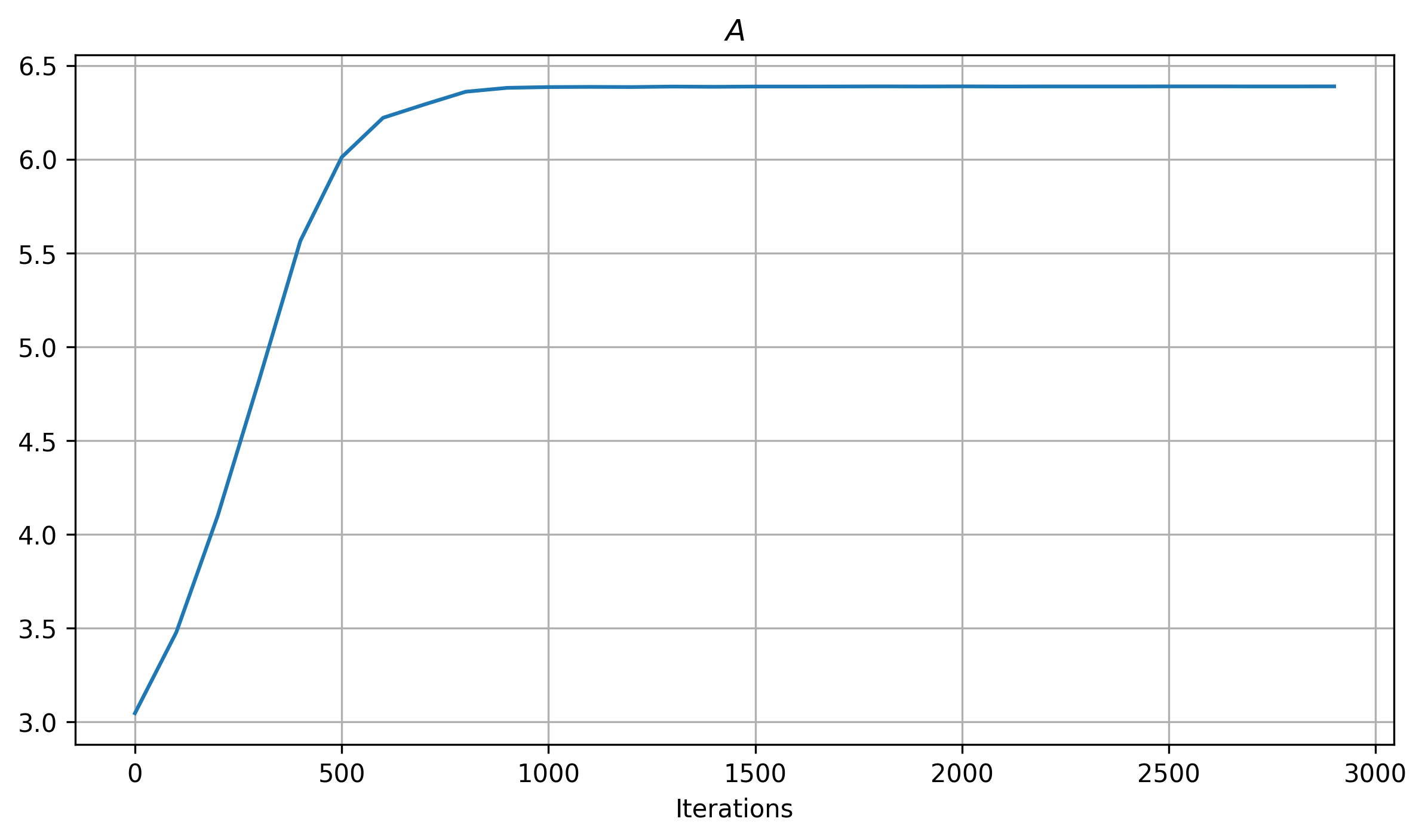}
	\caption{Evolution of the computed constant $A$.}
	\label{fig:5-2}
\end{subfigure}
\begin{subfigure}[t]{0.4\textwidth}
	\centering
	\includegraphics[width=\textwidth]{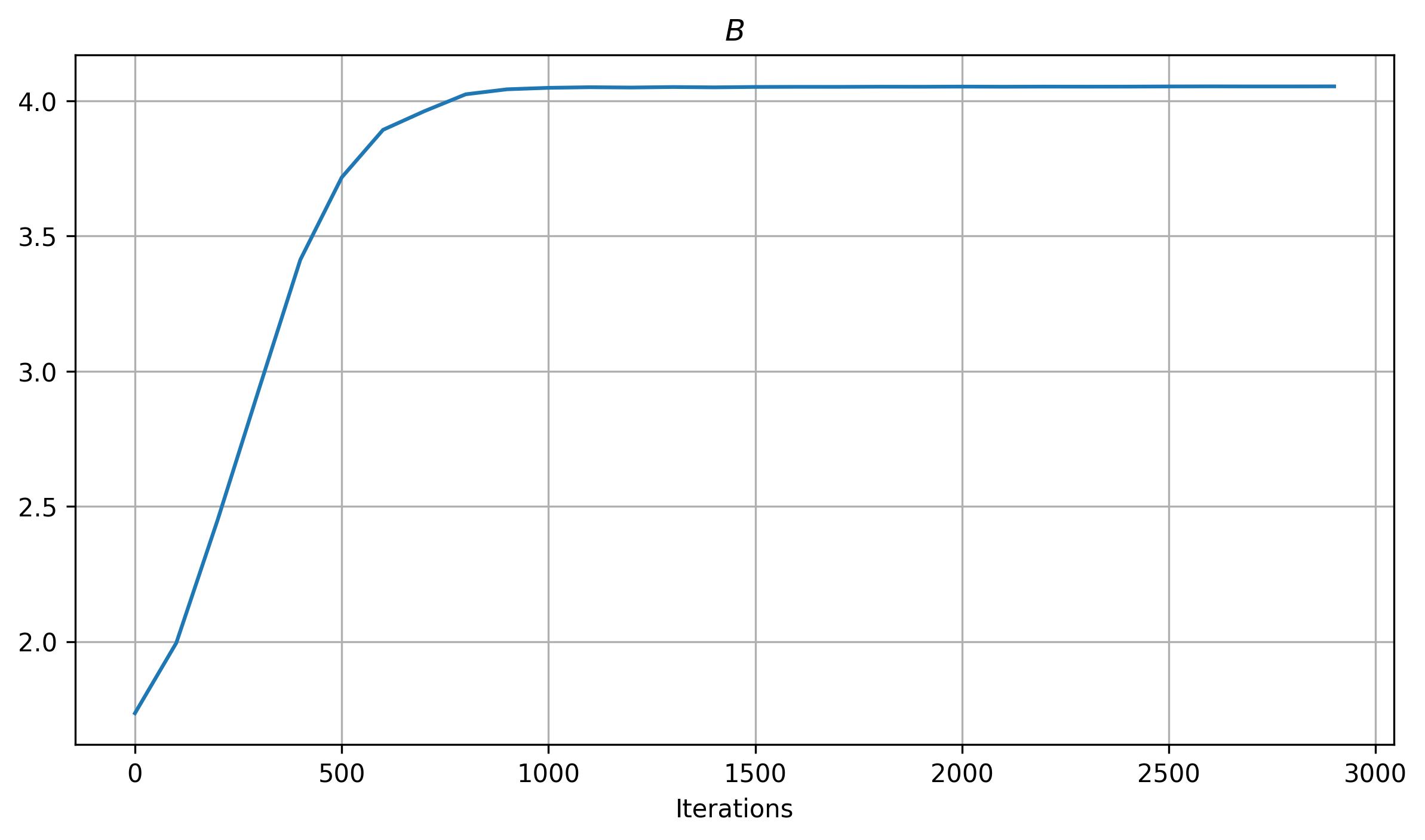}
	\caption{Evolution of the computed constant $B$.}
	\label{fig:5-3}
\end{subfigure}
\hspace{.2cm}
\begin{subfigure}[t]{0.4\textwidth}
	\centering
	\includegraphics[width=\textwidth]{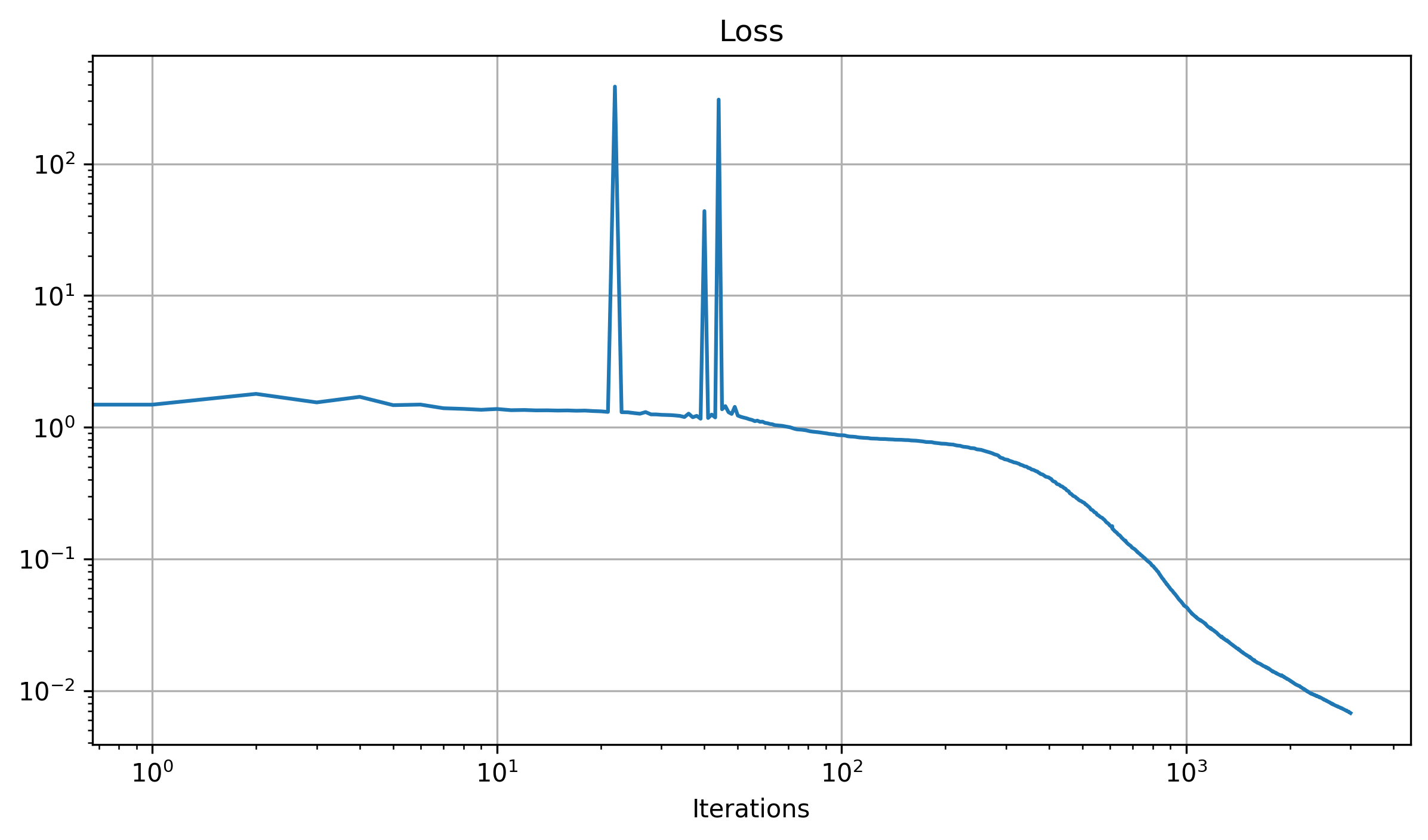}
	\caption{Error of the numerical approximation.}
	\label{fig:5-4}
\end{subfigure}
\caption{Computed evolution of the error and constants $\widetilde A$, $A$ and $B$ in the case of the KM breather.}
\label{fig:5}
\end{figure}

Table \ref{tab:iters-KM} summarizes the elapsed time and the evolution through the iterations of the  associated norms of the difference between the exact and the approximate solutions in the case of the KM breather with $a=\frac 34$. The time interval in these cases is $[-1,1]$, and the space interval is $[-5,5]$. These norms are computed using \eqref{Sprime-error},\eqref{LpW1q-error} and \eqref{LinfH1-error} with uniform weights equal 1, and uniform $N_\text{test}$ and $M_\text{test}$ given by 100. As a conclusion, from Table \ref{tab:iters-KM} we deduce that, during the numerical computations all the norms involved in Theorem \ref{MT} decreased to a value of order $10^{-2}$ and it takes around one minute and 30 seconds. 

\begin{table}[!ht]
\begin{tabular}{lllll}
\hline
Iterations & Elapsed time [s]  & error$_{L^{p}_tW^{1,q}_x}$ & error$_{L^{\infty}_tH^1}$ & error$_{\mathcal{C}'}$  \\ \hline
\hline
100        & {\color{black}3.237}               & {\color{black}3.371}                     & {\color{black}5.138}                     & {\color{black}5.138}                  \\
500        & {\color{black}15.371}           & {\color{black}1.405}                      & {\color{black}2.442}                     & {\color{black}2.442}                  \\
1000       & {\color{black}31.610}              & {\color{black}0.133}                      & {\color{black}0.256}                   & {\color{black}0.256}                  \\
3000       & {\color{black}92.605}              & {\color{black}$5.561 \times 10^{-2}$}     & {\color{black}$9.020 \times 10^{-2}$}                     & {\color{black}$9.020 \times 10^{-2}$}                   \\
\hline
\end{tabular}
\caption{{\color{black}Average} values of the elapsed time and norms involved in Theorem \ref{MT}, for five {\color{black}independent realizations of the algorithm} in the KM breather case.}
\label{tab:iters-KM}
\end{table}

\subsection{Peregrine breather} This is also another example of a solution with nonzero background at infinity for which the numerical approach works in a satisfactory fashion. We consider the case of the {\color{black}Peregrine} breather $B_\text{P}$ described in \eqref{BP}. {\color{black} Note that in this case there is no extra parameter that builds a family of solutions of the form of \eqref{BP}.}  Here, the space region is $[-10,10]$ and the time region {\color{black}$[-1.5,1.5]$}. However, for the Peregrine breather we choose {\color{black}$N_4 = 128$, $N_5 = 64$, $M_4=M_5=32$} {\color{black}and 4 hidden layers with 40 neurons each, as equal as in the KM breather case}. Our results are summarized in Fig.  \ref{fig:7}, where the continuous line represents the exact solution and the dashed line is the solution computed using the proposed PINNs minimization procedure. In particular, Fig. \ref{fig:7-1} and \ref{fig:7-2} present respectively the real and imaginary parts of the computed {\color{black}breather} solution for three different times.
\begin{figure}[!ht]
\centering
\begin{subfigure}[t]{0.45\textwidth}
	\centering
	\includegraphics[width=\textwidth]{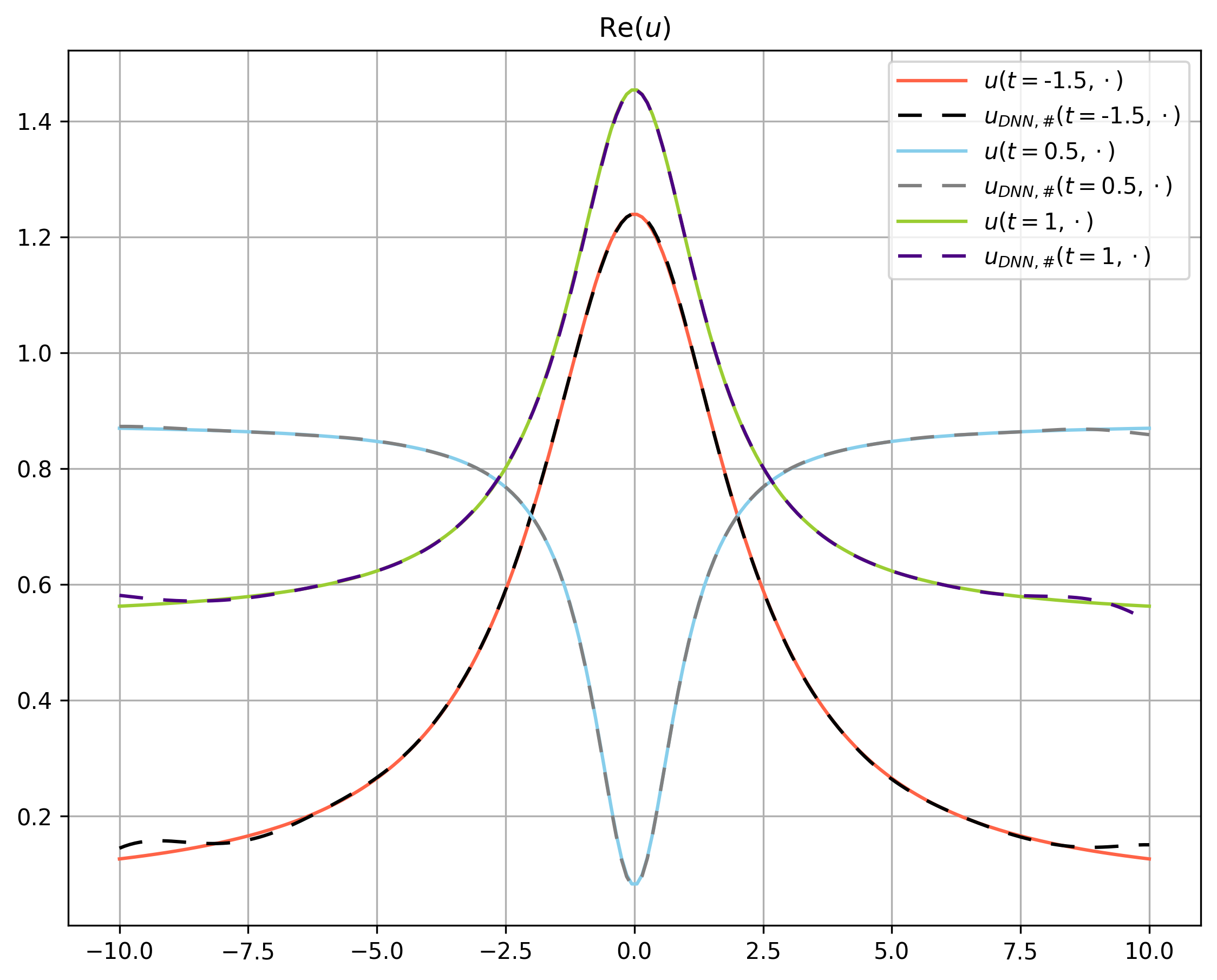}
	\caption{Real part of the exact (continuous line) and approximate (dashed line) Peregrine breather at times $t=$ -1.5, 0.5 and 1.}
	\label{fig:7-1}
\end{subfigure}
\hspace{.2cm}
\begin{subfigure}[t]{0.45\textwidth}
	\centering
	\includegraphics[width=\textwidth]{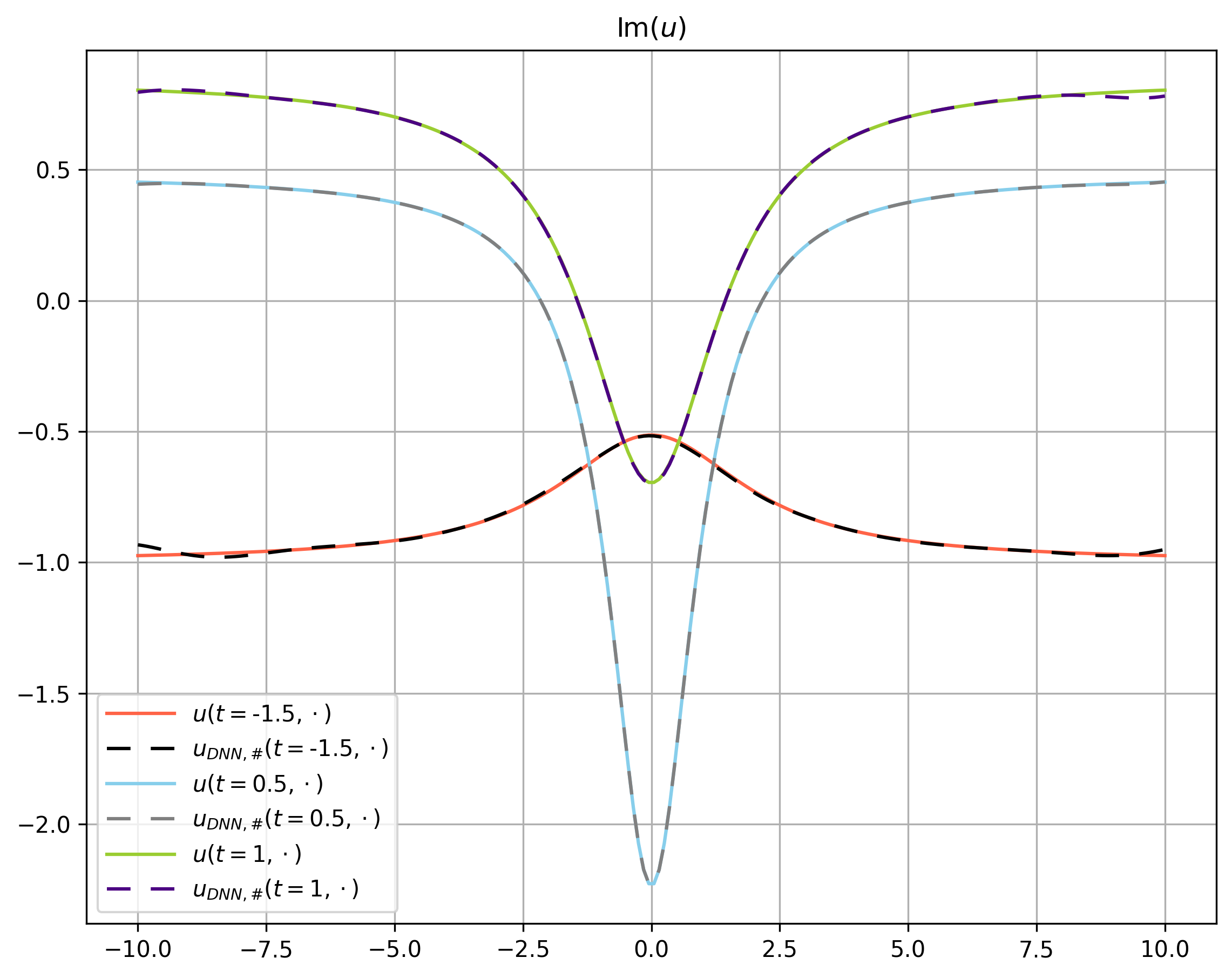}
	\caption{Imaginary part of the exact (continuous line) and approximate (dashed line) Peregrine breather at times $t=$ -1.5, 0.5 and 1.}
	\label{fig:7-2}
\end{subfigure}
\caption{Exact (continuous line) and approximate (dashed line) Peregrine breather.}
\label{fig:7}
\end{figure}

In Fig. \ref{fig:8}, the evolution of the computed constants $\widetilde A$, $A$, $B$, and the error of the approximation is presented in terms of the number of iterations of the numerical algorithm for the {\color{black}Peregrine} solution. Figs. \ref{fig:8-1}-\ref{fig:8-2}-\ref{fig:8-3} and \ref{fig:8-4} are in agreement with the description already found in the solitonic and 2-solitonic cases. 

\begin{figure}[!ht]
\centering
\begin{subfigure}[t]{0.4\textwidth}
	\centering
	\includegraphics[width=\textwidth]{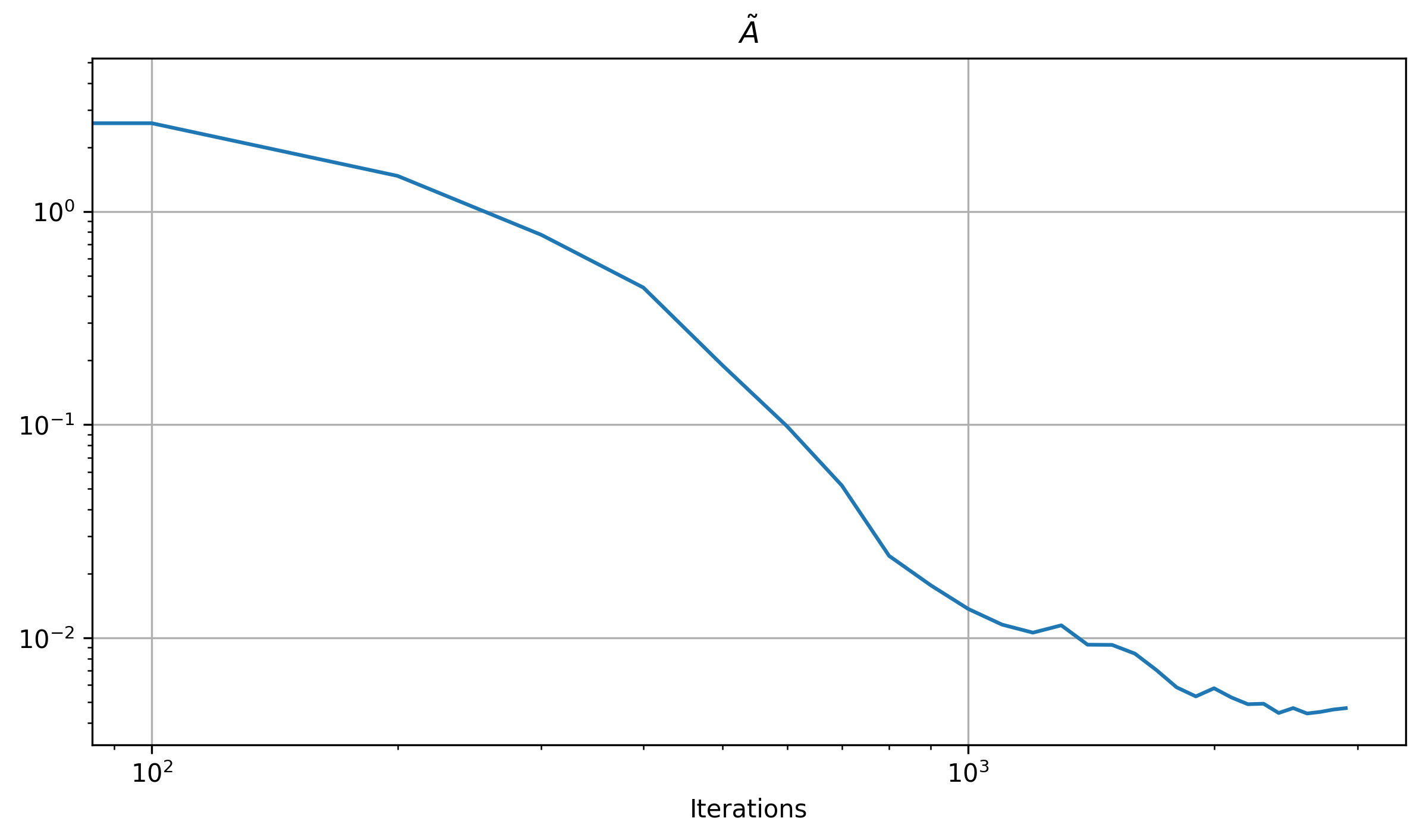}
	\caption{Evolution of the computed constant $\widetilde A$.}
	\label{fig:8-1}
\end{subfigure}
\hspace{.2cm}
\begin{subfigure}[t]{0.4\textwidth}
	\centering
	\includegraphics[width=\textwidth]{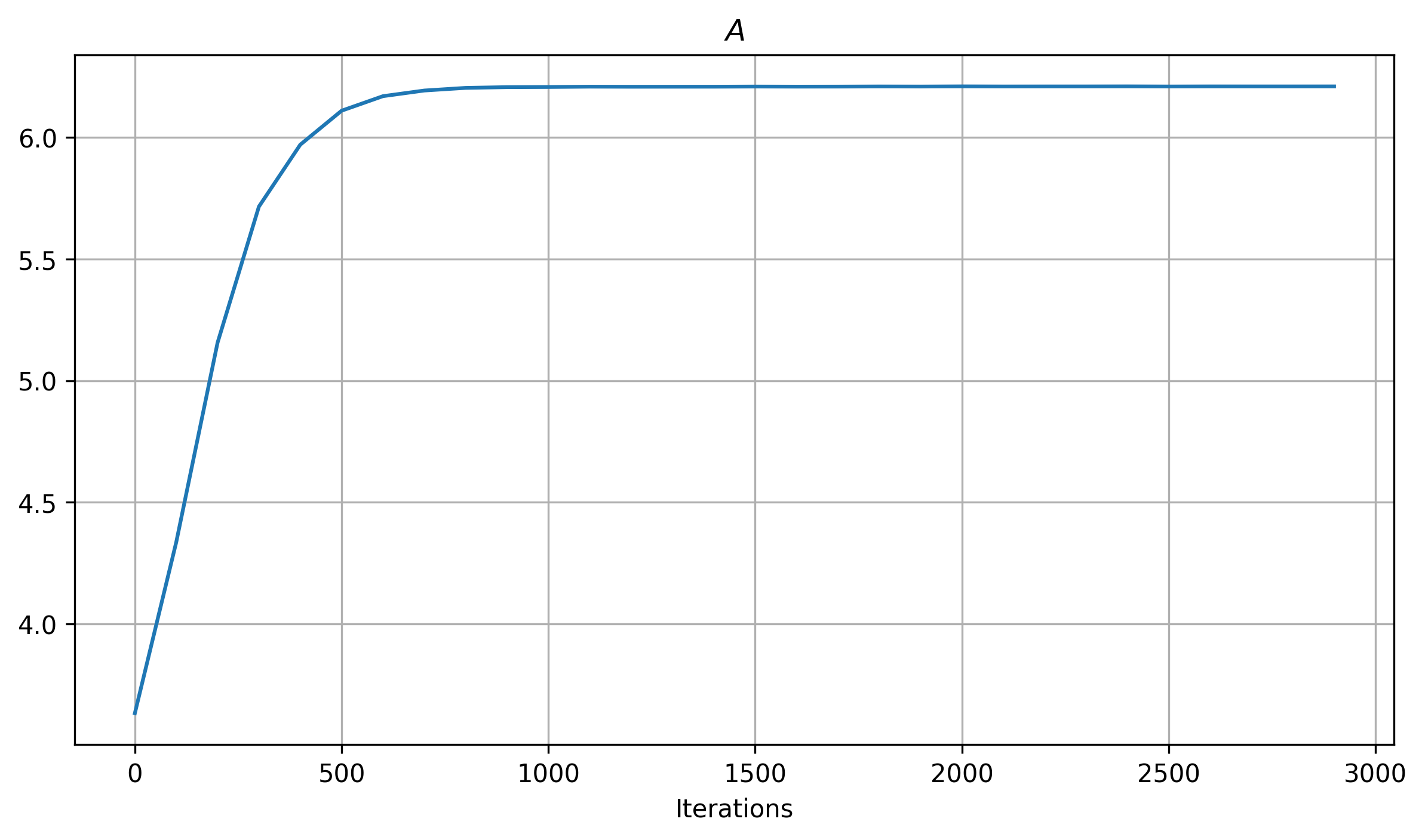}
	\caption{Evolution of the computed constant $A$.}
	\label{fig:8-2}
\end{subfigure}
\begin{subfigure}[t]{0.4\textwidth}
	\centering
	\includegraphics[width=\textwidth]{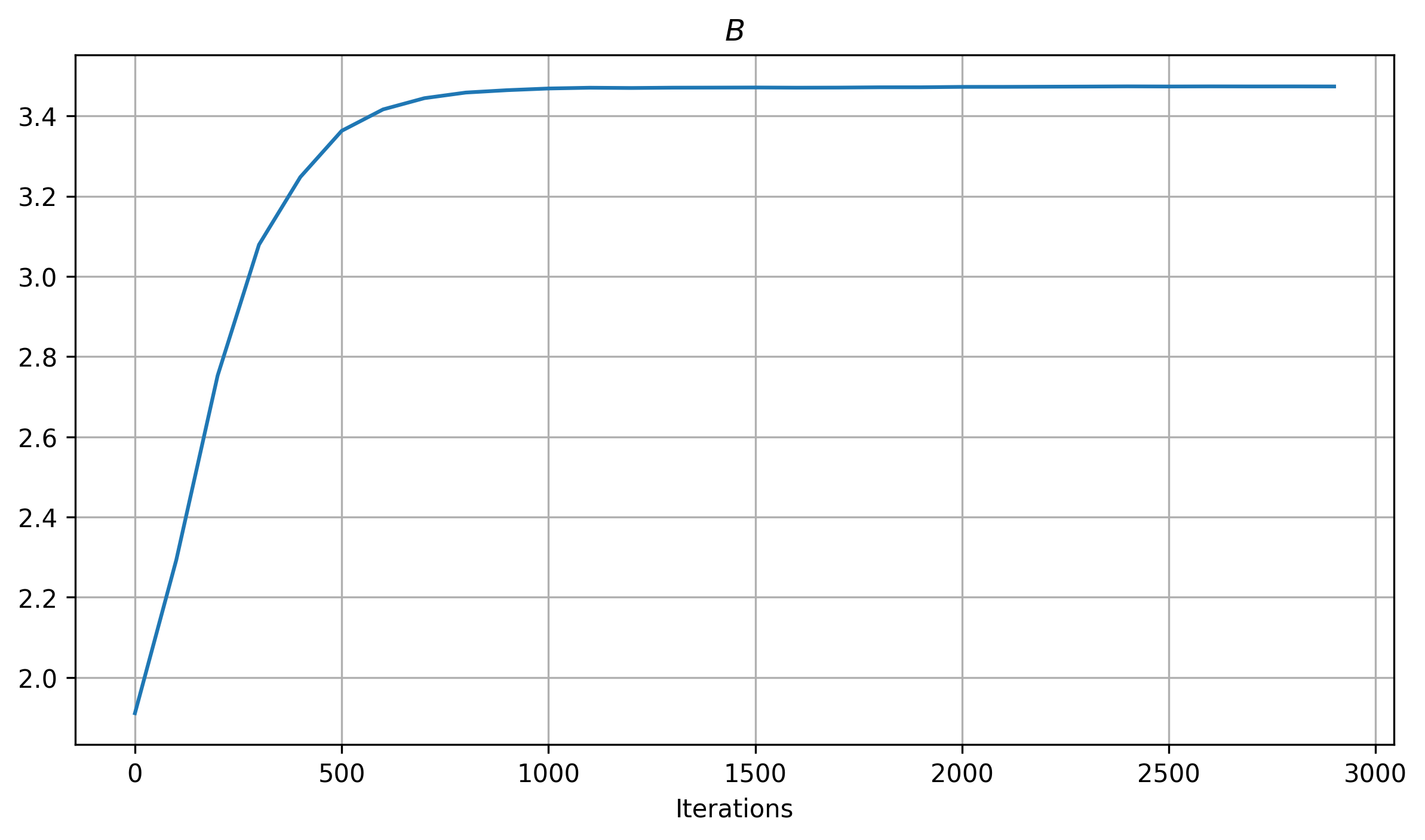}
	\caption{Evolution of the computed constant $B$.}
	\label{fig:8-3}
\end{subfigure}
\hspace{.2cm}
\begin{subfigure}[t]{0.4\textwidth}
	\centering
	\includegraphics[width=\textwidth]{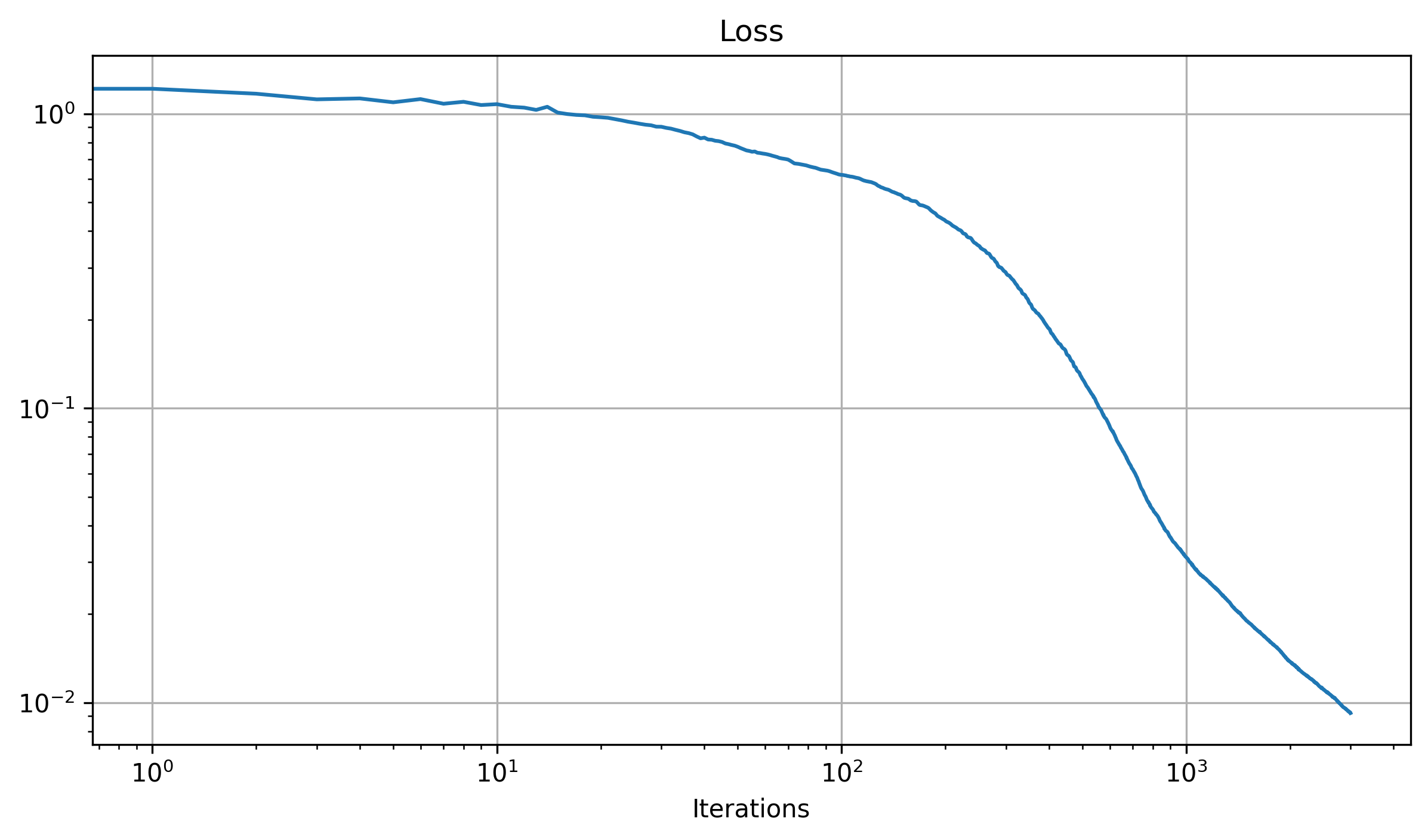}
	\caption{Error of the numerical approximation.}
	\label{fig:8-4}
\end{subfigure}
\caption{Computed evolution of the error and constants $\widetilde A$, $A$ and $B$ in the case of the Peregrine breather.}
\label{fig:8}
\end{figure}

Table \ref{tab:iters-Peregrine} summarizes the elapsed time and the evolution through the iterations of the  associated norms of the difference between the exact and the approximate solutions in the case of the Peregrine breather. The time interval in these cases is {\color{black}$[-1.5,1.5]$}, and the space interval is $[-10,10]$. These norms are computed using \eqref{Sprime-error},\eqref{LpW1q-error} and \eqref{LinfH1-error} with uniform weights equal 1, and uniform $N_\text{test}$ and $M_\text{test}$ given by 100. As a conclusion, from Table \ref{tab:iters-Peregrine} we deduce that, during the numerical computations the norms {\color{black}$L_t^\infty H^1$ and $\mathcal C'$} involved in Theorem \ref{MT} decreased to a value of order $10^{-1}$, {\color{black}while the norm $L^p_t W^{1,q}_x$ has order $10^{-2}$} and it takes less than 1 minutes {\color{black}30} seconds. Note in this case the convergence is slower than in the previous solitonic and KM breather cases. 

\begin{table}[!ht]
\begin{tabular}{lllll}
\hline
Iterations & Elapsed time [s]  & error$_{L^{p}_tW^{1,q}_x}$ & error$_{L^{\infty}_tH^1}$ & error$_{\mathcal{C}'}$  \\ \hline
\hline
100        & {\color{black}3.344}               & {\color{black}3.242}                     & {\color{black}5.428}                     & {\color{black}5.428}                  \\
500        & {\color{black}16.675}           & {\color{black}1.269}                      & {\color{black}2.369}                     & {\color{black}2.369}                  \\
1000       & {\color{black}31.885}              & {\color{black}0.192}                      & {\color{black}0.412}                  & {\color{black}0.412}                  \\
3000       & {\color{black}82.720}              & {\color{black}$7.068 \times 10^{-2}$}     & {\color{black}0.133}                    & {\color{black}0.133}                 \\
\hline
\end{tabular}
\caption{{\color{black}Average} values of the elapsed time and norms involved in Theorem \ref{MT}, for five {\color{black}independent realizations of the algorithm} in the Peregrine breather case.}
\label{tab:iters-Peregrine}
\end{table}

\subsection{Sensitivity analysis}
%{\color{red}
We will now focus in the solitonic and KM breather cases. The number of points in the grid was chosen as $M_4=M_5=32$ and $N_4=N_5=64$ for both cases. 

First of all we study the variation on the hyperparameters of the solitonic solution. In particular, Tables \ref{Tab:c1} and \ref{Tab:c3} present the values of computation time, loss function \eqref{eq:loss}, computed approximate values for $\widetilde A$, $A$, $B$, and the approximate norms $L_t^pW_x^{1,q}$ and $L_t^\infty H^1$ for the whole space-time computation. This is done for two and three different values of $c$ and $\nu$, respectively. For this end, we fix $[-10,10]$ the space region and $[-1,1]$ the time region, {and 5 hidden layers with 50 neurons each}. It is noticed that errors are usually between the orders $10^{-3}$ and $10^{-2}$. Then, the constant $\widetilde A$ is usually found small of order $10^{-3}$, as one might expect when approximating the solution for all times. Although the values of $A$ and $B$ are reasonably bounded below {\color{black}8, both quantities} increases when $\nu$ {\color{black}or $c$} increases. Finally, the three computed norms involved in Theorem \ref{MT} are suitably small, but one or two orders higher than the value of $\widetilde A$.  

\begin{table}[!ht]
\begin{tabular}{|l|l|l|l|}
\hline
$\nu$ & 1 & 3 & 5 \\
\hline
Elapsed Time $[s]$ & $82.427$ & $66.784$ & $77.742$ \\
\hline
Loss & $1.250 \times 10^{-3}$ & $1.569 \times 10^{-3}$ & $2.277 \times 10^{-3}$ \\
\hline
$\tilde A$ & $1.059 \times 10^{-3}$ & $9.660 \times 10^{-4}$ & $1.025 \times 10^{-3}$ \\
\hline
$A$ & $2.517$ & $3.786$ & $5.508$ \\
\hline
$B$ & $1.752$ & $2.700$ & $4.235$ \\
\hline
error$_{L_{t,x}^2}$ & $6.168 \times 10^{-3}$ & $4.101 \times 10^{-3}$ & $2.273 \times 10^{-3}$ \\
\hline
error$_{L_t^p W_x^{1,q}}$ & $1.262 \times 10^{-2}$ & $8.999 \times 10^{-3}$ & $5.972 \times 10^{-3}$ \\
\hline
error$_{L_t^\infty H^1}$ & $2.034 \times 10^{-2}$ & $1.440 \times 10^{-2}$ & $1.080 \times 10^{-2}$ \\
\hline
error$_{\mathcal{C}'}$ & $2.034 \times 10^{-2}$ & $1.440 \times 10^{-2}$ & $1.080 \times 10^{-2}$ \\
\hline
\end{tabular}
\caption{Computed values of the elapsed time, errors, constants and norms involved in Theorem \ref{MT}, for $c=1$ and three different values of $\nu$, in the solitonic case.}
\label{Tab:c1}
\end{table}

\begin{table}[!ht]
\begin{tabular}{|l|l|l|l|}
\hline
$\nu$ & 1 & 3 & 5 \\
\hline
Elapsed Time $[s]$ & $84.788$ & $87.744$ & $87.027$ \\
\hline
Loss & $5.553 \times 10^{-3}$ & $7.164 \times 10^{-3}$ & $1.278 \times 10^{-2}$ \\
\hline
$\tilde A$ & $1.560 \times 10^{-3}$ & $3.088 \times 10^{-3}$ & $4.104 \times 10^{-3}$ \\
\hline
$A$ & $3.948$ & $5.426$ & $7.563$ \\
\hline
$B$ & $3.022$ & $4.394$ & $6.594$ \\
\hline
error$_{L_{t,x}^2}$ & $7.861 \times 10^{-3}$ & $7.040 \times 10^{-3}$ & $9.493 \times 10^{-3}$ \\
\hline
error$_{L_t^p W_x^{1,q}}$ & $1.032 \times 10^{-2}$ & $1.285 \times 10^{-2}$ & $2.271 \times 10^{-2}$ \\
\hline
error$_{L_t^\infty H^1}$ & $1.955 \times 10^{-2}$ & $2.305 \times 10^{-2}$ & $4.211 \times 10^{-2}$ \\
\hline
error$_{\mathcal{C}'}$ & $1.955 \times 10^{-2}$ & $2.305 \times 10^{-2}$ & $4.211 \times 10^{-2}$ \\
\hline
\end{tabular}
\caption{Computed values of the elapsed time, errors, constants and norms involved in Theorem \ref{MT}, for $c=3$ and three different values of $\nu$, in the solitonic case.}
\label{Tab:c3}
\end{table}

Later, in Table \ref{Tab:2} we made the same analysis for the KM breather. This is done for four different values of $a$. For this end we fix $[-4,4]$ the space region, $[-0.5,0.5]$ the time region, and 4 hidden layers with 40 neurons each. It is noticed that errors are usually of the order $10^{-2}$ with the exception on the errors $L_{t}^\infty H^1$ and $\mathcal C'$ in the cases $a = 3/2,2$. As in the solitonic case, the constant $\widetilde A$ is usually found small {\color{black} of the order $10^{-3}$}. The norms $A$ and $B$ are reasonably bounded below {\color{black}10}, and both quantities increase when $a$ increases. Finally, the three computed norms involved in Theorem \ref{MT} are suitably small, but two order higher than the value of $\widetilde A$, and the $\mathcal{C}'$ norm is dictated by the $L^{\infty}_t H^1$ norm.

\begin{table}[!ht]
\begin{tabular}{|l|l|l|l|l|}
\hline
$a$ & 3/4 & 1 & 3/2 & 2 \\
\hline
Elapsed Time $[s]$ & $57.469$ & $52.725$ & $54.462$ & $63.542$ \\
\hline
Loss & $2.769 \times 10^{-3}$ & $3.906 \times 10^{-3}$ & $6.210 \times 10^{-3}$ & $1.090 \times 10^{-2}$ \\
\hline
$\tilde A$ & $4.022 \times 10^{-3}$ & $4.065 \times 10^{-3}$ & $7.018 \times 10^{-3}$ & $8.304 \times 10^{-3}$ \\
\hline
$A$ & $6.236$ & $7.055$ & $8.471$ & $9.724$ \\
\hline
$B$ & $4.042$ & $4.636$ & $5.709$ & $6.684$ \\
\hline
error$_{L_{t,x}^2}$ & $1.165 \times 10^{-2}$ & $1.239 \times 10^{-2}$ & $2.097 \times 10^{-2}$ & $2.834 \times 10^{-2}$ \\
\hline
error$_{L_t^p W_x^{1,q}}$ & $3.638 \times 10^{-2}$ & $3.895 \times 10^{-2}$ & $6.452 \times 10^{-2}$ & $7.251 \times 10^{-2}$ \\
\hline
error$_{L_t^\infty H^1}$ & $5.861 \times 10^{-2}$ & $6.658 \times 10^{-2}$ & $1.046 \times 10^{-1}$ & $1.332 \times 10^{-1}$ \\
\hline
error$_{\mathcal{C}'}$ & $5.861 \times 10^{-2}$ & $6.658 \times 10^{-2}$ & $1.046 \times 10^{-1}$ & $1.332 \times 10^{-1}$ \\
\hline
\end{tabular}
\caption{Computed values of the elapsed time, errors, constants and norms involved in Theorem \ref{MT}, for four different values of $a$, in the KM breather case.}
\label{Tab:2}
\end{table}

{\color{black} Additionally, we investigate the effect on the weight factor for the derivatives presented in Remark \ref{rem:loss}. To this end, we consider the solitonic case with parameters $c=3$ and $\nu=5$, and test five different values for the weight: $0, 10^{-3}, 10^{-2}, 0.1$ and 1. Table \ref{Tab:weight_der} summarizes the computed norms corresponding to each weight value. As shown in Table \ref{Tab:weight_der}, the choice of the weighting factor has a significant impact on the performance of the algorithm: large weights introduce an additional numerical errors, which in turn lead to higher approximative errors.  On the other hand, completely neglecting the derivative terms yields acceptable results, but suboptimal compared to using a small, strictly positive weight. 

Finally, to have almost the same orders in Table \ref{Tab:weight_der}, the value of $\gamma$ needed to be changed. In particular, if $\gamma>1$ in the cases $0.1$ and $1$, the algorithm stoped in the first 100 iterations, converging to the trivial solution $u \equiv 0$.

\begin{table}[!ht]
\begin{tabular}{lllll}
\hline
Weight & $\gamma$ & error$_{L^{p}_tW^{1,q}_x}$ & error$_{L^{\infty}_tH^1}$ & error$_{\mathcal C'}$  \\ \hline
\hline
0  & 3 & $2.988 \times 10^{-2}$ & $5.316 \times 10^{-2}$ & $5.316 \times 10^{-2}$  \\
$0.001$ & 3 & $2.271 \times 10^{-2}$ & $4.211 \times 10^{-2}$ & $4.211 \times 10^{-2}$ \\
$0.01$ & 2 & $6.945 \times 10^{-2}$ &$1.160 \times 10^{-1}$ & $1.160 \times 10^{-1}$ \\
$0.1$ & 0.5 & $3.641 \times 10^{-2}$ &  $6.621 \times 10^{-2}$ &  $6.621 \times 10^{-2}$\\
1  & 0.1 &  $2.971 \times 10^{-2}$ & $6.045 \times 10^{-2}$ & $6.045 \times 10^{-2}$\\
\hline
\end{tabular}
\caption{Computed values of the norms involved in Theorem \ref{MT} for different values of the weight factor in the derivative as shown in Remark \ref{rem:loss}, in the solitonic case with $c=3$ and $\nu=5$.}
\label{Tab:weight_der}
\end{table}
}

\medskip
{\bf Effect of time:} In order to study the performance of the algorithm, we will change the size of the time interval. For this aim, we consider 5000 iterations, 40 neurons per hidden layer {\color{black}with $N_4=128$ and $N_5=M_4=M_5=32$}. Fig. \ref{fig:soliton-T4} summarizes the performance of the solitonic case with $c = \nu = 1$ in the space region $[-10,10]$ and the time region $[-4,4]$, where the continuous line represents the exact solution and the dashed line is the solution computed using the proposed PINNs minimization procedure. In particular, Fig. \ref{fig:soliton-T4-1} and \ref{fig:soliton-T4-2} present respectively the real and imaginary parts of the computed soliton solution for three different times. Our simulations suggest that for greater values of the time interval, the train grid needs to be fine-tuned.

\begin{figure}[!ht]
\centering
\begin{subfigure}[t]{0.45\textwidth}
	\centering
	\includegraphics[width=\textwidth]{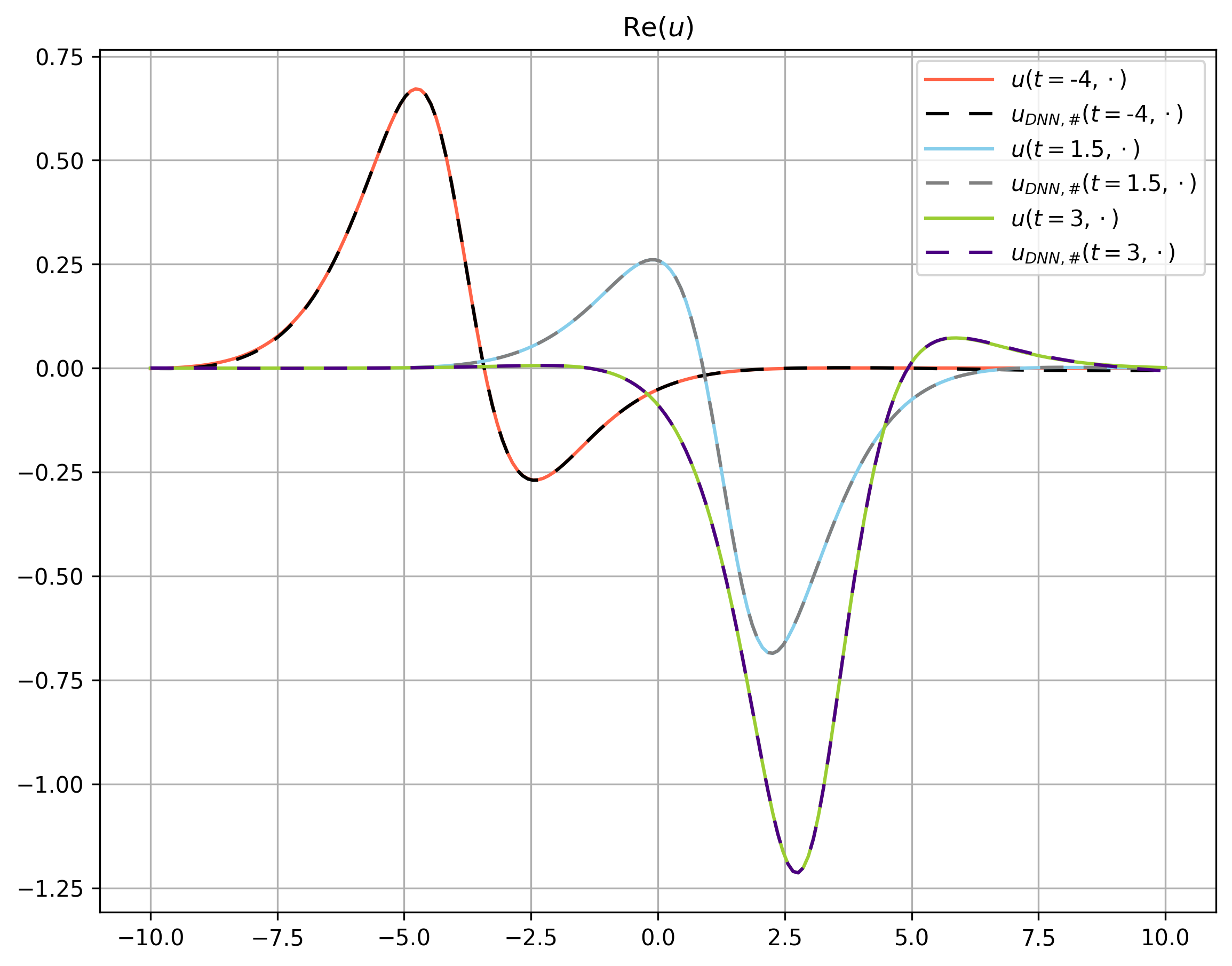}
	\caption{Real part of the exact (continuous line) and approximate (dashed line) soliton solution at times $t=$ -4, 1.5 and 3.}
	\label{fig:soliton-T4-1}
\end{subfigure}
\hspace{.2cm}
\begin{subfigure}[t]{0.45\textwidth}
	\centering
	\includegraphics[width=\textwidth]{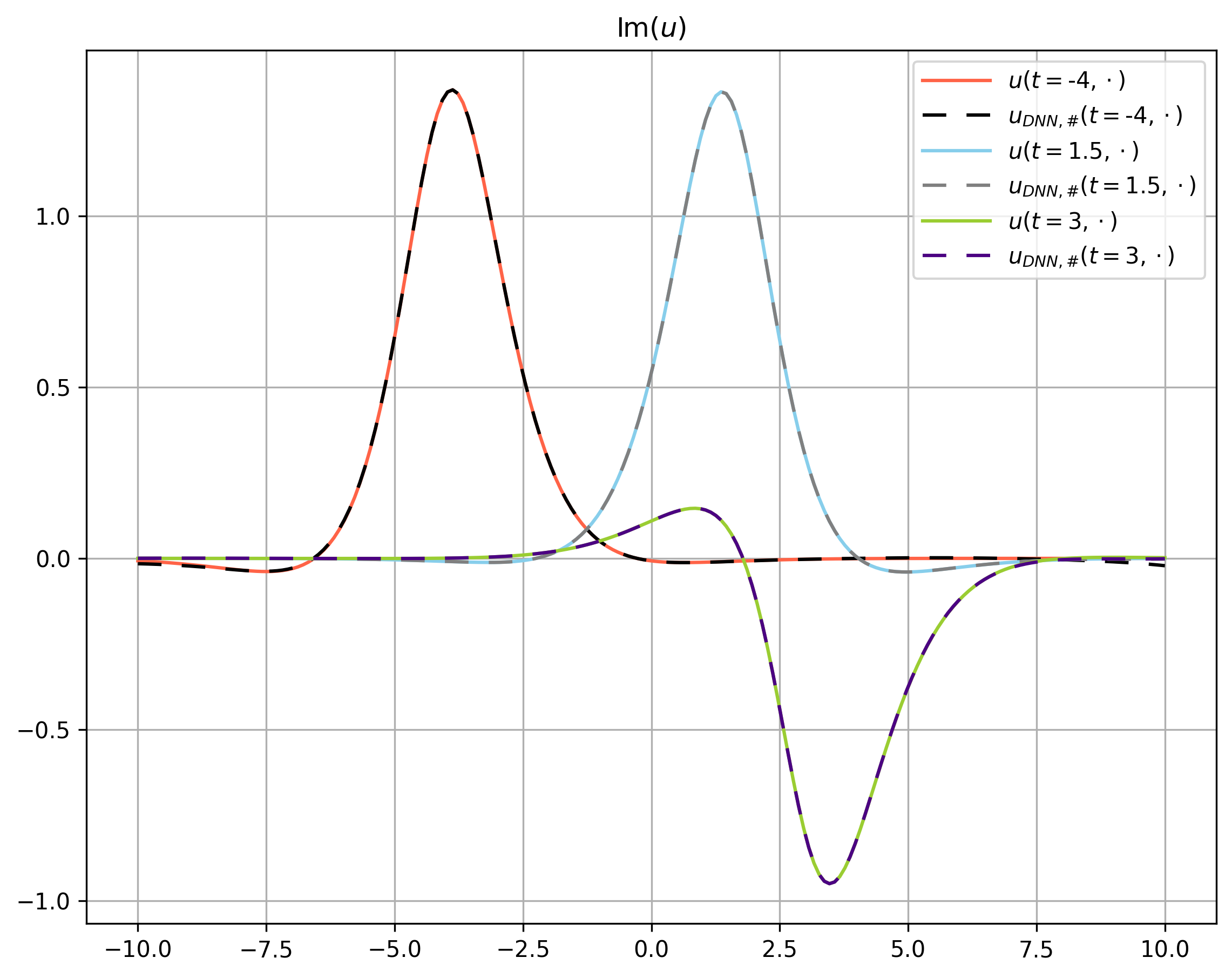}
	\caption{Imaginary part of the exact (continuous line) and approximate (dashed line) soliton solution at times $t=$ -4, 1.5 and 3.}
	\label{fig:soliton-T4-2}
\end{subfigure}
\caption{Approximation of the soliton solution in the case $c=\nu=1$.}
\label{fig:soliton-T4}
\end{figure}

In this case the algorithm takes {\color{black}145.072} seconds, and obtaining errors {\color{black}error$_{L^{p}_tW^{1,q}_x} = 1.720 \times 10^{-2}$, error$_{L^{\infty}_t H^1} = 2.780 \times 10^{-2}$, error$_{\mathcal{C}'} = 2.780 \times 10^{-2}$, loss$= 6.345 \times 10^{-4}$ with constants $\widetilde A = 3.455 \times 10^{-4}$, $A = 2.517$ and $B = 2.083$.}
%}

{\color{black}
\subsection{Comparison with other approximation scheme} In this section we will compare the PINN algorithm with the finite differences (FD) method implemented in \cite{ABB13} for the Schr\"odinger equation \eqref{eq:NLS-anyalpha}. In the following comparisons we will treat only with the solitonic case. As similar as in the sensitivity analysis section we will use $c \in \{1,3\}$ and $\nu \in \{1,3,5\}$ and the space-time region will be $[-10,10]\times[-1,1]$ as before. Notice that the breather cases cannot be compared using the FD method. For these cases the time-splitting sine pseudospectral (TSSP) method, also described in \cite{ABB13} but which will not be discussed in this paper, can be used.

Table \ref{Tab:FDM} summarizes the elapsed time and the $L^{p}_tW^{1,q}_x$, $L^{\infty}_t H^1_x$ and $L^2_{t,x}$ errors in the six different pairs $(c,\nu)$ {\color{black}for the FD method}. The FD method is performed with 400 points in the space grid and $1.8 \times 10^{6}$ points in the time grid. 
\begin{table}[!ht]
\begin{tabular}{llllll}
\hline
c & $\nu$ & Elapsed time [s]  & error$_{L^{p}_tW^{1,q}_x}$ & error$_{L^{\infty}_tH^1}$  & error$_{L^2_{t,x}}$ \\ \hline
\hline
1 & 1 & 24.101 & 1.999 $\times 10^{-3}$ & 3.452 $\times 10^{-3}$ & 1.423 $\times 10^{-3}$ \\
1 & 3 & 24.473 & 1.382 $\times 10^{-2}$ & 2.352 $\times 10^{-2}$ & 7.115 $\times 10^{-3}$ \\
1 & 5 & 25.238 & 6.830 $\times 10^{-2}$ & 1.184 $\times 10^{-1}$ & 2.703 $\times 10^{-2}$ \\
3 & 1 & 23.935 & 1.873 $\times 10^{-2}$ & 2.852 $\times 10^{-2}$ & 1.259 $\times 10^{-2}$ \\
3 & 3 & 24.690 & 7.939 $\times 10^{-2}$ & 1.152 $\times 10^{-1}$ & 3.410 $\times 10^{-2}$ \\
3 & 5 & 27.291 & 2.815 $\times 10^{-1}$ & 4.057 $\times 10^{-1}$ & 8.890 $\times 10^{-2}$ \\
\hline
\end{tabular}
\caption{Computed values of the elapsed time and norms involved in Theorem \ref{MT} using the FD method, for two and three different values of $c$ and $\nu$, respectively, in the solitonic case.}
\label{Tab:FDM}
\end{table}

By comparing the Tables \ref{Tab:c1} and \ref{Tab:c3} with the Table \ref{Tab:FDM}, the following conclusions can be made. First, although the FD method takes one third of time than our PINNs algorithm, they have the same order in the computation times. Additionally, {\color{black}only in the setting $c=1, \nu = 1$ the FD method performs quite better than our algorithm. In the other settings of $c$ and $\nu$, the PINNs algorithm outperforms the FD method}.

Finally, the worst case of the FD method is $c=3$, $\nu = 5$: both errors $L^p_t W^{1,q}_x$ and $L^{\infty}_t H^1_x$ have order {\color{black}$O(10^{-1})$. In contrast, our PINNs method both errors are one less order.}
%\begin{table}
%\begin{tabular}{|l|l|l|l|l|}
%\hline
%$\nu$                     & 1 & 3 & 5  \\ \hline
%Elapsed Time {[}s{]} & 28.477            & 28.364        & 28.041                                 \\ \hline
%error$_{L^p_tW^{1,q}_x}$  & 1.265 $\times 10^{-3}$  & 6.707 $\times 10^{-3}$   & 3.521 $\times 10^{-2}$                                            \\ \hline
%error$_{L^{\infty}_t H^1}$  & 4.579 $\times 10^{-3}$  & 9.682 $\times 10^{-3}$    & 5.371 $\times 10^{-2}$                                \\ \hline
%\end{tabular}
%\caption{Computed values of the elapsed time and norms involved in Theorem \ref{MT}, for $c=1$ and three different values of $\nu$, in the solitonic case.}
%\label{Tab:c1-FDM}
%\end{table}

%\begin{table}
%\begin{tabular}{|l|l|l|l|l|}
%\hline
%$\nu$                     & 1 & 3 & 5  \\ \hline
%Elapsed Time {[}s{]} & 28.582            & 27.960        & 27.894                                 \\ \hline
%error$_{L^p_tW^{1,q}_x}$  & 1.341 $\times 10^{-2}$  & 3.523 $\times 10^{-2}$   & 9.278 $\times 10^{-2}$                                            \\ \hline
%error$_{L^{\infty}_t H^1}$  & 1.443 $\times 10^{-2}$  & 3.754 $\times 10^{-2}$    & 9.916 $\times 10^{-2}$                                \\ \hline
%\end{tabular}
%\caption{Computed values of the elapsed time and norms involved in Theorem \ref{MT}, for $c=3$ and three different values of $\nu$, in the solitonic case.}
%\label{Tab:c3-FDM}
%\end{table}
}

\section{Discussions and conclusions}\label{sec:discussion}

\subsection{Discussion} 

The Physics-Informed Neural Networks (PINN) methodology has recently been applied successfully to solve various nonlinear Schr\"odinger (NLS) equations, both analytically and numerically, in bounded domains. By integrating physical laws with neural networks, PINNs have achieved accurate solutions with minimal data requirements \cite{SSHC,BKM22,HD,PLC,WY,ZB,BMAR,RPK,Raissi,TSVS}. However, these studies are limited to spatially bounded domains and do not adequately address real-world physical phenomena, which often involve slow-decaying solutions or long-range potentials in unbounded domains.

For unbounded domains, recent advancements have emerged. In \cite{XCL}, a deep neural network (DNN) was developed to implement an absorbing boundary condition (ABC), which confines computations to a finite domain. The ABC minimizes unwanted reflections at the boundaries, effectively replicating the influence of the surrounding environment without fully discarding the unbounded regions. This method has been applied to wave and Schrödinger equations in unbounded domains. Despite their success, PINNs can be challenging to train and prone to practical issues, as demonstrated in \cite{FWP}, particularly for wave scattering problems in unbounded domains \cite{Colt}. Such problems typically require sampling collocation points across an infinite domain, rendering many traditional PINN approaches unsuitable. In response, \cite{FWP} employed boundary integral equations (BIEs), a robust method for solving partial differential equations (PDEs) in complex or unbounded domains by representing solutions in terms of boundary values.

In \cite{XBC}, the NLS equation on unbounded domains was addressed using a hybrid approach combining adaptive spectral and PINN methods. This approach used spectral basis functions and a priori assumptions about asymptotic spatial behavior to effectively describe the spatial dependence. Recent studies have introduced other techniques for dealing with unbounded domains: \cite{BE} introduced a similarity variable to handle infinity, while \cite{RRCWSL} truncated the computational domain to ensure that the wave exits the domain without disturbing the upstream field.

In contrast, our approach to handling unbounded domains does not rely on external basis functions. Instead, we demonstrate how to describe the critical focusing NLS evolution on the unbounded real line by imposing a condition on the evolution of the linear part and employing appropriate approximate norms. These norms, under certain smallness assumptions, allow for an accurate approximation of NLS dynamics in this setting.

{\color{black}

Regarding the numerical simulations, we have confirmed that Theorem \ref{MT} is valid in a good sense, specially after translating with care and detail into numerical methods, the norms involved in the statement if this result. In practice, Theorem \ref{MT} has needed some adjustment to fit better into the numerical analysis techniques. In particular, quadratic norms are usually used in numerics, while the theorem only considers linear norms. It is important to note that although the orders of the approximate norm are not as small as typically expected in numerical analysis, both the {\bf pointwise error and the mean squared error remain remarkably low}. This suggests that, within the standard functional spaces used in numerical analysis, the computed approximations are indeed accurate and reliable.

Then, it is worth noting that the loss function is not an exact approximation of the involved norms, principally because uniform weights equal to 1 were used instead of weights derived from a quadrature rule. Instead, the loss was normalized in the sense that the size of the time-space domain was not taken into account during the minimization process. This modification was intended to enhance the numerical optimization by avoiding the inclusion of a constant factor that is irrelevant to the optimizer. 

The choice of uniform weights (equal to 1) also contributed to greater stability of the algorithm. For comparison purposes, we conducted simulations using both type of weights, and those with uniform weights consistently outperformed the others. 

The proposed PINN method is capable of producing accurate approximations even when derivatives are omitted from the loss function, that is, by considering $L_x^q$ norms instead of $W_{x}^{1,q}$ norms. Nonetheless, in this work, we retained all derivatives in the simulations, to maintain consistency with the theoretical framework, along with the weight multiplication on the loss function to improve efficiency to the algorithm.

Finally, but not less important, there is the problem of finding a good approximation in the long time case, namely when $T$ is made large enough. We argue that this corresponds to a problem of different nature, by several reasons. The first is that if soliton's resolution conjecture is valid, after some large amount of time the dynamics will be strongly determined by soliton type solutions appearing from the dynamics, assuming of course that the dynamics is globally defined in time. Therefore, the understanding of the long time dynamics is closely related to the stability of solitons. Estimates as the one obtained in Theorem \ref{MT} are naturally condemned to fail, since modulations of solitons are always needed. We conclude the existence of a new PINNs problem, which is the description of the orbital stability property using deep neural networks. For instance: is it possible to design, via deep neural networks, modulations of velocity and scaling on a soliton, in such a way that the approximation is better suited for an important, large amount of time? This is a problem that will be described in detail in a forthcoming publication.

Another reason that makes the problem ``$T$ large'' important is the possible existence of blow up solutions in the case $\alpha>5$. This is another problem were PINNS may be able to show a strong approximation capacity, since after a suitable rescaling of the time variable, it is expected that blow up corresponds to a sort of asymptotic stability property of a well-chosen final state. This is another reason to consider the long time dynamics with even more care than the standard case $T$ small.
}
% \item Summary: A brief recap of our key results
% \item Interpretations: What do our results mean?
% \item Implications: Why do our results matter?
% \item Limitations: What can?t our results tell us?
% \item Recommendations: Avenues for further studies or analyses
%\end{itemize}
\subsection{Conclusions} We have studied the subcritical nonlinear Schr\"odinger (NLS) equation in one dimension on the unbounded real line. The non-compact nature of this setting introduces significant challenges for approximation schemes, particularly in the context of deep neural networks (DNNs), as there is no natural restriction on the data at spatial infinity. Previous works, such as \cite{RPK,BKM22}, have addressed NLS on bounded domains from both numerical and theoretical perspectives, advancing our understanding of dispersive solution approximations via DNN techniques.

In this paper, we propose a new approach using Physics-Informed Neural Networks (PINNs) that combines standard physical constraints with a novel component designed to handle the linear evolution of perturbed data in unbounded domains. We present what we believe are the first rigorous bounds on the associated approximation error in energy and Strichartz norms, which are fundamental in the NLS framework. Our analysis assumes the availability of suitable integration schemes capable of accurately approximating space-time norms.

Additionally, we applied this method to traveling waves, breathers, and solitons, and conducted numerical experiments that validate the effectiveness of the approximation. The results confirm that the proposed scheme performs well in capturing the dynamics of NLS solutions in unbounded domains.

\end{document}